\documentclass[10pt,letterpaper]{article}

\usepackage[letterpaper, top=1in, bottom=1in, left=1in, right=1.2in, includefoot]{geometry}

\usepackage{latexsym,
amsthm,
color,
amssymb,
url,
bm,
cite,
amssymb,
microtype,
amsmath,
amsfonts,
mathtools}

\usepackage[mathscr]{euscript}

\usepackage{caption}
\usepackage{setspace}
\usepackage{subcaption}
\usepackage{setspace}

 \usepackage[pdftex, plainpages = false, pdfpagelabels, 
                 bookmarks = true,
                 bookmarksopen = true,
                 bookmarksnumbered = true,
                 breaklinks = true,
                 linktocpage,
                 pagebackref,
                 colorlinks = true,  
                 linkcolor = blue,
                 urlcolor  = blue,
                 citecolor = blue,
                 anchorcolor = green,
                 hyperindex = true,
                 hyperfigures
                 ]{hyperref}

\usepackage[inline]{enumitem}
\usepackage[mathscr]{euscript}
\usepackage{amsthm}
\usepackage{nicefrac}
\usepackage{dsfont}
\usepackage{tikz}
\usepackage[ruled,linesnumbered]{algorithm2e}

\usepackage{nicefrac}

\RequirePackage[T1]{fontenc}
\RequirePackage[utf8]{inputenc}
\usepackage[english]{babel}
\usepackage{alphabeta}

\usetikzlibrary{math}

\makeatletter
\input{colordvi}

\usepackage{aliascnt}

\hypersetup{
    colorlinks,
    linkcolor={red!75!black},
    citecolor={green!75!black},
    urlcolor={blue!75!black},
}

\newcommand{\cref}[1]{\autoref{#1}}

\newcommand{\remove}[1]{}

\newcounter{func}
\newcommand{\funref}[1]{\hyperref[#1]{f_{\ref*{#1}}}} 

\tikzset{black node/.style={draw, circle, fill = black, minimum size = 5pt, inner sep = 0pt}}
\tikzset{white node/.style={draw, circlternary_treese, fill = white, minimum size = 5pt, inner sep = 0pt}}
\tikzset{normal/.style = {draw=none, fill = none}}
\tikzset{lean/.style = {draw=none, rectangle, fill = none, minimum size = 0pt, inner sep = 0pt}}
\usetikzlibrary{decorations.pathreplacing}
\usetikzlibrary{arrows.meta}

\usetikzlibrary{shapes}
\tikzset{diam/.style={draw, diamond, fill = black, minimum size = 7pt, inner sep = 0pt}}

\newcommand{\Ccal}{\mathcal{C}}
\newcommand{\Dcal}{\mathcal{D}}

\newcommand{\Gcal}{\mathcal{G}}
\newcommand{\Hcal}{\mathcal{H}}

\newcommand{\Mcal}{\mathcal{M}}

\newcommand{\Ocal}{\mathcal{O}}
\newcommand{\Pcal}{\mathcal{P}}
\newcommand{\Qcal}{\mathcal{Q}}
\newcommand\Rcal{\mathcal{R}}
\newcommand{\Scal}{\mathcal{S}}
\newcommand{\Tcal}{\mathcal{T}}

\newcommand{\Xcal}{\mathcal{X}}

\newcommand{\Nbbb}{\mathbb{N}}

\newcommand{\Rbbb}{\mathbb{R}}
\newcommand{\Sbbb}{\mathbb{S}}

\RequirePackage{stmaryrd}
\usepackage{textcomp}
\DeclareUnicodeCharacter{2286}{\subseteq}
\DeclareUnicodeCharacter{2192}{\ifmmode\to\else\textrightarrow\fi}
\DeclareUnicodeCharacter{2203}{\ensuremath\exists}
\DeclareUnicodeCharacter{183}{\cdot}
\DeclareUnicodeCharacter{2200}{\forall}
\DeclareUnicodeCharacter{2264}{\leq}
\DeclareUnicodeCharacter{2265}{\geq}
\DeclareUnicodeCharacter{8614}{\mathbin{\mapsto}}
\DeclareUnicodeCharacter{8656}{\Leftarrow}
\DeclareUnicodeCharacter{8657}{\Uparrow}
\DeclareUnicodeCharacter{8658}{\Rightarrow}
\DeclareUnicodeCharacter{8659}{\Downarrow}
\DeclareUnicodeCharacter{8669}{\rightsquigarrow}
\newcommand{\eqdef}{\stackrel{{\scriptsize\rm def}}{=}}
\DeclareUnicodeCharacter{8797}{\eqdef}
\DeclareUnicodeCharacter{8870}{\vdash}
\DeclareUnicodeCharacter{8873}{\Vdash}
\DeclareUnicodeCharacter{22A7}{\models}
\DeclareUnicodeCharacter{9121}{\lceil}
\DeclareUnicodeCharacter{9123}{\lfloor}
\DeclareUnicodeCharacter{9124}{\rceil}
\DeclareUnicodeCharacter{2208}{\in}
\DeclareUnicodeCharacter{9126}{\rfloor}
\DeclareUnicodeCharacter{9655}{\triangleright}
\DeclareUnicodeCharacter{9665}{\triangleleft}
\DeclareUnicodeCharacter{9671}{\diamond}
\DeclareUnicodeCharacter{9675}{\circ}
\DeclareUnicodeCharacter{10178}{\bot}
\DeclareUnicodeCharacter{10214}{} 
\DeclareUnicodeCharacter{10215}{} 
\DeclareUnicodeCharacter{10229}{\longleftarrow}
\DeclareUnicodeCharacter{10230}{\longrightarrow}
\DeclareUnicodeCharacter{10231}{\longleftrightarrow}
\DeclareUnicodeCharacter{10232}{\Longleftarrow}
\DeclareUnicodeCharacter{10233}{\Longrightarrow}
\DeclareUnicodeCharacter{10234}{\Longleftrightarrow}
\DeclareUnicodeCharacter{10236}{\longmapsto}
\DeclareUnicodeCharacter{10238}{\Longmapsto} 
\DeclareUnicodeCharacter{10503}{\Mapsto}    
\DeclareUnicodeCharacter{10971}{\mathrel{\not\hspace{-0.2em}\cap}}
\DeclareUnicodeCharacter{65294}{\ldotp}
\DeclareUnicodeCharacter{65372}{\mid}

\definecolor{MidnightBlack}{rgb}{0.1,0.1,.34}
\definecolor{MidnightBlue}{rgb}{0.1,0.1,0.43}
\definecolor{Black}{rgb}{0,0, 0}
\definecolor{Blue}{rgb}{0, 0 ,1}
\definecolor{Red}{rgb}{1, 0 ,0}
\definecolor{White}{rgb}{1, 1, 1}
\definecolor{grey}{rgb}{.6, .6, .6}
\definecolor{Mygreen}{rgb}{.0, .7, .0}
\definecolor{Yellow}{rgb}{.55,.55,0}
\definecolor{Mustard}{rgb}{1.0, 0.86, 0.35}
\definecolor{applegreen}{rgb}{0.55, 0.71, 0.0}
\definecolor{darkturquoise}{rgb}{0.0, 0.81, 0.82}
\definecolor{celestialblue}{rgb}{0.29, 0.59, 0.82}
\definecolor{green_yellow}{rgb}{0.68, 1.0, 0.18}
\definecolor{crimsonglory}{rgb}{0.75, 0.0, 0.2}
\definecolor{darkmagenta}{rgb}{0.30, 0.0, 0.30}
\definecolor{magenta}{rgb}{0.50, 0.0, 0.50}
\definecolor{internationalorange}{rgb}{1.0, 0.31, 0.0}
\definecolor{darkorange}{rgb}{1.0, 0.55, 0.0}
\definecolor{ao}{rgb}{0.0, 0.5, 0.0}
\definecolor{awesome}{rgb}{1.0, 0.13, 0.32}
\definecolor{darkcyan}{rgb}{0.0, 0.50, 0.50}
\definecolor{violet}{rgb}{0.93, 0.51, 0.93}
\definecolor{brown}{rgb}{0.65, 0.16, 0.16}
\definecolor{orange}{rgb}{1.0, 0.65, 0.0}
\definecolor{darkgreen}{rgb}{0,.5,0}

\definecolor{Red}{rgb}{1, 0 ,0}
\definecolor{Blue}{rgb}{0, 0 ,1}

\newtheorem{theorem}{Theorem}[section]

\newaliascnt{question}{theorem}

\aliascntresetthe{question}

\newaliascnt{lemma}{theorem}
\newtheorem{lemma}[lemma]{Lemma}
\aliascntresetthe{lemma}

\newaliascnt{claim}{theorem}
\newtheorem{claim}[claim]{Claim}
\aliascntresetthe{claim}

\newaliascnt{invariant}{theorem}

\aliascntresetthe{invariant}

\newaliascnt{proposition}{theorem}
\newtheorem{proposition}[proposition]{Proposition}
\aliascntresetthe{proposition}

\newaliascnt{observation}{theorem}
\newtheorem{observation}[observation]{Observation}
\aliascntresetthe{observation}

\newaliascnt{corollary}{theorem}

\aliascntresetthe{corollary}

\newaliascnt{definition}{theorem}

\aliascntresetthe{definition}

\newaliascnt{conjecture}{theorem}

\aliascntresetthe{conjecture}

\newaliascnt{counterexample}{theorem}

\aliascntresetthe{counterexample}

\newcommand{\hh}{\end{document}}

\newcommand{\idpr}{\mathsf{idpr}\xspace}
\newcommand{\bidim}{\mathsf{bidim}\xspace}%
\newcommand{\bidi}{\mathsf{bidi}}

\newcommand{\gall}{\mathcal{G}_{{\text{\rm  \textsf{all}}}}}
\newcommand{\gplanar}{\mathcal{G}_{\text{\rm  \textsf{planar}}}}

\newcommand{\gedgeapex}{\mathcal{G}_{\text{\rm  \textsf{edge-apex}}}}

\newcommand{\gprojective}{\mathcal{G}_{\text{\rm  \textsf{projective}}}}

\newcommand{\gsinglycrossing}{\mathcal{G}_{\text{\rm \textsf{singly-crossing}}}}

\newcommand{\hw}{\mathsf{hw}\xspace}
\newcommand{\tw}{\mathsf{tw}\xspace}
\newcommand{\p}{\mathsf{p}\xspace}
\newcommand{\size}{\mathsf{size}\xspace}
\newcommand{\psize}{\mathsf{psize}\xspace}

\newcommand{\cupall}{\pmb{\pmb{\bigcup}}}

\newenvironment{cproof}{\proof[Proof of claim]}{\endproof}

\newcommand{\ground}{\ensuremath{\mathsf{ground}}}

\newcommand{\torso}{\ensuremath{\mathsf{torso}}}

\newcommand{\poly}{\text{$\mathsf{poly}$}\xspace}

\newcommand{\inG}{\text{$\mathsf{inner}$}\xspace}
\newcommand{\outG}{\text{$\mathsf{outer}$}\xspace}

\newcommand{\bd}{\mathsf{bd}\xspace}

\newcommand{\lin}[1]{\langle #1\rangle}

\usepackage{color}

\newcommand{\defi}[1]{{\emph{#1}}}

\title{Excluding Pinched Spheres}

\author{Laure Morelle\thanks{LIRMM, Univ Montpellier, CNRS, Montpellier, France. \\ Email: \texttt{laure.morelle@lirmm.fr}.}~$^{,}$\thanks{Supported by the ANR project  ELIT (ANR-20-CE48-0008-01), the French-German Collaboration ANR/DFG Project UTMA (ANR-20-CE92-0027), and by the Franco-Norwegian project PHC AURORA 2024 (Projet n°\! 51260WL).}\and
Evangelos Protopapas\thanks{Faculty of Mathematics, Informatics and Mechanics, University of Warsaw, Poland.\\ Email: \texttt{eprotopapas@mimuw.edu.pl}}~$^{,}$\thanks{Supported by the ERC project BUKA and the French-German Collaboration ANR/DFG Project UTMA (ANR-20-CE92-0027).}\and
Dimitrios M. Thilikos\thanks{LIRMM, Univ Montpellier, CNRS, Montpellier, France.\\  Email: \texttt{sedthilk@thilikos.info}.}~$^{,}$\thanks{Supported by the French-German Collaboration ANR/DFG Project UTMA (ANR-20-CE92-0027), the ANR project GODASse ANR-24-CE48-4377,  and by the Franco-Norwegian project PHC AURORA 2024-2025 (Projet n°\! 51260WL).}\and
\\ Sebastian Wiederrecht\thanks{KAIST, School of Computing, Daejeon, South Korea. \\ Email:
\texttt{wiederrecht@kaist.ac.kr}}}

\date{}

\begin{document}

\maketitle

\begin{abstract}
\noindent 
The \emph{pinched sphere}  is the pseudo-surface  $\Sbbb^{\circ}_0$ obtained by identifying two distinct points of the sphere.
We provide a structural characterization of graphs excluding an $\Sbbb^{\circ}_0$-embeddable graph as a minor.
Given a graph $G$ and a vertex set $X$, the \emph{bidimensionality} of $X$ in $G$ is the maximum $k$ such that $G$ contains the $(k\times k)$-grid as an $X$-rooted minor, i.e., there exists a minor model of the $(k \times k)$-grid in~$G$ such that every branchset of this model contains a vertex of $X$.
We prove that there is a function~$f$ such that, if a graph $G$ excludes an $\Sbbb^{\circ}_0$-embeddable graph $H$ as a minor, $G$ has a tree decomposition where each torso $G_{t}$ contains some set of vertices $X,$ whose bidimensionality in $G_{t}$ is at most $f(k)$  such that $G_{t}$ can be reduced to a graph embeddable in the projective plane by identifying vertices from $X$.
This result is optimal in the sense that every graph admitting such a tree decomposition must exclude some $\Sbbb^{\circ}_0$-embeddable graph as a minor.
An alternative interpretation of this result can be obtained by the fact that edge-apex graphs, i.e., graphs that can be made planar by removing an edge, are graphs embeddable in the pinched sphere.
Several consequences and variants of this min-max duality are discussed.
\end{abstract}

\medskip
\noindent \textbf{Keywords:} Pinched sphere,  Edge-apex graphs, Graph minors, Tree decompositions, Universal obstructions, Vertex identifications.

\thispagestyle{empty}

\newpage
\thispagestyle{empty}
\tableofcontents

\newpage
\section{Introduction}\label{secintroduction}

Let $\gall$ be the class of all finite graphs without loops or multiple edges.
A \defi{graph parameter} $\p:\gall\to\Nbbb$ is a function mapping graphs to non-negative integers. 
Given two graph parameters $\p,\p'$,  we write $\p\preceq \p'$ if there exists  some function $f:\Nbbb\to\Nbbb$
such that  for every graph $G$, $\p(G)≤f(\p'(G))$.
Also we say that $\p$ and $\p'$ are \defi{equivalent}, if  $\p\preceq \p'$ and  $\p\preceq \p'$ and we 
denote this by  $\p\sim\p'$.
\medskip


An important result in structural graph theory is the ``grid minor theorem'', proved by Robertson and Seymour in \cite{RobertsonS86exclu}. This theorem asserts that,
for every planar graph $H$, there is a constant $c_{H}$ such that every graph excluding $H$ as a minor  has a tree decomposition of width most  $c_H$ (see next paragraph for the definitions of tree decompositions and treewidth). This result  had  important structural and algorithmic consequences and a long line of research has been devoted 
to optimizing the dependence between the size of $H$ and the constant $c_H$ (see for instance \cite{RobertsonS94quick,chuzhoy2021towards}). Moreover, it is easy to see that the result is ``parametrically tight'' in the sense that every class of graphs of bounded treewidth should exclude some planar graph as a minor. 

\subsection{Excluding graphs with certain properties}

An important question is whether a similar tree decomposition theorem exists when the excluded graph is a non-planar graph.  A general answer to this question was provided by the celebrated 
Graph Minor Structure Theorem (GMST) asserting that graphs excluding a (non-necessarily planar) graph as a minor can be tree-decomposed into pieces enjoying certain structural characteristics of topological nature.
The GMST appeared in \cite{Robertson2003GMXVI} and has  
been one of the cornerstone results of the Graph Minors Series of Robertson and Seymour (see \cite{kawarabayashi2020quickly,kawarabayashi2018anewp,gorsky2025polynomial} for simpler and self-contained proofs with explicit bounds). This result, however -- due to its vast generality -- does not provide more refined structure when the excluded 
graphs enjoy particular structural properties. 
A general research program of proving refined versions of the GMST can be outlined as follows.
\begin{eqnarray}
\begin{minipage}{14cm}
\textsl{Given a class of graphs $\Hcal$, find a graph parameter, defined in terms of tree decompositions, 
such that its value  is bounded for graphs excluding some graph in $\Hcal$ as a minor and, moreover, its value  is unbounded for the graphs in $\Hcal$.}
\end{minipage}
\label{@constraint}
\end{eqnarray}
%
As already mentioned, when 
$\Hcal$ is the class of planar graphs, denoted by $\gplanar$, the parameter in \eqref{@constraint} should be chosen to be  treewidth or any other parameter equivalent to treewidth.

\begin{quote}
What if the graphs in $\Hcal$  are not planar but instead ``\textsl{close to being planar}''?
\end{quote}

Clearly the GMST, being a general and ``all purpose'' theorem,  provides only a partial answer to \eqref{@constraint} and it is a challenge to detect which tree-decomposition-based parameter is the suitable answer for fixed instantiations of the  class $\Hcal$.
This paper is motivated by this general question and provides a precise answer for the case where $\Hcal$ is the class of \textsl{edge-apex graphs}, i.e. graphs that can be made planar by removing an edge.
In fact our result applies for all graphs that are minors of edge-apex graphs which is precisely the class of all graphs embeddable in the \textsl{pinched sphere} $\Sbbb^{\circ}_0$ (see \cref{horn_torus} for a visualization).

\subsection{General decomposition framework}\label{subsec_frame}

To further formalize Question~\eqref{@constraint}, we introduce a general framework based on tree decompositions.

\paragraph{Tree decompositions.}

A \defi{tree decomposition} of a graph $G$ is a tuple $\mathcal{T}=(T,\beta)$ where $T$ is a tree and $\beta\colon V(T)\rightarrow 2^{V(G)}$ is a function, whose images are called the \emph{bags} of $\mathcal{T},$ such that
\begin{enumerate}
\item $\bigcup_{t\in V(T)} \beta (t)= V (G),$
\item for every $e\in E(G)$, there exists $t\in V(T)$ with $e\subseteq \beta(t),$ and
\item for every $v\in V(G)$, the set $\{t\in V(T) \mid v\in \beta(t)\}$ induces a subtree of $T.$
\end{enumerate} 
The \defi{width} of $\mathcal{T}$ is the maximum size of a bag minus $1$ and the \defi{treewidth} of $G$, denoted by $\tw(G),$ is the minimum width over all tree decompositions of $G$.
The \defi{torso} of $\mathcal{T}$ at a node $t$ is the graph, denoted by $G_{t}$, obtained by adding all edges between the vertices of $\beta (t) \cap \beta (t')$ for every neighbor $t'$ of $t$ in $T$. 
The \defi{clique-sum closure} of a graph class $\mathcal{G}$, denoted by  $\mathcal{G}^\star$, is the graph class containing  every graph that has a tree decomposition whose torsos belong to $\mathcal{G}$.

One may use tree decompositions in order to define graph parameters using simpler ones as follows: if $\p \colon \gall \to \Nbbb$ is a graph parameter, then we define the \defi{clique-sum extension} of $\p$ as the graph parameter $\p^\star \colon \gall \to \mathbb{N}$ such that
$$\p^\star(G) = \min\big\{ k \mid G \in \{H \mid \p(H) \leq k\}^\star \big\},$$
in other words, $\p^\star(G) \leq k$ if and only if $G$ is in the clique-sum closure of all graphs for which the value of $\p$ is at most $k$.
Using this notation and given some graph class $\Hcal$, we may formalize Question~\eqref{@constraint} as follows:
\begin{eqnarray}
\begin{minipage}{14cm}
\textsl{Find a graph parameter $\p_{\Hcal} \colon \gall \to \Nbbb$, such that
\begin{itemize}
\item[(A)] there is a function $f \colon \Nbbb\to\Nbbb$ where for every $H \in \Hcal$, if a graph $G$ excludes the graph $H$ as a minor then $\p_{\Hcal}^\star(G) \leq f(|H|)$, and
\item[(B)] the values of $\p_{\Hcal}^\star$ for the graphs in $\Hcal$ are unbounded.
\end{itemize}}
\end{minipage}
\label{@consstraint}
\end{eqnarray}

Note that, if we have an answer to the question above for some $\Hcal$, then we essentially have an answer for its \emph{minor-closure}, that is the class of all minors of graphs in $\Hcal$.
Motivated by this, we always assume that $\Hcal$ is a minor-closed graph class, i.e., contains all minors of its elements.
One may observe that if $\p_{\Hcal}$ is an answer to \eqref{@consstraint} then any parameter equivalent to $\p_{\Hcal}$ is also a correct answer.
In fact, every parameter $\p'$ where $\p_{\Hcal}^{\star} \preceq \p' \preceq \p_{\Hcal}$ is also a correct answer\footnote{Because, for any parameter $\p$, $(\p^\star)^\star=\p^\star$, and thus $(\p')^\star\sim\p_{\Hcal}^{\star}$.}.
We refer to condition (A) (resp. (B)) as the \emph{upper bound} (resp. \emph{lower bound}) of \eqref{@consstraint}.
For the case of the class $\gplanar$ of planar graphs, we may pick $\p_{\gplanar} = \size$, where $\mathsf{size} \colon \gall\to\Nbbb$ maps a graph $G$ to its size $|V(G)|$.
Indeed, it is enough to observe that the treewidth of a graph $G$ is equal to $\mathsf{size}^\star(G) - 1$.

In the most general case where $\Hcal=\gall$, an answer to \eqref{@consstraint} is given by the GMST. Here, we may pick $\p_{\gall}$
as the parameter mapping a graph $G$ to the minimum $k$ for which 
$G$ contains at most $k$ vertices, called \emph{apices} whose removal yields a graph that is ``$k$-almost embeddable'' in the sense that it admits a drawing with at most $k$ ``vortices'' each of ``width'' at most $k$
in a surface of Euler genus at most $k$.
We postpone the technical definition of vortices to \cref{sec_draw} and their width to \autoref{sec_global}. Instead, 
we wish to give an alternative, more streamlined, answer to \eqref{@consstraint} for the case $\Hcal=\gall$, that is based on the notion of \textsl{bidimensionality} (see \cite{ThilikosW2025excluding}).

\paragraph{Rooted minors and bidimensionality.}

Given two graphs $G$ and $H$ and a set $X \subseteq V(G)$, we say that $H$ is an \defi{$X$-rooted minor} of $G$ (or simply an \emph{$X$-minor} of $G$) if there is a collection $\Scal = \{S_v \mid v \in V(H)\}$ of pairwise-disjoint connected\footnote{A set $X$ is \emph{connected} in $G$ if the subgraph of $G$ induced by $X$ is a connected graph.} subsets of $V(G)$, each containing at least one vertex of $X$ and such that, for every edge $xy \in E(H)$, the set $S_x \cup S_y$ is connected in $G$.
An \defi{annotated graph} is a pair $(G, X)$ where $G$ is a graph and $X$ is a set of vertices of $G$.

The \emph{bidimensionality} of a vertex set $X$ of a graph $G,$ denoted by $\mathsf{bidim}(G,X)$, is the maximum $k$ for which $G$ contains a $(k \times k)$-grid as an $X$-minor.
Intuitively, bidimensionality measures to what extent $X$ can be ``spread'' as a 2-dimensional grid inside the graph $G$.
%
 
According to \cite{ThilikosW2025excluding}, an answer to \eqref{@consstraint} when $\Hcal = \gall$ is given by $\p_{\gall}$ where $\p_{\gall}(G)$ is the minimum $k$ for which $G$ contains a vertex set $X$ with $\bidim(G,X) \leq k$ such that $G - X$ can be embedded in a surface of Euler genus at most $k$.
Notice that the definition of  $\p_{\gall}$ is based on a \textsl{modification operation}: We are looking for a set $X$ of ``bounded'' bidimensionality whose removal yields a graph embeddable in a surface of ``bounded''  Euler genus.
So the modification is stated in terms of \textsl{vertex removals} and the \textsl{target} of this modification is surface embeddability.
More refined results have been proven in \cite{ThilikosW2025excluding} providing an answer to Question~\eqref{@consstraint} in the case where $\Hcal$ is the class of graphs embeddable in some particular surface.
For every choice of a surface, the resulting graph parameter is defined in terms of modifications based on vertex deletion.

\paragraph{Excluding singly-crossing graphs.}

The focus of this paper is to find an answer to Question \eqref{@consstraint} for graph classes $\Hcal$ that are ``close to be planar''.
In this direction, Robertson and Seymour considered in \cite{RobertsonS91} the class of \defi{singly-crossing graphs}, denoted by $\gsinglycrossing$, i.e., graphs that can be drawn in the sphere such that there is at most one pair of edges that share a common point.
According to \cite{RobertsonS91}, if $G$ excludes a singly-crossing graph as a minor, then $G$ has a tree decomposition whose non-planar torsos have size that is bounded by some constant, depending on $H$.
Assume now that $\psize \colon \gall \to \Nbbb$ is the graph parameter where $\psize(G)$ is zero if $G$ is planar, otherwise $\psize(G)$ is the size of $G$.
It can be proven that the values of $\psize^\star$ are unbounded for the graphs in $\gsinglycrossing$.
This, along with the the aforementioned result of \cite{RobertsonS91} implies that we may pick $\p_{\gsinglycrossing} = \psize$ (or any other equivalent parameter)
as a valid answer to \eqref{@consstraint} when $\Hcal = \gsinglycrossing$.

%
%
%

\subsection{Main results}

\paragraph{Edge-apex graphs.}

Our results deal with a more relaxed version of being ``close to be planar''.
As we have already mentioned, a graph $G$ is an \defi{edge-apex graph} if it contains some edge $e$ where $G - e \in \gplanar$.
Notice that every singly-crossing graph is obviously an edge-apex graph, however the inverse is not correct as the edge $e$ to be removed may ``fly over'' more than one edge in any embedding of $G - e$.
We use $\gedgeapex$ for the class of all minors of edge-apex graphs.

\paragraph{A pseudo-surface interpretation of $\gedgeapex$.}

Given a surface $\Sigma$, the \emph{pinched} version of $\Sigma$, denoted by $\Sigma^{\circ}$, is the pseudo-surface obtained by identifying two distinct points of $\Sigma$.
A result of Knor \cite{Knor96charac} says that for every surface $\Sigma$, the class of graphs embeddable in $\Sigma^{\circ}$ is minor-closed.
One may define pseudo-surfaces by identifying more sets of points of surfaces or unions of surfaces.
However, as shown in \cite{Knor96charac}, minor-closedness of graphs embeddable in pseudo-surfaces holds only for pinched surfaces and for those pseudo-surfaces that are ``spherically reducible'' to pinched surfaces (see also \cite{SiranG92kurato}).

\begin{figure}[ht]
\begin{center}
\includegraphics[scale=0.3]{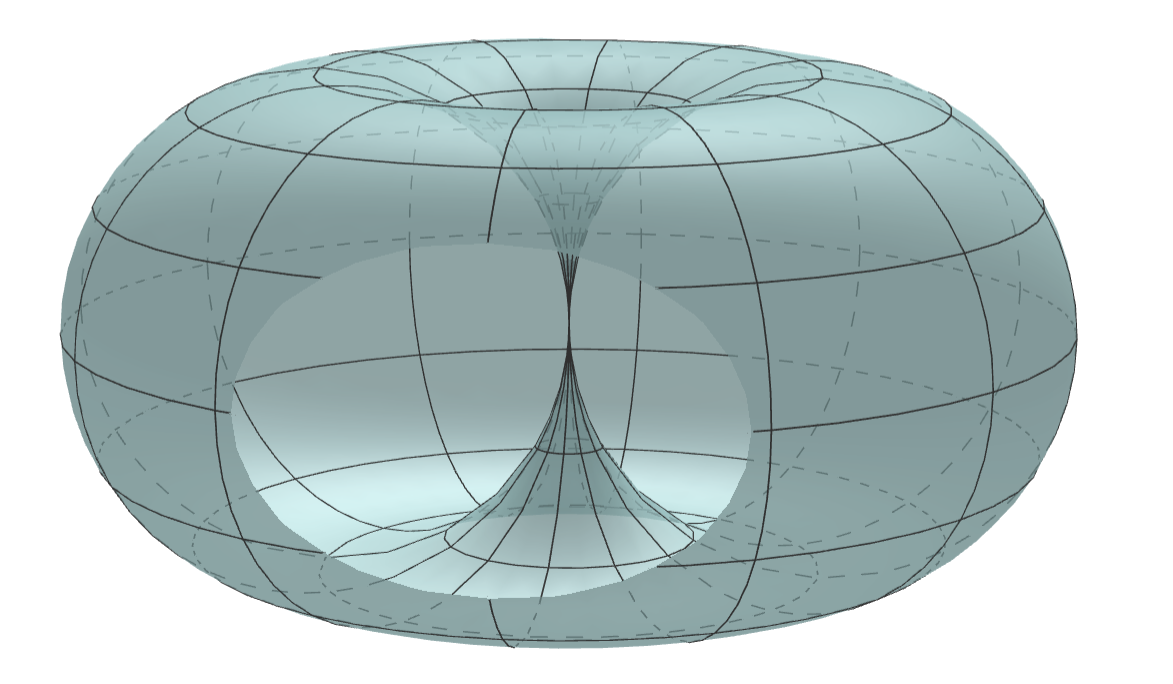}
\end{center}
\caption{A horn torus with a hole, permitting the visibility of the pinch point.}
\label{horn_torus}
\end{figure}

We denote the sphere by $\Sbbb_{0}$, that is the set of points $\{ (x,y,z) \in \Rbbb^{3} \mid x^{2} + y^{2} + z^{2} = 1 \}$.
Now, the pinched sphere $\Sbbb^{\circ}_0$ can be seen as the set of points of the horn-torus
$$\{ (x,y,z) \in \Rbbb^{3} \mid (x^2+y^2+z^{2})^2 = 4(x^2+y^2) \}$$
depicted in \cref{horn_torus}.
It is easy to verify the following.

\begin{observation}
The class $\gedgeapex$ of graphs is precisely the class of all $\Sbbb^{\circ}_0$-embeddable graphs.
\end{observation}

At this point we wish to mention that the result in \cite{RobertsonS91} has a similar embeddability interpretation.
Just consider the sphere with some particular point ``duplicated'': the graphs in $\gsinglycrossing$
are exactly the graphs that are embeddable in this slightly augmented sphere.

\paragraph{Our results.}

Our purpose is to define a graph parameter $\p_{\gedgeapex}$ that may serve as the answer to \eqref{@consstraint} in the case where $\Hcal = \gedgeapex$.
Motivated by the aforementioned results of \cite{ThilikosW2025excluding} on surface-embeddable graphs, one may suspect that the answer is again some modification-based parameter where the modification operation is vertex deletion.
As we show, it is correct that we have to consider some modification-based parameter, however vertex deletion ceases to be the correct operation.
Instead, the correct modification operation turns out to be \textsl{vertex identifications}.
To formalize this, we need additional definitions.

Given an annotated graph  $(G,X)$, we define $\Mcal(G,X)$ as the class of all graphs obtainable from $G$ by identifying vertices in $X$ (see in the end of \autoref{prelims_defs} for the formal definition).
We say that a graph is \emph{projective} if it is embeddable in the projective plane and we use $\gprojective$ to denote the class of projective graphs.
We define $\mathsf{idpr} \colon \gall \to \Nbbb$ as the graph parameter where
\begin{align}
\mathsf{idpr}(G) \ &\coloneqq \ \min\{k \mid \text{there exists } X \subseteq V(G)\text{ s.t.}\ \bidim(G,X) \le k \text{~and~} \Mcal(G,X) \cap \gprojective \neq \emptyset \},\label{apoxbisintw}
\end{align}
in other words, $\mathsf{idpr}(G) \leq k$ if and only if $G$ contains a set $X$ of bidimensionality at most $k$ such that $G$ can be turned into a projective graph by identifying vertices of $X$ in $G$.

Our results imply that $\mathsf{idpr}$ provides an answer to \eqref{@consstraint} for the class $\gedgeapex$.
In particular, we can answer to \eqref{@consstraint} in the case of $\Sbbb^{\circ}_0$-embeddable graphs as follows.

\begin{theorem}[Upper bound]\label{upper_b}
There exists a function $f \colon \Nbbb\to\Nbbb$ such that, for every $H \in \gedgeapex$, every $H$-minor-free graph $G$ satisfies $\mathsf{idpr}^\star(G) \leq f(|H|)$.
That is, any such graph $G$ admits a tree decomposition $(T,\beta)$ where every torso $G_{t}$, $t\in V(T)$, contains a set $X$ of bidimensionality at most $f(|H|)$ in $G_{t}$ such that $G_t$ can be made projective by identifying vertices in $X$.
\end{theorem}

\begin{theorem}[Lower bound]\label{lower_b}
For every $k \in \Nbbb$, there is a graph $G$ in $\gedgeapex$ such that $\mathsf{idpr}^\star(G) \geq k$.
 \end{theorem}

%

%

Let us define a variant of $\mathsf{idpr}$, namely $\mathsf{idpl}$ by considering in \eqref{apoxbisintw} $\gplanar$ instead of $\gprojective$, i.e., we now demand that the surfaces where the torsos of the tree decomposition are almost embedded are spheres. 
By a simple variant of our proof strategy, we obtain the following variant of \cref{upper_b}.

\begin{theorem}\label{upper_sb}
There exists a function $f \colon \Nbbb^2\to\Nbbb$ such that for every  $H_{1} \in \gprojective$ and every $H_2 \in \gedgeapex$, every $\{ H_1,H_2\}$-minor-free graph $G$ satisfies $\mathsf{idpl}^\star(G) \leq f(|H_1|,|H_2|)$.
That is, $G$ admits a tree decomposition $(T,\beta)$ where every torso $G_{t}$, $t\in V(T)$, contains a set $X$ of bidimensionality at most $f(|H_1|,|H_2|)$ in $G_{t}$ such that $G_t$ can be made planar by identifying vertices in $X$. 
\end{theorem}

The optimality of the decomposition of \cref{upper_sb} for the set of classes $\{\gprojective, \gedgeapex\}$ follows from \cref{lower_b}.
The counterpart for $\gprojective$ and $\mathsf{idpl}^\star$ is provided by \cref{i9opb1}.

\subsection{Overview of the proof}
An important idea for our approach is to express the class $\gedgeapex$ using parametric graphs.

\paragraph{Parametric graphs.}

A \emph{parametric graph} is a sequence $\mathscr{H} \coloneqq \langle \mathscr{H}_{t} \rangle_{t \in \mathbb{N}}$ of graphs such that, for every $t \in \mathbb{N}$, $\mathscr{H}_{t}$ is a minor of $\mathscr{H}_{t+1}$.
Given two parametric graphs $\mathscr{H}^1 = \langle \mathscr{H}_{t}^1 \rangle_{t \in \Bbb{N}}$ and $\mathscr{H}^2 = \langle \mathscr{H}_{t}^2 \rangle_{t \in \Bbb{N}}$, we write $\mathscr{H}^1 \lesssim \mathscr{H}^2$ if there is a function $f \colon \mathbb{N}\to\mathbb{N}$ such that, for every $k \in \mathbb{N}$, $\mathscr{H}_{k}^1$ is a minor of $\mathscr{H}_{f(k)}^2$.
In case $f$ is a linear function, we say that $\mathscr{H}^1$ and $\mathscr{H}^2$ are \emph{linearly equivalent}.

Given a parametric graph $\mathscr{H} \coloneqq  \langle \mathscr{H}_{t} \rangle_{t \in \mathbb{N}}$, we define the parameter $\p_{\mathscr{H}} \colon \gall\to\Nbbb$ so that $\p_{\mathscr{H}}(G)$ is the maximum $k$ for which $G$ contains $\mathscr{H}_{k}$ as a minor.
Notice that $\mathscr{H}^1$ and $\mathscr{H}^2$ are equivalent if and only if $\p_{\mathscr{H}^1}$ and $\p_{\mathscr{H}^2}$ are equivalent.
Moreover, notice that $\mathsf{p}_{\mathscr{H}}$ is a \emph{minor-monotone} parameter, i.e., for every graph $G$ and every minor $H$ of $G$, we have $\p_{\mathscr{H}}(H) \leq \p_{\mathscr{H}}(G)$.

\paragraph{Cylindrical grids and their enhancements.}

For every two non-negative integers $t_1$ and $t_2$, we define the \emph{$(t_1 \times  t_2)$-cylindrical grid} as the Cartesian product of a cycle on $t_1$ vertices and a path on $t_2$ vertices.
Notice that the $(t_1 \times  t_2)$-cylindrical grid contains $t_2$ cycles (each with $t_1$ vertices) and $t_1$ paths crossing them (each with $t_2$ vertices).

The \emph{annulus grid} or order $t$, denoted by $\mathscr{A}_t,$ is the $(4t \times  t)$-cylindrical grid.
We define four parametric graphs: The annulus grid $\mathscr{A} = \lin{\mathscr{A}_t}_{t \in \Nbbb}$, the single-cross grid $\mathscr{S}= \lin{\mathscr{S}_t}_{t \in \Nbbb}$, the long-jump grid  $\mathscr{J} = \lin{\mathscr{J}_t}_{t \in \Nbbb}$, and the crosscap grid $\mathscr{C}_{t}$, where $\mathscr{S}_t$ (resp. $\mathscr{J}_t$, $\mathscr{C}_t$) are obtained by adding add two (resp. $t+1$, $2t$) edges in $\mathscr{A}_t$, as indicated in \cref{first_par_gr}.
In any case, we refer to the index $t$ as the \emph{order} of the corresponding parametric graph.

\begin{figure}[ht]
\begin{center}
\includegraphics[scale=0.77]{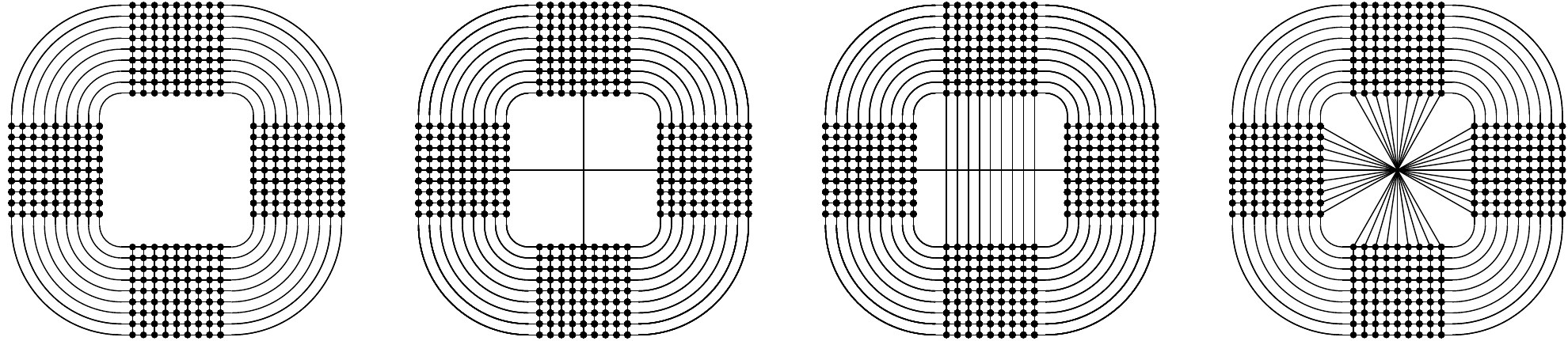}
\end{center}
\caption{The annulus grid $\mathscr{A}_{9}$, the single-cross grid $\mathscr{S}_{9}$, the long-jump grid $\mathscr{J}_{9}$, and the crosscap grid $\mathscr{C}_{9}$.}
\label{first_par_gr}
\end{figure}

As was already observed in \cite[(1.5)]{RobertsonS94quick}, planar graphs are exactly the minors of the graphs from $\mathscr{A}$ (note that \cite[(1.5)]{RobertsonS94quick} is stated in terms of grids that, seen as parametric graphs, are equivalent to the annulus grids).
With a simple adaptation of the proof of \cite[(1.5)]{RobertsonS94quick} (see also \cite{gavoille2023minoruniversal}), one may prove the following

\begin{proposition}
\label{tsw_asll}
$\gsinglycrossing$ is the set of all minors of the graphs from $\mathscr{S}$ and $\gedgeapex$ is the set of all minors of the graphs from $\mathscr{J}$.
\end{proposition}

Using the terminology of parametric graphs and \cref{tsw_asll}, the answer to Question \eqref{@consstraint} for the case of $\gplanar$ is $\size$, because $\p_{\mathscr{A}} \sim \size^\star$ (from \cite{RobertsonS86exclu}), and the answer to Question \eqref{@consstraint} for the case of $\gsinglycrossing$ is $\psize$, because $\p_{\mathscr{S}} \sim \psize^\star$ (from \cite{RobertsonS91}).

Our objective is to prove that the  the answer to Question \eqref{@consstraint} for the case of $\gedgeapex$ is $\mathsf{idpr}$.
This is because $\p_{\mathscr{J}} \sim \mathsf{idpr}^\star$.
Indeed, \cref{upper_b} is equivalent to $\mathsf{idpr}^\star \preceq \p_{\mathscr{J}}$ and \cref{lower_b} is equivalent to $\p_{\mathscr{J}}\preceq \mathsf{idpr}^\star$.
The upper bound $\mathsf{idpr}^\star \preceq \p_{\mathscr{J}}$ is proven in Sections \ref{sec_lst}, \ref{sec_loc_to_glob}, and \ref{sec_global} and the lower bound $\p_{\mathscr{J}} \preceq \mathsf{idpr}^\star$ is given in \cref{power_b_ec}.

For the proof of the upper bound we assume that a graph $G$ excludes $\mathscr{J}_{k}$ as a minor and and then prove that $G$ admits a tree decomposition where each torso is ``$t$-almost embeddable'' in the projective plane, for some $t$ that depends on $k$.

{Our proof is inspired by the local structure theorem in \cite{kawarabayashi2020quickly}, stating that for every ``big enough'' wall~$W$ in $G$ there is a bounded-size vertex set  of ``apex'' vertices $A$  such that $G-A$ has a $\Sigma$-decomposition: i.e., $G'\coloneqq G-A $ can be drawn in a surface $\Sigma$ of  bounded Euler genus 
so that all edge crossings appear inside a collection of ``cells'' of $\Sigma$ such that,  among these cells, those with more than 3 vertices in their boundary are 
called  \textsl{vortices}, are ``few'', and are of bounded ``depth'' (we postpone 
the formal definitions  around $\Sigma$-decompositions  in \cref{sec_draw}).

Our main structural result is that, in absence of a $\mathscr{J}_{k}$-minor, any large enough wall $W$ is ``flat'' in a $\Sigma$-decomposition as the above where $A=\emptyset$ and, moreover, the surface $\Sigma$ is either the sphere or the projective plane.
For our proof, in \cref{sec_lst} we provide a preliminary refinement of the structure of vortices.  
These results are used in  \cref{sec_loc_to_glob} to prove that the absence of $\mathscr{J}_{k}$ permits us to assume that $A=\emptyset$.
This concludes with the proof of our local structure theorem.}

Our next aim is to transform our local structure theorem into a global one.
That is, we wish to obtain a tree decomposition where every torso of ``big enough'' size admits a $\Sigma$-decomposition like the one from our local structure theorem.
A major obstacle here is that the standard proof technique to go ``from  local to global'', 
introduced in \cite{RobertsonS91obst} and used in \cite{kawarabayashi2020quickly, diestel2016graph, DiestelKMW12onth, ThilikosW24kill, DvorakT14list}, does not work for our purpose since the inductive set-up requires the introduction of additional apex vertices that do not inherently stem from the local structure theorem.
In \cite{DvorakT14list}, Dvořák and  Thomas had to overcome a similar problem and did so by introducing additional vortices.
We propose a new approach, tailor-made for our setting, that allows us to go from local to global without the introduction of additional apices and avoiding some of the drawbacks of the method in \cite{DvorakT14list}.
Our approach, as presented in \cref{sec_global}, is based on the fact that the surface of our decomposition is the projective plane or the sphere and that these are the only two surfaces where the removal of a \textsl{any} cycle creates a connected component that is a disk.

The next step is to transform our global decomposition to one that bounds $\mathsf{idpr}(G)$ according to \eqref{apoxbisintw}.
Let  $X$ be the vertices drawn in the vortices of the $\Sigma$-decomposition of some torso $G_{t}$.
An important property of such decompositions, proved in \cite{ThilikosW2025excluding}, is that the bidimensionality of $X$ in $G_{t}$ is ``upper bounded''.
We may now identify all vertices  drawn in some vortex  into a single vertex and, by performing these identifications for each vortex, obtain a planar or a  projective graph.
This completes the outline of the steps we use to prove the upper bound.

For the proof of the lower bound, we observe first that $\mathsf{idpr}$, and thus $\mathsf{idpr}^\star$, is minor-monotone (\autoref{lem_id_min}), and we prove that $\mathsf{idpr}^\star(\mathscr{J}_{k}) = \Omega(k^{\alpha})$ for some $\alpha > 0$ (\autoref{jopegftasi}).
In \cref{sub1}, we first prove that $\mathsf{idpr}(\mathscr{J}_{k}) = \Omega(k^{1/4})$ and then use a general purpose result from \cite{ThilikosW2025excluding} to show that this lower bound also holds for the clique-sum extension of $\mathsf{idpr}$.

\medskip
We stress that the proof of \cref{upper_sb} follows a simple variant of the steps above where the absence of a crosscap-grid as a minor further imposes that the surfaces in the obtained decompositions are spheres (see \cref{mainresult_pl}).  
 
\section{Preliminaries}
\label{prelims_defs}

\paragraph{Sets and integers.}

We denote the set of non-negative integers by $\mathbb{N}$.
Given two integers $p, q,$ where $p \leq q,$ we denote by $[p, q]$ the set $\{p, \dots, q\}.$
For an integer $p \geq 1,$ we set $[p] = [1, p]$ and $\mathbb{N}_{\geq p} = \mathbb{N} \setminus [0, p - 1].$
For a set $S,$ we denote the set of all subsets of $S$ by $2^{S}$ and the set of all subsets of $S$ of size $2$ by $S \choose 2$.
If $\mathcal{S}$ is a collection of objects where the operation $\cup$ is defined, we denote $\cupall \Scal = \bigcup_{X \in \mathcal{S}} X.$

\paragraph{Basic concepts on graphs.}

A graph $G$ is any pair $(V, E)$ where $V=V(G)$ is a finite set and $E=E(G) \subseteq {V \choose 2},$ i.e. all graphs in this paper are undirected, finite, and without loops or multiple edges.
We refer the reader to~\cite{diestel2016graph}  for any undefined terminology. 
We denote an edge $\{x,y\}$ by $xy$ (or $yx$).
We write $\gall$ for the set of all graphs.
Given a vertex $v \in V(G),$ we denote by $N_{G}(v)$ the set of vertices of $G$ that are adjacent to $v$ in $G.$
For a set $S \subseteq V(G),$ we set $N_{G}(S) = \bigcup_{v \in S} N_{G}(v) \setminus S$.
For $S \subseteq V(G),$ we set $G[S] = (S, E \cap {S \choose 2})$ and use $G - S$ to denote $G[V(G) \setminus S].$ We say that $G[S]$ is an \emph{induced \emph{(}by $S$\emph{)} subgraph} of $G$.
Given two graphs $G_{1}$ and $G_{2},$ we define their union as $G_{1}\cup G_{2}=(V(G_{1})\cup V(G_{2}),E(G_{1})\cup E(G_{2})).$

Given an edge $e = uv \in E(G),$ we define the \emph{subdivision} of $e$ to be the operation of deleting $e,$ adding a new vertex and making it adjacent to $u$ and $v.$ 
Given two graphs $H$ and $G,$ we say that $H$ is a \emph{subdivision} of $G$ if $H$ can be obtained from $G$ by subdividing edges.

\paragraph{Bridges.}
Let $H$ be a subgraph of a graph $G$.
An \emph{$H$-bridge} in $G$ is a connected subgraph $B$ of $G$ such that none of its edges is an edge of $H$ and either $E(B)$ consists of a unique edge with both endpoints in $H$, or for some connected component $C$ in $G-V(H)$, $E(B)$ is the set of all edges of $G$ with at least one endpoint in $V(C)$.
The vertices in $V(B)\cap V(H)$ are called the \emph{attachments} of $B$.

\paragraph{Minors.}
A graph $H$ is a \emph{minor} of a graph $G$ if $H$ can be obtained from a subgraph of $G$ by a series of edge contractions.
Equivalently, $H$ is a minor of $G$ if there is a collection $\Scal=\{S_v\mid v\in V(H)\}$ of pairwise-disjoint connected subsets of $V(G)$ such that, for each edge $xy\in E(H)$, the set $S_x\cup S_y$ is connected in $V(G)$. $\Scal$ is called a \emph{model} of $H$ in $G$.

\paragraph{Convention on $\mathscr{J}$ and $\mathscr{C}$.}

In the previous section we defined the parametric graphs $\mathscr{J}$ (the long-jump grid) and $\mathscr{C}$ (the crosscap grid) presented in \autoref{first_par_gr}.
For ease of proof in the rest of the paper, in place of $\mathscr{J}$ and $\mathscr{C}$ we consider the parametric graphs given in \autoref{figlongjumpbis}, that are obviously linearly equivalent to the ones presented in the introduction.

\begin{figure}[ht]
\begin{center}
\includegraphics[scale=.9]{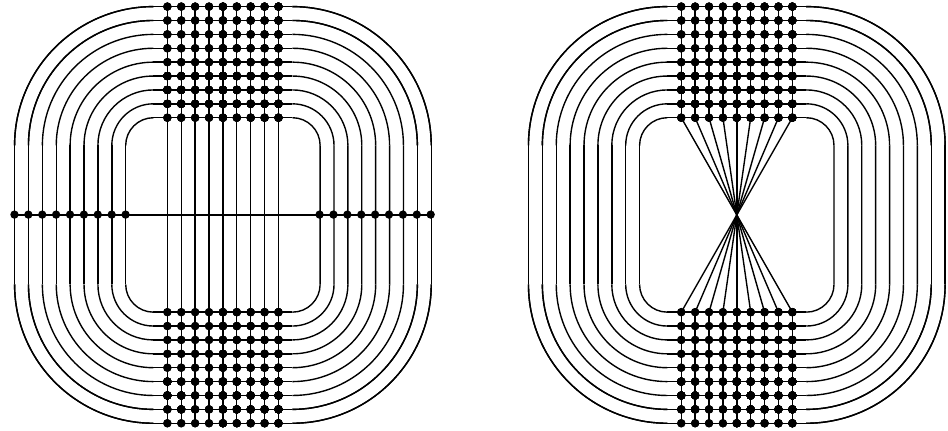}
\end{center}
\caption{The long-jump grid $\mathscr{J}_k$ (left) and the crosscap grid $\mathscr{C}_k$ (right) for the rest of the paper (here $k=9$).}
\label{figlongjumpbis}
\end{figure}

\paragraph{Partitions.}

Given $p \in \Nbbb$, a \emph{$p$-partition} of a set $X$ is a tuple $(X_1, \ldots, X_p)$ of non-empty pairwise disjoint subsets of $X$ such that $X = \bigcup_{i \in [p]} X_i$.
A \emph{partition} of $X$, denoted by $\Pcal(X)$, is a $p$-partition of $X$ for some $p \in \Nbbb$.
Given a set $U$, two subsets $X, A \subseteq U$, and $\Xcal = (X_1, \ldots, X_p) \in \Pcal(X)$, $\Xcal \cap A$ denotes the partition $(X_1 \cap A, \ldots, X_p \cap A)$ of $X \cap A$.

\paragraph{The identification operation.}

Let $G$ be a graph and $u,v\in V(G)$.
The \emph{identification of $u$ and $v$} in $G$ is the operation that transforms $G$ in a graph $G'$, denoted by $G/\!\!/\{u,v\}$, obtained by deleting $u$ and $v$ from $G$ and adding instead a new vertex $w$ adjacent to every vertex in $N_G(u)\cup N_G(v)$.
Note that contrary to a contraction, we do not require $u$ and $v$ to be adjacent.
With a slight abuse of notation, the vertex $w$ obtained from the identification of $u$ and $v$ may also be referred to as either $u$ or $v$.

Given a graph $G$ and $S\subseteq V(G)$, 
the \emph{identification of all vertices in $S$} in $G$ is the operation that transforms $G$ in a graph $G'$, denoted by $G/\!\!/S$, obtained by identifying an arbitrary vertex $s\in S$ with all vertices in $S\setminus s$.

Remark that no matter the choice of $s$ and the order of the identifications, the resulting graph $G'$ is the same.
For instance, $G/\!\!/\{u,v,w\}=G/\!\!/\{u,v\}/\!\!/\{u,w\}=G/\!\!/\{u,w\}/\!\!/\{u,v\}=G/\!\!/\{u,v\}/\!\!/\{v,w\}.$
Note also that given $S_1,S_2\subseteq V(G)$ such that $S_1\cap S_2\ne\emptyset$, we have
$G/\!\!/S_1/\!\!/S_2=G/\!\!/(S_1\cup S_2)$.

Let $X\subseteq V(G)$ and $\Pcal=(S_1,...,S_r)\in\Pcal(X)$ be a partition of $X$.
We denote $X/\!\!/\Pcal$ the sequence of identifications $X/\!\!/S_1/\!\!/S_2/\!\!/.../\!\!/S_r$.
Note that the graph obtained through this sequence of identification is the same up to a permutation of the order in the partition.

We denote by $\Mcal(G,S)$ the set of all graphs $G'$ that can be obtained by a sequence of identifications of vertices of $S$.
In other words, for each $G'\in\Mcal(G,S)$, there is a partition $\Pcal=(S_1,...,S_r)\in\Pcal(S)$ such that $G'=G/\!\!/\Pcal$.

As we already mentioned, the identification operation plays and important 
role in the definition of the parameters $\mathsf{idpr}$ and $\mathsf{idpl}$ 
and
the statement of our main results \cref{upper_b} and \cref{upper_sb}. 

%

%
%

\section{Excluding a long-jump transaction from a society}
\label{sec_lst}

In this section, and in Sections \ref{sec_loc_to_glob} and \ref{sub1}  we prove \autoref{upper_b} in three steps.
We first show a result (\autoref{mainter}) on societies (the relevant definitions are provided in \autoref{sec_draw})  in \autoref{sec_lst} that essentially says the following:
a society $(G,\Omega)$ either contains a big long-jump transaction, or has a rendition in the projective plane with a bounded number of vortices, all of small depth.
We then combine our result on societies with a new version of the flat wall theorem (\autoref{th:FWt}) that avoids the introduction of apices to obtain \autoref{mainresult}.
This local structure theorem does not introduce apices and provides a ``$\Sigma$-decomposition'' (see \autoref{sec_draw}) where $\Sigma$ is the projective plane.
That is, if $G$ has big treewidth, then either $G$ contains a big long-jump grid as a minor, or $G$ has a $\Sigma$-decomposition with a bounded number of vortices, that are of small depth, where $\Sigma$ is the projective plane.
We then use in \autoref{sec_global} our local structure theorem to prove a first global theorem (\autoref{thm_projective_global}) that says that if $G$ excludes a long-jump grid as a minor, then $G$ has a tree decomposition such that the torso at each bag has an almost embedding in the projective plane with a bounded number of vortices, that are of small depth.
Finally, we deduce from this first global structure theorem the one of \autoref{upper_b} (more precisely \autoref{th_proj_global}) by identifying each vortex to a single vertex.
As our local decomposition is ``apex-less'', this needs some special attention 
due to the fact that the classic ``local to global'' argument may introduce apices.

Additionally, we also give similar results in case we exclude both a long-jump grid and a crosscap grid (\autoref{mainresult_pl}, \autoref{corr_plan_global}, \autoref{th_plan_global}).
The results are 
similar (and simpler), as the surface is now the sphere instead of the projective plane.

The proof of \autoref{mainter} is rather involved, so we provide here a quick sketch of the proof, assuming knowledge of the concepts defined in \cref{sec_draw}.

As said above, the goal here is to prove that a society $(G,\Omega)$ excluding a long-jump transaction (hence such that $G$ excludes a long-jump grid has minor) has a rendition in the \emph{projective plane} with a bounded number of vortices of small depth (\autoref{mainter}).
We use many techniques from \cite{kawarabayashi2020quickly}, where Kawarabayashi, Thomas, and Wollan prove that a society $(G,\Omega)$, where $G$ excludes a big clique as a minor, has a rendition in the some bounded genus surface,  with a bounded number of apices, and a bounded number of vortices, all of small depth.
Given that we exclude a simpler graph, most results in terms of cliques can be simplified by directly finding the long-jump (\autoref{allinone2}).

To prove \autoref{mainter}, we first show that a society $(G,\Omega)$ excluding a long-jump transaction either has a crosscap transaction $\Qcal$, or a rendition in the \emph{plane} with a bounded number of vortices, that are of small depth (\autoref{main1}).
In the second case, we can immediately conclude.
In the first case, we find another society $(G',\Omega')$ inside $(G,\Omega)$ avoiding $\Qcal$, to which we again apply \autoref{main1}.
If we find a second crosscap transaction, this would imply that $(G,\Omega)$ contains a long-jump, a contradiction.
Otherwise, $(G',\Omega')$ has a rendition in the \emph{plane} with a bounded number of vortices of small depth, and thus $(G,\Omega)$ has a rendition in the \emph{projective plane} with a bounded number of vortices of small depth.

To prove \autoref{main1}, we first observe that if $(G,\Omega)$ excludes a long-jump of order $k$ as a minor, then it either has a rendition in the plane with a unique vortex of small depth, or contains a big crosscap transaction, or a big planar transaction $\Qcal$.
In the first two cases, we can already conclude.
In the third case, there is $\Qcal'\subseteq\Qcal$ such that the (\textsl{strip}) society $(G_{\Qcal'},\Omega_{\Qcal'})$ corresponding to $\Qcal'$ (\autoref{planarstrip}) has a vortex-free rendition in the plane.
In this case, we split $(G,\Omega)$ into two societies $(G_1,\Omega_1)$ and $(G_2,\Omega_2)$ so that they are separated by $(G_{\Qcal'},\Omega_{\Qcal'})$ (\autoref{sec:split}, \autoref{main2}).
For $i\in[2]$, if $k_i$ is the maximum order of a long-jump in $(G_i,\Omega_i)$, then $(G_i,\Omega_i)$ excludes a long-jump of order $k_i+1$.
Given that $(G,\Omega)$ excludes a long-jump of order $k$, we can actually prove that $k_i+1<k$.
We may thus recurse on $(G_1,\Omega_1)$ and $(G_2,\Omega_2)$, excluding a long-jump of order $k_1+1$ and $k_2+1$, respectively.
If one of them contains a big crosscap transaction, we may conclude.
Otherwise, both have a rendition in the \emph{plane} with a bounded number of vortices of small depth, where the number of vortices and the depth depends of $k_i+1$.
We then combine both renditions along with the one of $(G_{\Qcal'},\Omega_{\Qcal'})$ to get a rendition in the \emph{plane} with a bounded number of vortices of small depth (depending on $k$).

\subsection{Almost embeddings}
\label{sec_draw}

We begin by defining renditions, vortices, and all other necessary notions.
  
\paragraph{Drawing a graph in a surface.}
  Let $\Sigma $ be a surface, possibly with boundary.
\label{@partialities}
  A \emph{drawing} (with crossings) in $\Sigma $ is a triple $\Gamma=(U,V,E)$ such that
  \begin{itemize}[itemsep=-2pt] 
  \setlength\itemsep{0em}
  \item $V$ and $E$ are finite, 
  \item $V\subseteq U\subseteq\Sigma ,$ 
  \item $V\cup\bigcup_{e\in E}e=U$ and $V\cap (\bigcup_{e\in E}e)=\emptyset,$ 
  \item for every $e\in E,$ $e=h((0,1)),$ where $h\colon[0,1]_{\mathbb{R}}\to U$ is a homeomorphism onto its image with $h(0),h(1)\in V$
   and
  \item if $e,e'\in E$ are distinct, then $|e\cap e'|$ is finite.
  \end{itemize}
  We call the set $V,$ sometimes denoted by $V(\Gamma),$ the \emph{vertices of $\Gamma$} and the set $E,$ denoted by $E(\Gamma),$ the \emph{edges of $\Gamma$}. We also denote $U(\Gamma )=U.$
  If $G$ is a graph and $\Gamma=(U,V,E)$ is a drawing with crossings in a surface $\Sigma $ such that $V$ and $E$ naturally correspond to $V(G)$ and $E(G)$ respectively, we say that $\Gamma$ is a \emph{drawing of $G$ in $\Sigma $ (possibly with crossings)}. In the case where no two edges in $E(\Gamma)$ have a common point, we say that $\Gamma$ is a \emph{drawing of $G$ in $\Sigma $ without crossings}. In this last case, the connected components of $\Sigma \setminus U$ are the \emph{faces} of $\Gamma.$
  
\begin{figure}[h]
\center
\includegraphics[scale=0.76]{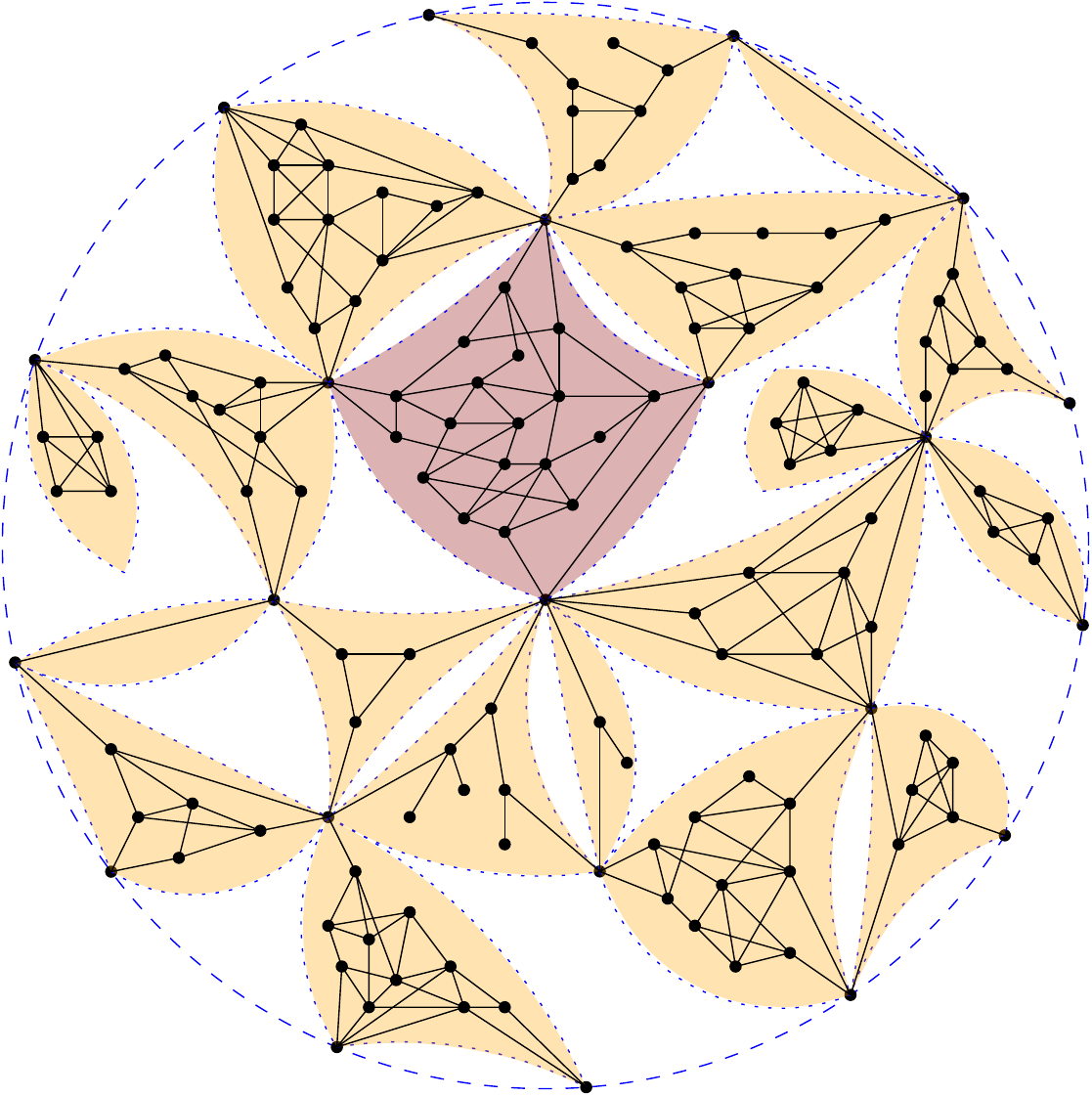}
\caption{A rendition in the plane with one vortex (in red).}
\label{fig_rendition}
\end{figure}
  
\paragraph{$\Sigma $-decompositions.}
  Let $\Sigma $ be a surface. When $\Sigma $ has a boundary, then we denote it by $\bd(\Sigma ).$ Also we refer to $\Sigma \setminus\bd(\Sigma )$ as the \emph{interior} of $\Sigma .$
  A \emph{$\Sigma $-decomposition} of a graph $G$ is a pair $\delta =(\Gamma,\mathcal{D}),$ where $\Gamma$ is a drawing of $G$ in $\Sigma $ with crossings, and $\mathcal{D}$ is a collection of closed disks, each a subset of $\Sigma $ such that 
  \begin{enumerate}[itemsep=-2pt]
  \setlength\itemsep{0em}
  \item the disks in $\mathcal{D}$ have pairwise disjoint interiors, 
  \item the boundary of each disk in $\mathcal{D}$ intersects $\Gamma$ in vertices only, 
  \item if $\delta _1,\delta _2\in\mathcal{D}$ are distinct, then $\delta _1\cap\delta _2\subseteq V(\Gamma),$ and 
  \item every edge of $\Gamma$ belongs to the interior of one of the disks in $\mathcal{D}.$ 
  \end{enumerate} 

\paragraph{$\Sigma $-embeddings.}
A \emph{$\Sigma $-embedding} of a graph $G,$ is a $\Sigma $-decomposition $\delta  = (\Gamma ,\Dcal)$ where  $\Dcal$ is a collection of closed disks such that, for any disk in $\Dcal,$  only a single edge of $\Gamma $ is drawn in its interior.  
For simplicity, we make the convention that, when we refer to a $\Sigma $-embedding, we just refer to the drawing of $\Gamma ,$ as the choice  of $\Dcal$ is obvious in this case.

\paragraph{Nodes, cells, and ground vertices.}
For a $\Sigma $-decomposition $\delta  = (\Gamma , \Dcal)$, 
  let $N$ be the set of all vertices of $\Gamma$ that do not belong to the interior of the disks in $\mathcal{D}.$ 
  We refer to the elements of $N$ as the \emph{nodes} of $\delta .$ 
  If $\delta \in\mathcal{D},$ then we refer to the set $\delta -N$ as a \emph{cell} of $\delta .$ 
  We denote the set of nodes of $\delta $ by $N(\delta )$ and the set of cells by $C(\delta ).$ 
  For a cell $c\in C(\delta )$ the set of nodes that belong to the closure of $c$ is denoted by $\tilde{c}.$ Given a cell $c\in C(\delta ),$ we define 
  its \emph{disk} as 
  $\delta _{c}=\bd(c)\cup c.$
  For a cell $c\in C(\delta )$ we define the graph $\sigma_{\delta }(c),$ or $\sigma(c)$ if $\delta $ is clear from the context, to be the subgraph of $G$ consisting of all vertices and edges drawn in $\delta _{c}.$ 
  We define $\pi_{\delta }\colon N(\delta )\to V(G)$ to be the mapping that assigns to every node in $N(\delta )$ the corresponding vertex of $G.$
We also define \emph{ground vertices} in $\delta $ as $\ground(\delta ) = \pi_{\delta }(N(\delta )).$
 

\paragraph{Vortices.}
  Let $G$ be a graph, $\Sigma $ be a surface, and $\delta =(\Gamma,\mathcal{D})$ be a $\Sigma $-decomposition of $G.$
  A cell $c\in C(\delta )$ is called a \emph{vortex} if $|\tilde{c}|\geq 4.$
  Moreover, we call $\delta $ \emph{vortex-free} if no cell in $C(\delta )$ is a vortex.

\paragraph{Societies.}
  Let $\Omega$ be a cyclic permutation of the elements of some set which we denote by $V(\Omega).$
  A \emph{society} is a pair $(G,\Omega),$ where $G$ is a graph and $\Omega$ is a cyclic permutation with $V(\Omega)\subseteq V(G).$
  
  \paragraph{Crosses.}
  A \emph{cross} in a society $(G,\Omega)$ is a pair $(P_1,P_2)$ of disjoint paths\footnote{When we say two paths are \emph{disjoint} we mean that their vertex sets are disjoint.} in $G$ such that $P_i$ has endpoints $s_i,t_i\in V(\Omega)$ and is otherwise disjoint from $V(\Omega),$ and the vertices $s_1,s_2,t_1,t_2$ occur in $\Omega$ in the order listed.

\paragraph{Renditions.}\label{@cigarettes}
Let $(G,\Omega)$ be a society and let $\delta $ be a closed disk in $\Sigma $.
A \emph{rendition} in $\Sigma $ of $G$ is a $\Sigma $-decomposition $ρ$ of $G$ such that $\pi_{\rho}(N(\rho)\cap \bd(\delta ))=V(\Omega)$, mapping one of the two cyclic orders (clockwise or counterclockwise) of $\bd(\delta )$ to the order of $\Omega$.
See \autoref{fig_rendition} for an illustration.

\paragraph{Rural societies.} A society is \emph{rural} if it has a vortex-free rendition.

\begin{proposition}[\!\!\cite{kawarabayashi2018anewp,RobertsonS90disj}]\label{thm:crossreduct}
A society $(G, Ω)$ in the disk has no cross if and only if it is rural.
\end{proposition}

\paragraph{Grounded graphs.}
Let $\delta $ be a $\Sigma $-decomposition of a graph $G$ in a surface $\Sigma .$ 
Let $Q \subseteq G$ be either a cycle or a path that uses no edge of $\sigma(c)$ for every vortex $c \in C(\delta ).$ 
We say that $Q$ is \emph{grounded in $\delta $} if either $Q$ is a non-zero length path with both endpoints in $\pi_{\delta }(N(\delta )),$ or $Q$ is a cycle, and it uses edges of $\sigma(c_{1})$ and $\sigma(c_{2})$ for two distinct cells $c_{1}, c_{2} \in C(\delta ).$
A $2$-connected subgraph $H$ of $G$ is said to be \emph{grounded in $\delta $} if every cycle in $H$ is grounded in $\delta .$

\paragraph{Tracks.}
Let $\delta $ be a $\Sigma $-decomposition of a graph $G$ in a surface $\Sigma .$ 
For every cell $c \in C(\delta )$ with $|\tilde{c}| = 2$ we select one of the components of $\bd(c) - \tilde{c}.$  
This selection is called a \emph{tie-breaker} in $\delta ,$ and we assume every $\Sigma $-decomposition to come equipped with a tie-breaker. 
If $Q$ is grounded in $\delta $, we define the \emph{track} of $Q$ as follows.
Let $P_{1}, \dots, P_{k}$ be distinct maximal subpaths of $Q$ such that $P_{i}$ is a subgraph of $\sigma(c)$ for some cell $c.$ Fix an index $i.$
The maximality of $P_{i}$ implies that its endpoints are $\pi_\delta (n_{1})$ and $\pi_\delta (n_{2})$ for distinct $\delta $-nodes $n_{1}, n_{2} \in N(\delta ).$
If $|\tilde{c}| = 2,$ define $L_{i}$ to be the component of $\bd(c) - \{ n_{1}, n_{2} \}$ selected by the tie-breaker, and if $|\tilde{c}| = 3,$ define $L_{i}$ to be the component of $\bd(c) - \{ n_{1}, n_{2} \}$ that is disjoint from $\tilde{c}.$
Finally, we define $L'_{i}$ by slightly pushing $L_{i}$ to make it disjoint from all cells in $C(\delta ).$ We define such a curve $L'_{i}$ for all $i$ while ensuring that the curves intersect only at a common endpoint.
The \emph{track} of $Q$ is defined to be $\bigcup_{i \in [k]} L'_{i}.$
So the track of a cycle is the homeomorphic image of the unit circle, and the track of a path is an arc in $\Sigma $ with both endpoints in $N(\delta ).$ 

\paragraph{$\delta $-aligned disks.}
We say  a closed disk $\delta $ in $\Sigma $ is \emph{$\delta $-aligned} if
 its boundary intersects $\Gamma $ only in nodes of $\delta $.
 We denote by $Ω_{\delta }$ one of the cyclic orderings of the vertices on the boundary of $\delta .$
We define the \emph{inner graph} of a $\delta $-aligned closed disk $\delta $ as 
$$\inG_{\delta }(\delta ) \coloneqq \bigcup_{\textrm{$c \in C(\delta )$ and $c \subseteq \delta $}} \sigma(c)$$ 
and the \emph{outer graph of $\delta $} as 
$$\outG_{\delta }(\delta ) \coloneqq \bigcup_{\textrm{$c \in C(\delta )$ and $c \cap \delta  \subseteq \ground(\delta )$}} \sigma(c).$$ 
  
If $\delta $ is $\delta $-aligned, we define $\Gamma  \cap \delta $ to be the drawing of $\inG_{\delta }(\delta )$ in $\delta $ which is the restriction of $\Gamma $ in $\delta .$
If moreover $|Ω_{\delta }|≥4,$ we denote by $\delta [\delta ]$ the rendition $(\Gamma  \cap \delta , \{ \delta _{c} \in \mathcal{D} \mid c \subseteq \delta \})$ of $(\inG_{\delta }(\delta ), Ω_{\delta })$ in $\delta .$

Let $\delta  = (\Gamma,\mathcal{D})$ be a $\Sigma $-decomposition of a graph $G$ in a surface $\Sigma .$
Let $C$ be a cycle in $G$ that is grounded in $\delta ,$ such that the track $T$ of $C$ bounds a closed disk $\delta _{C}$ in $\Sigma .$
We define the \emph{outer} (resp. \emph{inner}) \emph{graph} of $C$ in $\delta $ as the graph $\outG_{\delta }(C) \coloneqq  \outG_{\delta }(\delta _C)$ (resp. $\inG_{\delta }(C) \coloneqq  \inG_{\delta }(\delta _C)$).

\paragraph{Paths.}
If~$P$ is a path and~$x$ and~$y$ are vertices on~$P,$ we denote by~${xPy}$ the subpath of~$P$ with endpoints~$x$ and~$y.$
Moreover, if~$s$ and~$t$ are the endpoints of~$P,$ and we order the vertices of~$P$ by traversing~$P$ from~$s$ to~$t,$ then~${xP}$ denotes the path~${xPt}$ and~${Px}$ denotes the path~${sPx}.$
Let~$P$ be a path from~$s$ to~$t$ and~$Q$ be a path from~$q$ to~$p.$
If~$x$ is a vertex in~${V(P) \cap V(Q)}$ such that~$Px$ and~$xQ$ intersect only in $x$, then~${PxQ}$ is the path obtained from the union of~$Px$ and~$xQ.$
Let~${X,Y \subseteq V(G)}.$
A path is an \emph{$X$-$Y$-path} if it has one endpoint in $X$ and the other in $Y$ and is internally disjoint from~${X \cup Y},$
Whenever we consider~$X$-$Y$-paths we implicitly assume them to be ordered starting in $X$ and ending in $Y,$ except if stated otherwise.
An \emph{$X$-path} is an $X$\nobreakdash-$X$\nobreakdash-path of length at least one.
In a society~$(G,\Omega),$ we write~$\Omega$-path as a shorthand for a~$V(\Omega)$-path.

\paragraph{Linkages.}\label{@mediterranean}   Let~$G$ be a graph.
  A \emph{linkage} in~$G$ is a set of pairwise vertex-disjoint paths.
  In slight abuse of notation, if~$\mathcal{L}$ is a linkage, we use~$V(\mathcal{L})$ and~$E(\mathcal{L})$ to denote~${\bigcup_{L\in\mathcal{L}}V(L)}$ and~${\bigcup_{L\in\mathcal{L}}E(L)}$ respectively.
  Given two sets~$A$ and~$B$, we say that a linkage~$\mathcal{L}$ is an \emph{$A$-$B$-linkage} if every path in~$\mathcal{L}$ has one endpoint in~$A$ and one endpoint in~$B.$
  We call $|\mathcal{L}|$ the \emph{size} of $\mathcal{L}.$

\paragraph{Segments.}
Let~$(G,\Omega)$ be a society.
A \emph{segment} of~$\Omega$ is a set~${S \subseteq V(\Omega)}$ such that there do not exist~${s_1,s_2 \in S}$ and~${t_1,t_2 \in V(\Omega) \setminus S}$ such that~${s_1,t_1,s_2,t_2}$ occur in~$\Omega$ in the order listed. 
A vertex~${s \in S}$ is an \emph{endpoint} of the segment~$S$ if there is a vertex~${t \in V(\Omega) \setminus S}$ which immediately precedes or immediately succeeds~$s$ in the order~$\Omega.$
For vertices~${s,t\in V(\Omega)},$ if~$t$ immediately precedes~$s$,
 we define~$s\Omega t$ to be the \emph{trivial segment}~$V(\Omega),$
and otherwise we define~$s\Omega t$ to be the uniquely determined segment with first vertex~$s$ and last vertex~$t.$
  
  \paragraph{Transactions.}
  Let~${(G,\Omega)}$ be a society. 
  A \emph{transaction} in~${(G,\Omega)}$ is an $A$-$B$-linkage for disjoint segments~$A,B$ of~$\Omega.$ 
  We define the \emph{depth} of~${(G,\Omega)}$ as the maximum order of a transaction in~${(G,\Omega)}.$

Let $\mathcal{T}$ be a transaction in a society $(G,\Omega).$ 
We say that $\mathcal{T}$ is \emph{planar} if no two members of $\mathcal{T}$ form a cross in $(G,\Omega).$ 
An element $P\in\mathcal{T}$ is \emph{peripheral} if there exists a segment $X$ of $\Omega$ containing both endpoints of $P$ and no endpoint of another path in $\mathcal{T}.$ 
A transaction is \emph{crooked} if it has no peripheral element.

Let $\Tcal = \{P_1,...,P_t\}$ be a transaction between segments $A$ and $B$.
For $i\in[t]$, let $a_i$ (resp. $b_i$) be the endpoint of $P_i$ in $A$ (resp. $B$).
Up to a permutation, we may assume that $a_1,...,a_t$ occur in this order in $\Omega$.
We say that $\Tcal$ is a \emph{$t$-crosscap} transaction if $b_1,...,b_t$ occur in this order in $\Omega$.
We say that $\Tcal$ is a \emph{crosscap} transaction if it is a $t$-crosscap transaction for any $t\in\Nbbb$.
A transaction is called \emph{monotone} if it is either a planar or a crosscap transaction.
Suppose $t\ge1$.
We say that $\Tcal$ is a \emph{$(t-1)$-long-jump} transaction if $b_1,b_t,b_{t-1},\dots,b_3,b_2$ occur in this order in $\Omega$. (In other words, $P_1$ crosses the $t-1$ paths of the planar transaction $\{P_2,...,P_t\}$; such transactions 
have been introduced in \cite{RobertsonS90disj}
under the name \textsl{leap transactions}.)

\begin{proposition}[\!\!\cite{ErdosS35acom}]\label{planar or crosscap}
Let $r,s\in\Nbbb$.
Let $(G,\Omega)$ be a society.
Let $\Qcal$ be a transaction in $(G,\Omega)$ of size $(r-1)(s-1)+1$.
Then $\Qcal$ contains either a planar transaction $\Qcal'\subseteq \Qcal$ of size $r$, or a crosscap transaction $\Qcal'\subseteq \Qcal$ of size $s.$
\end{proposition}

\paragraph{Vortex societies.}

Let~$\Sigma $ be a surface and~$G$ be a graph.
Let~${\delta  = (\Gamma,\mathcal{D})}$ be a $\Sigma $-decomposition of~$G.$
Every vortex~$c$ defines a society~${(\sigma(c),\Omega)},$ called the \emph{vortex society} of~$c,$ by saying that~$\Omega$ consists of the vertices in $\pi_{\delta }(\tilde{c})$   in the order given by~$\Gamma$ (there are two possible choices of~$\Omega,$ namely~$\Omega$ and its reversal.
Either choice gives a valid vortex society).
The \emph{breadth} of $ \delta$ is the number of cells $c\in C( \delta)$ which are a vortex and the \emph{depth} of $ \Delta$ is the maximum depth of the vortex societies $( \sigma(c),\Omega)$ over all vortex cells $c\in C( \Delta).$

\paragraph{Cylindrical renditions.}

Let $(G, Ω)$ be a society, $ρ = (\Gamma , \mathcal{D})$ be a rendition of $(G, Ω)$ in a disk, and let $c_{0} \in C(ρ)$ be such that no cell in $C(ρ) \setminus \{ c_{0} \}$ is a vortex. We say that the triple $(\Gamma , \mathcal{D}, c_{0})$ is a \emph{cylindrical rendition} of $(G, Ω)$ around $c_{0}.$

\begin{proposition}[Lemma 3.6, \cite{kawarabayashi2020quickly}, see also \cite{Robertson2003GMXVI}]\label{lem:GM9} 
Let $(G,\Omega)$ be a society and $p\geq 4$ be a positive integer.
Then $(G,\Omega)$ has a crooked transaction of size $p$, or a cylindrical rendition of depth at most $6p$.
\end{proposition}

\paragraph{Nests and railed nests.} \label{@fingerprint}

Let $\delta  = (\Gamma,\mathcal{D})$ be a $\Sigma $-decomposition of a graph $G$ in a surface $\Sigma $ and let $\delta  \subseteq \Sigma $ be an arcwise connected set.
A \emph{nest in $\delta $ around $\delta $ of order $s$} is a sequence $\mathcal{C}= \langle C_1,C_2,\dots,C_s \rangle$ of disjoint cycles in $G$ such that each of them is grounded in $\delta $ and the track of $C_i$ bounds a closed disk $\delta _{C_{i}}$ in such a way that $\delta  \subseteq \delta _{C_{1}} \subsetneq \delta _{C_{2}} \subsetneq \dots \subsetneq \delta _{C_{s}} \subseteq\Sigma .$ We call $C_{1}$ (resp. $C_{s}$) the \emph{internal} (resp. \emph{external}) cycle of $\Ccal.$
We call the sequence $\langle \delta _{C_{1}}, \delta _{C_{2}}, \dots,  \delta _{C_{s}}\rangle$ the \emph{disk sequence} of the nest in $\delta $ around $\delta $.
If $\delta = (\Gamma, \Dcal, c_{0})$ is a cylindrical rendition, then we say that $\Ccal$ is a nest in $\rho$ around $c_{0}$.

Moreover, let $A = V(C_{1})\cap \pi_{\delta }(N(\delta )),$ $B = V(C_{s})\cap \pi_{\delta }(N(\delta )),$ and assume that $\{P_{1},\ldots,P_{r}\}$ is an $A\mbox{-}B$-linkage such that for every $(i,j) \in [s] \times [r]$ the graph $C_{i}\cap P_{j}$ is a (possibly edgeless) path.
We call the pair $(\Ccal, \Pcal)$ a \emph{railed nest in $\delta $ around $\delta $ of order $(s, r)$}.
Notice that $\cupall \Pcal$ is disjoint from $\inG_{\delta }(C_{1}) - V(C_{1})$ and $\outG_{\delta }(C_{s}) - V(C_{s}).$

\paragraph{Orthogonal and exposed transactions.}
Let $\rho=(\Gamma , \mathcal{D}, c_{0})$ be a cylindrical rendition in a society $(G,\Omega)$.
Let $\Ccal$ be a nest in $\rho$ around $c_0$ and $\Qcal$ be a transaction in $(G,\Omega)$.
We say that $\Qcal$ is \emph{orthogonal} to $\Ccal$ if, for each $C\in\Ccal$ and each $Q\in\Qcal$, the graph $C\cap Q$ has at most two components. 
We say that $\Qcal$ is \emph{exposed}\footnote{Please note that what we call ``\textsl{exposed}'' here has been dubbed as ``\textsl{unexposed}'' in the work of Kawarabayashi et al.\@ \cite{kawarabayashi2020quickly}. Following the convention of \cite{gorsky2025polynomial} we adapt this slight change as we believe it to be more intuitive.} in $\rho$ if each $Q\in\Qcal$ has at least one edge in $\sigma(c_0)$.
Note that any crooked transaction is necessarily exposed.
If $\Qcal$ is orthogonal and exposed, then every element of $\Qcal$ contains exactly two disjoint minimal subpaths which each have one endpoint in $V(\Omega)$ and the other in $V(\sigma(c_0))$.
Let $\Pcal$ be the union of all such minimal subpaths over the elements of $\Qcal$.
$\Pcal$ is called the \emph{rail truncation} of $\Qcal$.
Note that $(\Ccal,\Pcal)$ is a railed nest.

\paragraph{Coterminal transactions.}
Let $\rho=(\Gamma , \mathcal{D}, c_{0})$ be a cylindrical rendition in a society $(G,\Omega)$.
Let $(\Ccal=(C_1,\dots,C_s),\Pcal)$ be a railed nest in $\rho$ around $c_0$ and $\Qcal$ be a transaction in $(G,\Omega)$.
We say that $\Qcal$ is \emph{coterminal} with $\Pcal$ up to level $C_i$ if there exists a subset $\Pcal'\subseteq \Pcal$ such that $\outG_\rho(C_i)\cap\Qcal=\outG_\rho(C_i)\cap\Pcal$.
When it is clear from the context which nest we are referring to, we say that $\Qcal$ is \emph{coterminal} with $\Pcal$ up to level $i$.

\begin{proposition}[Lemma 4.5,\cite{kawarabayashi2020quickly}]\label{lem_crookedrooted}
Let $r, s$ be positive integers with
$s \ge 2r+7$.
Let $(G, \Omega)$ be a society and $\rho=(\Gamma,\Dcal,c_0)$ be a cylindrical rendition of $(G, \Omega)$.
 Let $((C_1, C_2, \dots, C_s),\Pcal)$ be a railed nest of order $(s,4r+6)$ in $\rho$ around $c_0$.
If there exists a crooked transaction of  size at least $r$ in $(G, \Omega)$, then there exists a crooked transaction
 of size at least $r$
in $(G, \Omega)$ that is coterminal with $\Pcal$ up to level $C_{2r+7}$.
\end{proposition}

\subsection{Recognizing a long-jump}
\label{sec:find_jump}

In this part, we define a few more parametric graphs (resp. types of transactions), and prove that they all contain a long-jump grid as a minor (resp. a long-jump transaction).

\medskip
We consider here five new parametric graphs defined by adding edges in the \emph{double-parameterized annulus grid} $\Gamma = \langle \Gamma_{k,r}\rangle$, where $\Gamma_{k,r}$ is the $(k\times r)$-cylindrical grid, as indicated in \autoref{figfindlongjump}.
\begin{figure}[h!]
\begin{center}
\includegraphics[scale=1]{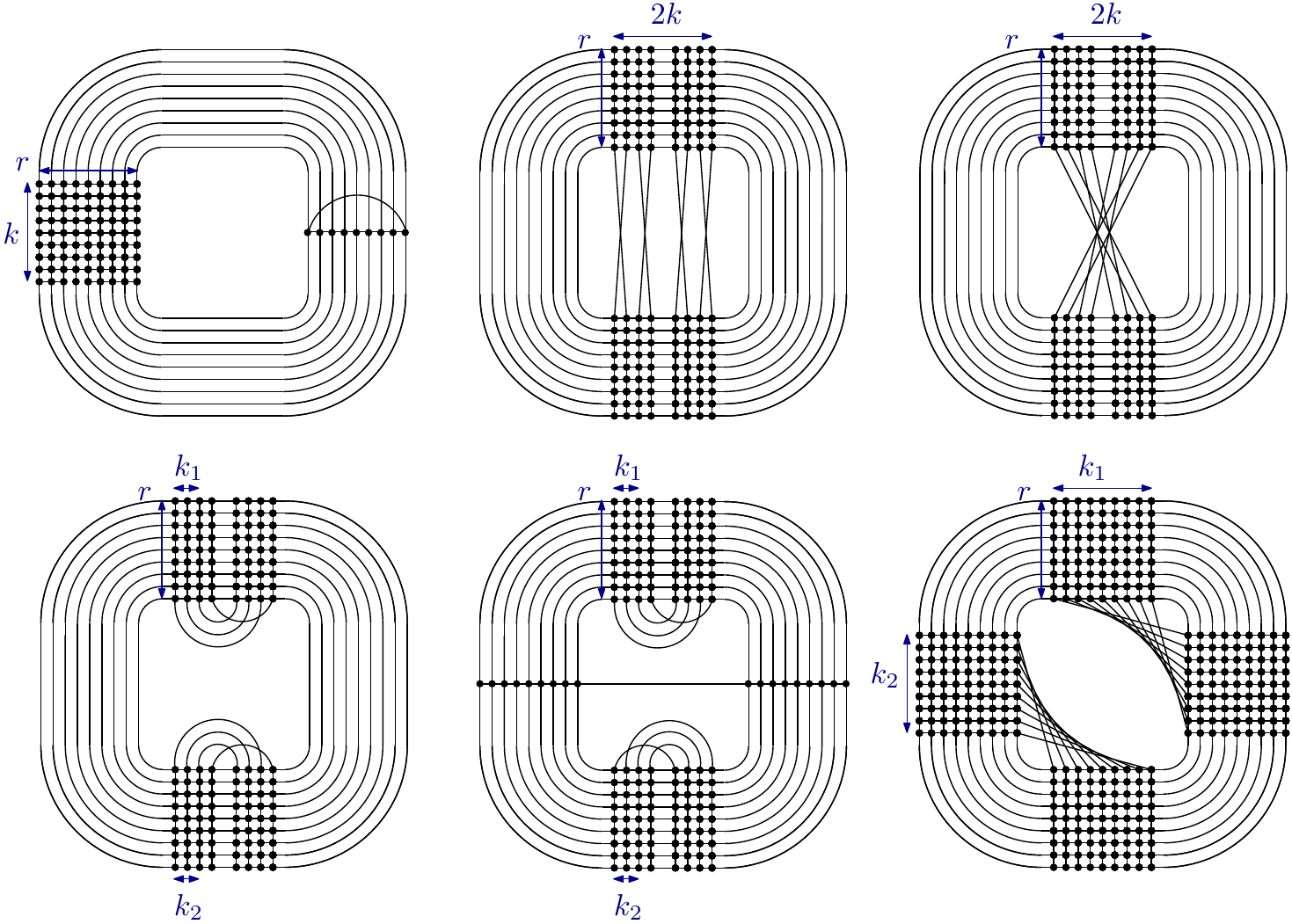}
\end{center}
  \caption{From up to down and left to right: the alternative jump $\hat{\mathscr{J}}_k^r$, the nested-crosses grid $\mathscr{NC}_k^r$, the twisted-crosses grid $\mathscr{TC}_k^r$, the double-jump grid $\mathscr{P}_{k_1,k_2}^r$, the alternative double-jump grid $\mathscr{Q}_{k_1,k_2}^r$, and the klein grid $\mathscr{K}_{k_1,k_2}^r$.}
\label{figfindlongjump}
\end{figure}
These are the the alternative jump grid $\hat{\mathscr{J}}_k^r$,the nested-crosses grid $\mathscr{NC}_k^r$, the twisted-crosses grid $\mathscr{TC}_k^r$, the double-jump grid $\mathscr{P}_{k_1,k_2}^r$, the alternative double-jump grid $\mathscr{Q}_{k_1,k_2}^r$, and the klein grid $\mathscr{K}_{k_1,k_2}^r$.

We also define the transactions corresponding to the grids $\mathscr{NC}_k^r$, $\mathscr{TC}_k^r$, $\mathscr{P}_{k_1,k_2}^r$, $\mathscr{Q}_{k_1,k_2}^r$, and $\mathscr{K}_{k_1,k_2}^r$, respectively.
Let $(G,\Omega)$ be a society and $k,k_1,k_2\in\Nbbb$.
Let $A$ and $B$ be two segments in $(G,\Omega)$ and $\Pcal=\{P_1,\dots,P_t\}$ be a transaction in $(G,\Omega)$ between $A$ and $B$ such that the endpoints of $P_i$ are $a_i$ and $b_i$ with $a_1, \dots, a_t$ occurring in $A$ in this order.
\begin{itemize}
\item $\Pcal$ is a \emph{$k$-nested crosses transaction} if $t=2k$ and 
for all $1 \le i < j \le 2k$, $P_i$ and $P_j$ cross if and only if $i$ is odd and $j=i+1$.
\item $\Pcal$ is a \emph{twisted $k$-nested crosses transaction} if $t=2k$ and
for all $1 \le i < j \le 2k$, $P_i$ and $P_j$ do not cross if and only if $i$ is odd and $j=i+1$.
In other words, $b_2,b_1,b_4,b_3,\dots,b_{2k-2},b_{2k-3},b_{2k},b_{2k-1}$ occur in $\Omega$ in this order.
\item $\Pcal$ is a \emph{$(k_1,k_2)$-double-jump transaction} if $t=k_1+k_2+2$ and 
$\Pcal$ can be partitioned into two planar transactions $\Qcal_1$ and $\Qcal_2$ of size $k_1$ and $k_2$ respectively and two isolated paths $Q_1$ and $Q_2$ such that $\Qcal_i'\coloneqq \Qcal_i\cup Q_i$ is a $k_i$-long-jump transaction for $i\in[2]$, $\Qcal_1'$ and $\Qcal_2'$ do not cross, and there is an endpoint $q_1$ of $Q_1$ and an endpoint $q_2$ of $Q_2$ such that $\Pcal\setminus\{Q_1,Q_2\}$ is a either a  $q_1\Omega q_2$-linkage or a $q_2\Omega q_1$-linkage.
In other words, no path of $\Pcal\setminus\{Q_1,Q_2\}$ has an endpoint in one of $q_1\Omega q_2$ and $q_2\Omega q_1$.
\item $\Pcal$ is an \emph{ alternative $(k_1,k_2)$-double-jump transaction} if $t=k_1+k_2+3$ and 
$\Pcal$ can be partitioned into two planar transactions $\Qcal_1$ and $\Qcal_2$ of size $k_1$ and $k_2$ respectively and three isolated paths $Q$, $Q_1$, and $Q_2$ such that $\Qcal_i'\coloneqq \Qcal_i\cup Q_i$ is a $k_i$-long-jump transaction for $i\in[2]$, for any endpoint $q_1$ of $Q_1$ and endpoint $q_2$ of $Q_2$, $\Pcal\setminus\{Q_1,Q_2\}$ is a neither a  $q_1\Omega q_2$-linkage nor a $q_2\Omega q_1$-linkage, and 
if $a$ and $b$ are the endpoints of $Q$, then one of $\Qcal_1$ and $\Qcal_2$ is a $a\Omega b$-linkage, and the other is a $b\Omega a$-linkage.
\item Finally, $\Pcal$ is an \emph{$(k_1,k_2)$-klein transaction} if $t=k_1+k_2$ and 
$\Pcal$ can be partitioned into two crosscap transactions $\Qcal_1$ and $\Qcal_2$ of size $k_1$ and $k_2$ respectively that do not cross.
\end{itemize}

\begin{lemma}\label{allinone2}
Let $k,k_1,k_2,r\in\Nbbb$ such that $k_1+k_2\ge k$.
Then $\mathscr{J}_k^r$ is a minor of $\hat{\mathscr{J}}_{2r}^{2r+k-2}$, $\mathscr{NC}_k^{r+1}$, $\mathscr{TC}_{k+2r}^{r+1}$, $\mathscr{P}_{k_1,k_2}^{r+1}$, $\mathscr{Q}_{k_1,k_2}^{r+1}$, and $\mathscr{K}_{k,k+1}^{r+k}$.
\end{lemma}
\begin{proof}
For this proof, an illustration is more convenient than a long description.
See \autoref{figlongjumpingrids}.
\begin{figure}[h!]
\begin{center}
\includegraphics[scale=1.05]{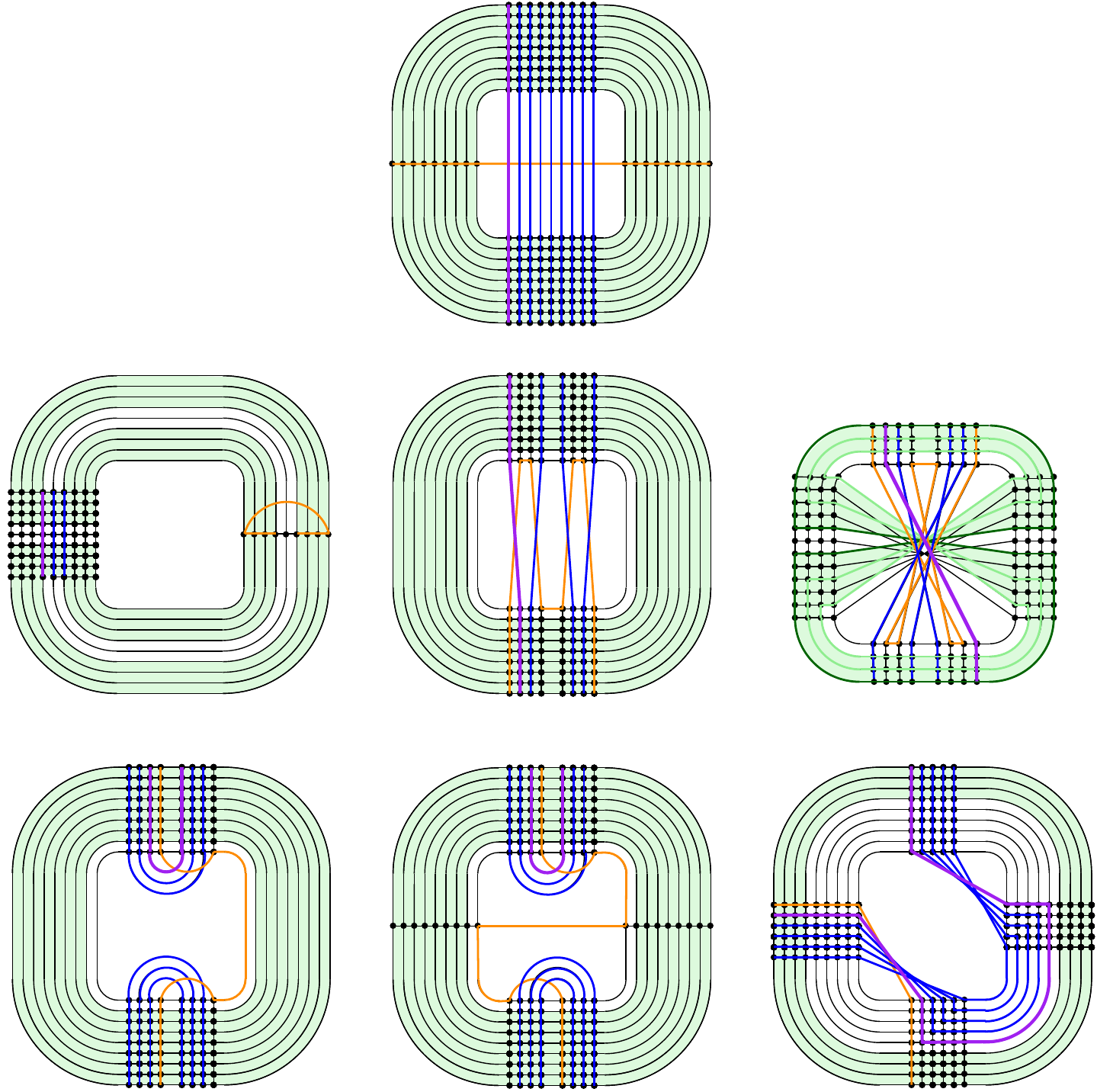}
\end{center}
\caption{Proof of \autoref{allinone2}.}
\label{figlongjumpingrids}
\end{figure}
In each grid, we need to find a path (in orange in \autoref{figlongjumpingrids}), jumping over a parallel linkage (in blue).
The annulus grid is represented in green.
\end{proof}

\begin{lemma}\label{allinone}
Let $k,k_1,k_2,r\in\Nbbb$ such that $k_1+k_2\ge k$.
Let $(G,\Omega)$ be a society and $\rho = (\Gamma, \Dcal, c_0)$ be a cylindrical rendition of $(G,\Omega)$ with a railed nest $(\Ccal, \Pcal)$ in $\rho$ around $c_0$, where $\Ccal =(C_1,\dots,C_s)$ is of order $s \ge a$ and $a = 2$ for items (i)-(iv) and $a = k+1$ for item (v).
Let $\Qcal$ be a transaction in $(G, \Omega)$ orthogonal to $\Pcal$ that is either:
\begin{itemize}
\item[(i)] a $k$-nested crosses transaction, or
\item[(ii)] a twisted $3k$-nested transaction, or
\item[(iii)] a $(k_1,k_2)$-double-jump transaction, or
\item[(iv)] an alternative $(k_1,k_2)$-double-jump transaction, or
\item[(v)] a $(k,k+1)$-klein transaction.
\end{itemize}
Then there is a $k$-long-jump transaction that is coterminal with $\Pcal$ up to level $a$.
\end{lemma}
\begin{proof}
Let $H$ be the graph obtained from the union of $\Ccal$, $\Pcal$, $\Qcal$, and a cycle $C$ with vertex set $V(\Omega)$ and an edge $uv$ between any two consecutive vertices $u, v$ in the cyclic ordering of $\Omega$ (note that these edges may not be edges in $G$).
If $\Qcal$ is one of items (i) to (v), then $H$ is one of $\mathscr{NC}_k^{s+1}$, $\mathscr{TC}_k^{s+1}$, $\mathscr{P}_{k_1,k_2}^{s+1}$, $\mathscr{Q}_{k_1,k_2}^{s+1}$, and $\mathscr{K}_{k_1,k_2}^{s+1}$.
Then, by \autoref{allinone2}, $\mathscr{J}_k^{s-a+2}$ is a minor of $H$.
Let $X$ be the set of vertices in the cycles $C$ and $C_{a},\dots,C_s$.
More particularly, if we look at \autoref{figlongjumpingrids}, $\mathscr{J}_k^{s-a+2}$ is an $X$-minor of $H$.
Hence, we deduce that there is $\Qcal'\subseteq \Qcal$ that is a $k$-long-jump transaction that is coterminal with $\Pcal$ up to level $C_a$.
\end{proof}

\subsection{Finding an isolated and rural strip in a vortex}
\label{sec:strip}

In this part, we prove that a society $(G,\Omega)$ containing a planar transaction $\Qcal$ either contains a long-jump transaction, or is such that the society (called \emph{strip society}) bordered by some strip $\Qcal'\subseteq \Qcal$ has a vortex-free rendition.

We first define strip societies (from \cite{kawarabayashi2020quickly}).

\paragraph{Strip societies.}
Let $\Pcal=(P_1,...,P_m)$ be a monotone transaction of size $m\ge2$ in a society $(G, \Omega)$.
Let $X_1$ and $X_2$ be disjoint segments of $\Omega$ such that $\Pcal$ is a linkage from $X_1$ to $X_2$.
For each $i\in[m]$, the endpoints of $P_i$ are denoted by
$a_i$ and $b_i$ in such a way that $a_i\in X_1$, $b_i\in X_2$, and 
the endpoints occur in $\Omega$ in the order 
$a_1,a_2,\ldots,a_m,b_m,b_{m-1},\ldots,b_1$ if $\Pcal$ is a planar transaction,
and in the order $a_1,a_2,\ldots,a_m,b_1,b_{2},\ldots,b_m$ if $\Pcal$ is a crosscap transaction.

Let $H$ denote the subgraph of $G$ obtained from the union of the elements of $\Pcal$ by adding the elements of  $V(\Omega)$
as isolated vertices.
Let $H'$ be the subgraph of $H$ consisting of $\Pcal$, $a_1\Omega a_m$, and $b_1\Omega b_m$.
Let us consider all $H$-bridges of $G$ with at least one attachment in $V(H')\setminus V(P_1\cup P_{n})$, and for each such
$H$-bridge $B$, let $B'$ denote the graph obtained from $B$ by deleting all attachments that do not belong to  $V(H')$.
Finally, let $G_1$ denote the union of $H'$ and all the graphs $B'$ as above.

If $\Pcal$ is a planar transaction,
let the cyclic permutation $\Omega_1$ be defined by saying that
$V(\Omega_1)=a_{1}\Omega a_m\cup b_m\Omega b_{1}$
and that the order on $\Omega_1$ is the one induced by $\Omega$.
If $\Pcal$ is a crosscap transaction,
let the cyclic permutation $\Omega_1$ be defined by saying that
$V(\Omega_1)=a_{1}\Omega a_m\cup b_1\Omega b_m$
and that the order on $\Omega_1$ is obtained by following $a_1 Ωa_m$ in the order
given by $Ω$, and then following $b_1 Ωb_m$ in the \emph{reverse order} from the one given by $Ω$.

Thus $(G_1,\Omega_1)$ is a society, and we call it the {\em $\Pcal$-strip society of $(G, \Omega)$} with respect to $(X_1,X_2)$.
When there can be no confusion as to the choice of the segments $(X_1, X_2)$, we will omit specifying them.

We say that $P_1$ and $P_m$  are the {\em boundary paths} of the $\Pcal$-strip society  $(G_1,\Omega_1)$.
We say that the $\Pcal$-strip society of $(G, \Omega)$ is {\em isolated} in $G$ if no edge of $G$ has one
endpoint in $V(G_1)\setminus V(P_1 \cup P_m)$ and the other endpoint in $V(G)\setminus V(G_1)$.
Thus $(G_1,\Omega_1)$ is isolated if and only if every $H$-bridge of $G$ with at least one attachment in $V(H')\setminus V(P_1\cup P_m)$
has all its attachments in $V(H')$.

\medskip
The following lemma, that is very much inspired from \cite[Theorem 5.11]{kawarabayashi2020quickly}, says that if a society has a big monotone transaction $\Qcal$, then either it contains a long-jump transaction, or the $\Qcal'$-strip society defined by some $\Qcal' \subseteq \Qcal$ is isolated and rural.

\begin{lemma}\label{planarstrip}
Let $l,k,s\in\Nbbb$ with $s\ge 9$.
Let $(G, \Omega)$ be a society.
Let $\rho = (\Gamma, \Dcal, c_0)$ be a cylindrical rendition of $(G, \Omega)$.
Let $\Qcal$ be an exposed planar (resp. crosscap) transaction in $(G, \Omega)$ of order $l'\ge k(l+3k)$ (resp.  $l'\ge 3k(l+3k)$).
Let $(\Ccal = (C_1, \dots, C_{s}),\Pcal)$ be a railed nest of order $(s,2l')$ in $\rho$ around $c_0$
where $\Pcal$ is the rail truncation of $\Qcal$.
Let $X_1,X_2$ be disjoint segments of $\Omega$ such that $\Qcal$ is a linkage from $X_1$ to $X_2$.
Then one of the following holds.
\begin{itemize}
\item[(i)] there is a $k$-long-jump transaction in $(G,\Omega)$ that is coterminal with $\Pcal$ up to level $C_9$, or
\item[(ii)] there is a transaction $\Qcal' \subseteq \Qcal$ of size at least $l$
such that the $\Qcal'$-strip society of $(G, \Omega)$ with respect to $(X_1,X_2)$ is isolated and rural.
\end{itemize}
\end{lemma}
\begin{proof}
Assume the claim is false.  
Let $k'=k$ if $\Qcal$ is a planar transaction, and $k'=3k$ if $\Qcal$ is a crosscap transaction.
Let $X^1 \subseteq X_1$ and $X^2\subseteq X_2$ be the two disjoint minimal segments of $\Omega$ such  that every path
in $\Qcal$ has one endpoint in $X^1$ and the other point in $X^2$.
Let $\Qcal=\{Q_1,Q_2,\ldots,Q_{l'}\}$, where the paths are numbered in such a way that
their endpoints appear in $X_1$ (and therefore also in $X_2$) in order. 

Let $I_1',...,I_{k'}'$ be intervals of length $l+3k$ with union $[l']$.
For $i\in[k']$, let $I_i$ be obtained from $I_i'$ by deleting the last $k$ elements.
Thus $|I_i|=l+2k$.

For $i\in[k']$, let $\Qcal_i$ be the set $\{Q_j\mid j \in I_i\}$.
For $i\in[k']$, and $j\in[2]$, let $X^j_i$ be the minimum subset of $X^j$ forming a segment of $\Omega$ containing an endpoint of every element of $\Qcal_i$.
Let $(H_i, \Omega_i)$ be the $\Qcal_i$-strip society of $(G, \Omega )$ with respect to $(X^1_i,X_i^2)$.

Let $i \in [k']$.
We define a society $(H_i', \Omega_i)$ which is closely related to the strip society $(H_i, \Omega_i)$ as follows.
Let $J$ be the graph consisting of the union of  $V(\Omega)$ treated as isolated vertices and $\Qcal$.
For every $J$-bridge $B$ in $G$, let $B'$ be the subgraph obtained by deleting all attachments of $B$ not in $V(\Qcal_i) \cup X^1_i \cup X^2_i$.  
Let $\alpha+1$ be the smallest value in $I_i$.  
Let $H_i'$ be the union of $J[V(\Qcal_i) \cup X^1_i \cup X^2_i]$ and $B'$ for every $J$-bridge $B$ in $G$ with at least one attachment in $(V(\Qcal_i) \cup X^1_i \cup X^2_i) \setminus V(Q_{\alpha+1} \cup Q_{\alpha+|I_i|})$ (that is, an attachment not in the first or last element of $\Qcal_i$).
Note that the difference between $H_i$ and $H_i'$ is that, to define $H_i$, we consider $Q_{i} \cup V(\Omega)$-bridges, while for $H_i'$, we consider $Q \cup V(\Omega)$-bridges.

\begin{claim}\label{cl:stripdisjoint}
The subgraphs $H_i'$ are pairwise vertex disjoint.
\end{claim}
\begin{cproof}
Assume the claim is false and let us show that there is a long-jump transaction of order $k$ in $(G,\Omega)$.
Let $i< i'$ be indices such that $H_i'$ intersects $H_{i'}'$.  Thus, there exists a path $R$ in $G$ with one endpoint $x_i$ in $X^1_i \cup X^2_i \cup V(\Qcal_i)$, the other endpoint $x_{i'}$ in $X^1_{i'} \cup X^2_{i'} \cup V(\Qcal_{i'})$ and no internal vertex in $V(\Omega) \cup V(\Qcal)$.

Let $F$ be the subgraph formed by the union of $C_1$ and the inner graph of $C_1$.  We fix the path $R$ to minimize the number of edges not contained in $E(F)$.  Let $\bar{R} = R \cap F$.

Consider a maximal subpath $T$ of $R$ with all internal vertices contained in $V(R) \setminus V(\bar{R})$.  There are two possible cases given the cylindrical rendition and the fact that $R$ is internally disjoint from $\Qcal$: either $T$ has one endpoint in $\{x_i,x_{i'}\}$ and one endpoint in $C_1$ or, alternatively, $T$ has both endpoints in a component of $C_1 - V(\Qcal)$.  By replacing any such subpath $T$ with an appropriately chosen subpath of $C_1$, it follows that there exists a path $R'$ contained in $C_1 \cup \bar{R}$ with endpoints $x_i'$ and $x_{i'}'$ such that
\begin{itemize}
\item $R'$ is internally disjoint from $\Qcal$ and,
\item $x_i'$ is contained in $V(\Qcal_i)$ and $x_{i'}'$ is contained in $V(\Qcal_{i'}')$.
\end{itemize}
Note that if $x_i$ is in an element of $\Qcal$, we can choose $x_i'$ to be in the same element of $\Qcal$ as $x_i$.  If $x_i \in V(\Omega)$, we can choose $x_i'$ to be in either of the two elements of $\Qcal$ closest to $x_i$ on $\Omega$.

We conclude, by the choice of $R'$, that $R'$ is a subgraph of $F$.
Thus, both endpoints $x_i'$ and $x_{i'}'$ are contained in $V(\Qcal) \cap V(F)$.
This path $R'$ along with the $k$ paths $(Q_{\alpha+|I_i|+1},...,Q_{\alpha+|I_i|+k})$ as well as $\Qcal_i$, and $\Qcal_{i'}$, create a $k$-long-jump transaction in $(G,\Omega)$ that is coterminal with $\Pcal$ up to level $C_2$, and hence $C_9$, which is outcome (i). 
See \autoref{figwallsinwall} for an illustration of the proof of \autoref{ca}, which is similar.
\end{cproof}

\begin{claim}
There exists an index $i\in[k']$, such that $(H_i', \Omega_i')$ is rural.
\end{claim}
\begin{cproof}
The proof is exactly the same as the one of \cite[Theorem 5.11, Claim 3]{kawarabayashi2020quickly} with the specificity that for us, $Z = \emptyset$, and thus $\alpha=0$ and $a=5$.
The idea is that if none of the $(H_i', \Omega_i')$ is rural, then each contain a cross.
Thus, we will find a transaction $\Qcal'$ from $X_1$ to $X_2$ that is coterminal with $\Qcal$ up to level $C_{8}$, such that $\Qcal'$ is either a $k'$-nested crosses transaction (if $\Qcal$ is planar) or a twisted $k'$-nested crosses transaction (if $\Qcal$ is a crosscap transaction).
Then, by \autoref{allinone}, there is a $k$-long-jump transaction that is coterminal with $\Pcal$ up to level $C_9$, which is outcome (ii) of the statement.
\end{cproof}

We fix, for the remainder of the proof, an index $i$ such that $(H_i', \Omega_i')$ is rural, and fix a vortex-free rendition $\rho'$ of $(H_i', \Omega_i')$ in the disk $\Delta$.  Let $I^\star$ be the interval obtained from $I_i$ by deleting the first and last $k$ elements.  Note that $|I^\star| = l$.  Let the first element in $I^\star$ be $\beta + 1$.  Let $\Qcal^\star$ be the set $\{Q_j\mid j \in I^\star\}$, and $X^{1*}$ and $X^{2*}$ be the minimal segments contained in $X^1_i$ and $X^2_i$, respectively, containing an endpoint of each element of $\Qcal^\star$.  Let $T_1$ and $T_l$ be the track of $Q_{\beta + 1}$ and $Q_{\beta +l}$ in the rendition $\rho'$.

Of the three connected components of $\Delta \setminus ( T_1 \cup T_l)$, let $\Delta^\star$ be the (unique) component whose boundary contains $T_1$ and $T_l$.  Let $H^\star = \bigcup_{c \in C(\rho'): c \subseteq \Delta^\star} \sigma(c) \cup Q_{\beta + 1} \cup Q_{\beta + l}$.  Let $\Omega^\star$ be the cyclically ordered set with $V(\Omega^\star) = X^{1*} \cup X^{2*}$ obtained by restricting $\Omega_i'$ to $V(\Omega^\star)$.

The rendition $\rho'$ restricted to the disk $\Delta^\star$ can be extended to a vortex-free rendition
of $(H^\star, \Omega^\star)$ by mapping the vertices of $Q_{\beta+1}$ and $Q_{\beta+l}$ to the boundary, and thus the following claim immediately follows.

\begin{claim}
$(H^\star, \Omega^\star)$ is rural.
\end{claim}

We now see that no edge to $V(H^\star)$ avoids the paths $Q_{\beta+1}$ and $Q_{\beta+l}$.  Recall that the definition of the subgraph $J$ of $G$ is the subgraph consisting of the union of  $V(\Omega)$ treated as isolated vertices and $\Qcal$.

\begin{claim}\label{cl:nojumpedge}
There does not exist an edge $xy$ of $G$ with $x$ in $V(H^\star) \setminus (V(Q_{\beta+1}) \cup V(Q_{l}))$ and $y$ in $V(G) \setminus V(H^\star)$.
\end{claim}

\begin{cproof}
Assume that such an edge $xy$ exists.  Given the rendition $\rho'$, it follows that the edge $xy \notin E(H_i')$, and thus the vertex $y$ is the attachment of a $J$-bridge $B$ in $G$ such that $y \in (V(\Omega) \setminus V(H_i')) \cup \bigcup_{j \notin I_i} V(Q_j)$ and $B$ has an attachment $x'$ in $V(\Omega^\star) \cup V(\Qcal^\star)$.  Fix a path $P$ in $B$ from $x'$ to $y$ which is internally disjoint from $V(\Omega) \cup V(\Qcal)$.
Then, 
using $P$, $\Qcal^\star$, and $\{Q_j\mid j\in I_i\setminus I^\star\}$,
we conclude that there is a $k$-long-jump transaction in $(G,\Omega)$ that is coterminal with $\Pcal$ up to level $C_2$, a contradiction.
\end{cproof}

Finally, we have the following.

\begin{claim}
The $\Qcal^\star$-strip society of $(G, \Omega)$ is a subgraph of $H^\star$.
\end{claim}

\begin{cproof}
The proof is exactly the same as the one of \cite[Theorem 5.11, Claim 6]{kawarabayashi2020quickly} with $Z=\emptyset$.
Essentially, if that happened, then it would contradict \autoref{cl:nojumpedge}.
\end{cproof}

The theorem now follows as the vortex-free rendition of $(H^\star, \Omega^\star)$ restricted to the $\Qcal^\star$-strip society shows that the strip society is rural, and \autoref{cl:nojumpedge} shows that the $\Qcal^\star$-strip society is isolated.
\end{proof}

\subsection{Splitting a vortex}
\label{sec:split}

In this part, we prove that, if a society $(G,\Omega)$ contains a unique vortex $c_0$ and a strip society that is rural and isolated, then $c_0$ can be split into two vortices $c_1$ and $c_2$ of smaller size that are separated by the strip society.\medskip

We require the following technical result.

\begin{proposition}[Lemma 5.15, \cite{kawarabayashi2020quickly}]\label{lem:planarstriprestriction}
Let $(G,\Omega)$ be a society with a cylindrical rendition $\rho=(\Gamma,\Dcal,c_0)$.
Let $\Ccal=(C_1,...,C_s)$ be a nest in $\rho$ and let $\Qcal$ be an exposed monotone transaction in $(G,\Omega)$ of order at least $3$ orthogonal to $\Ccal$.
Let $(H,\Omega_H)$ be the $\Qcal$-strip society in $(G,\Omega)$ and let $Y_1$ and $Y_2$ be the two segments of $\Omega\setminus V(H)$.
Assume that there exists a linkage $\Pcal=\{P_1,P_2\}$ such that $P_i$ links $Y_i$ and $V(\sigma(c_0))$ for $i\in\{1,2\}$, $\Pcal$ is disjoint from $\Qcal$, and $\Pcal$ is orthogonal to $\Ccal$.

Let $i\geq7$. Let $(G',\Omega')$ be the inner society of $C_i$ and let $\Qcal'$ be the restriction of $\Qcal$ to $(G',\Omega')$.
Let $\rho'=(\Gamma',\Dcal',c_0)$ be the restriction of $\rho$ to be a cylindrical rendition of $(G',\Omega')$.
Then $\Qcal'$ is exposed and monotone in $\rho'$. 
Moreover, $\Qcal'$ is a crosscap transaction if and only if $\Qcal$ is a crosscap transaction.
If the $\Qcal$-strip society is rural and isolated in $(G,\Omega)$, then the $\Qcal'$-strip society is rural and isolated in $(G',\Omega')$.
\end{proposition}

The main result of this section is that, given a cylindrical rendition $\rho$ of a society $(G,\Omega)$ and a planar transaction $\Qcal$ such that the $\Qcal$-strip society in $(G,\Omega)$ is rural and isolated, there is another rendition $\rho'$ of $(G,\Omega)$ of breadth two containing two smaller cylindrical renditions, and otherwise vortex-free.

\begin{lemma}\label{main2}
Let $k, l,m, s,s' \in\Nbbb$ such that $l\ge1$, $m\ge\max\{2s'+5,k+2\}$, and $s\ge s'+8$.
Let $(G, \Omega)$ be a society and $\rho = (\Gamma,\Dcal,c_0)$ be a cylindrical rendition of $(G, \Omega)$ in a disk $\Delta$.
Let $X_1,X_2$ be two disjoint segments of $\Omega$ and let $Y_1,Y_2$ be the two segments of $\Omega\setminus(X_1\cup X_2)$.
Let $\Qcal$ be an exposed planar transaction from $X_1$ to $X_2$ of size $m$ such that the $\Qcal$-strip society in $(G,\Omega)$ is isolated and rural.
Let $(\Ccal = (C_1,C_2, \dots, C_s),\Pcal \cup \Pcal_1 \cup \Pcal_2)$ be a railed nest in $\rho$ of order $(s,2m+2l)$ around $c_0$ such that $\Pcal$ is the rail truncation of $\Qcal$ and,
for $i\in[2]$, $\Pcal_i$ is a linkage in $\rho$ from $Y_i$ to $V(\sigma(c_0))$ of size $l$.

Then either there is a $k$-long-jump transaction in $(G,\Omega)$ that is coterminal with $\Pcal\cup\Pcal_1\cup\Pcal_2$ up to level $C_8$, 
or there is a rendition $\rho'$ of $(G,\Omega)$ in $\Delta$ with exactly two vortices $c_1$ and $c_2$, and two disjoint $\rho'$-aligned disks $\Delta_1',\Delta_2'\subseteq \Delta$ such that
\begin{itemize}
\item for $i\in[2]$, $\rho_i = \rho'[\Delta_i']$ is a cylindrical rendition of $(\inG_{\rho'}(\Delta_i'),\Omega_{\Delta_i'})$ with vortex $c_i$,
\item there is a railed nest $(\Ccal_i=(C_1^i,\dots,C_{s'}^i),\Pcal_i')$ of order $(s',l)$ around $c_i$ in $\rho_i$, where, for each $P\in\Pcal_i'$, $P$ is a subpath of an element of $\Pcal_i$ and, for each $j\in[s']$, $P\cap C_j^i= P\cap C_{j+7}$,
\item $(C_8,\dots,C_s)$ is a nest around both the closure of $c_1$ and the closure of $c_2$, and
\item there is a path $Q\in\Qcal$ such that $\Delta_1'$ and $\Delta_2'$ are contained in different connected components of $\Delta\setminus T$, where $T$ is the track of $Q$ in $\rho'$.
\end{itemize}
\end{lemma}

\begin{figure}[h!]
\centering
\includegraphics[width=0.54\textwidth]{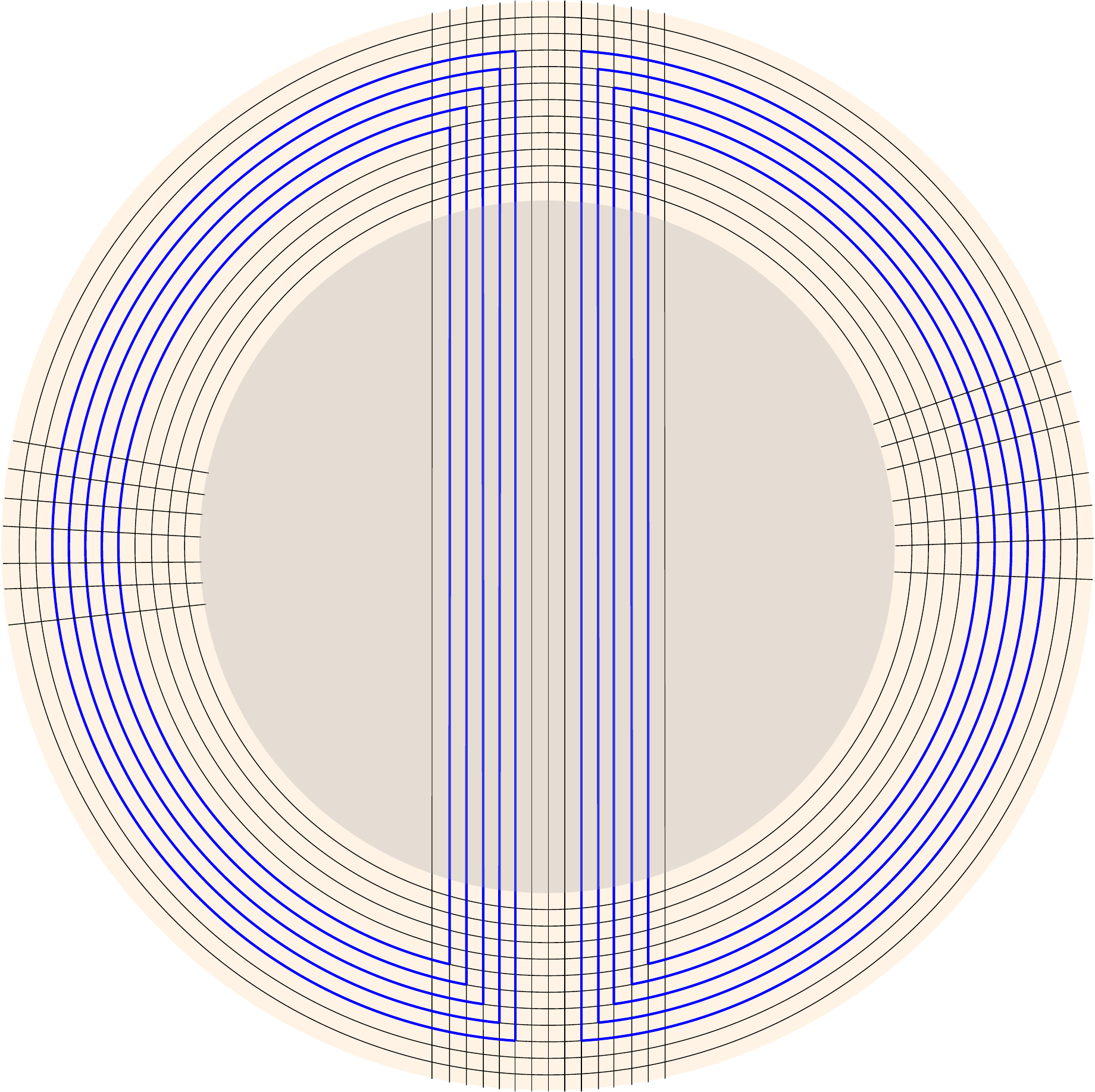}
 \caption{The rails, the nest, and the planar transaction of the original cylindrical rendition are represented in black and the rails and nests of the two new cylindrical renditions are represented in blue.}
\label{figsplit}
\end{figure}

\begin{proof}
See \autoref{figsplit} for an illustration.
Let us assume that $G$ does not contain a long-jump transaction of order $k$ in $(G,\Omega)$ that is coterminal with $\Pcal_1\cup\Pcal_2\cup\Pcal$ up to level $C_8$.
%
Let the elements of $\Qcal$ be enumerated $Q_1, Q_2, \dots, Q_{m}$ by the order in which their endpoints occur on $X_1$.

{\bf Restricting the strip society.}
We have $s\ge7$.
Let $(H, \Omega_H)$ be the inner society of $C_7$ in $(G, \Omega)$ with respect to the rendition $\rho$.  Let $\Delta_H$ be the closed subdisk bounded by the track of $C_7$.  Let $\Qcal_H$ be the restriction of $\Qcal$ to $(H, \Omega_H)$.
Let $(J, \Omega_J)$ be the $\Qcal_H$-strip society in $(H, \Omega_H)$.

Given that $l\ge1$, there are $P_1\in\Pcal_1$ and $P_2\in\Pcal_2$.
By \autoref{lem:planarstriprestriction} applied to the society $(G, \Omega)$, linkages $\Qcal$ and $\{P_1,P_2\}$, and cycle $C_7$,
the society $(J, \Omega_J)$ is rural and isolated in $(H, \Omega_H)$.  There exists
 a vortex-free rendition $\rho_{J}$ of $(J, \Omega_J)$ in $\Delta_H$
 such that
$\pi_\rho^{-1}(v)=\pi^{-1}_{\rho_{J}}(v)$ for all $v\in V(\Omega _J)$.
 Note that we are using the fact that $\Qcal$ is planar to ensure that the cyclic ordering $\Omega_J$ is the same as the cyclic ordering of $V(\Omega_J)$ in $\Omega_H$.

Let $T_1$ (resp. $T_2$) be the track of $Q_{1}$ (resp. $Q_{m}$) restricted to $(H, \Omega_H)$ in $\rho_{J}$.   
There is a unique connected component of $\Delta_H \setminus(T_1 \cup T_2)$ whose boundary includes both $T_1$ and $T_2$; 
let $\Delta^{*}$ be the closure of this connected component. 
Let $J'$ be the subgraph formed by
$\bigcup \{\sigma_{\rho_{J}}(c)\,:\,c \in C(\rho_{J}), c \subseteq \Delta^{*}\}$ 
along with any vertices $v$ of $J$ such that $\pi^{-1}_{\rho_{J}}(v)$ exists and belongs to $\Delta^{*}$.  
Let $\Omega_{J'}$ be the cyclically ordered set of vertices with $V(\Omega_{J'}) = V(\Omega_J) \cap V(J')$ with the cyclic order induced by $\Omega_J$.  
Let $\rho_{J'}$ be the restriction of $\rho_{J}$ to the disk $\Delta^{*}$.

{\bf Defining a new vortex-free rendition $\rho^\star$ avoiding the two future vortices.}
We define the society $(G', \Omega)$ as follows.
Let $G'$ be the union of $J'$ and the outer graph of $C_7$ with respect to $\rho$.
Thus, the union of $\rho_{J'}$ along with the restriction of $\rho$ to the complement of the interior of $\Delta_H$ gives us a vortex-free rendition $\rho^\star$ of $(G', \Omega)$ in the disk $\Delta$.

{\bf Defining $\Delta_1$ and $\Delta_2$.}
Consider $\Delta_H \setminus \Delta^{*}$.  
There is one connected component which contains $T_1$ in its boundary and one which contains $T_2$ in its boundary.  
For $i\in[2]$, let $U_i$ be the set of vertices $u \in V(G')$ such that
$\pi_{\rho*}^{-1}(u)$
exists and is contained in the boundary of the connected component of $\Delta_H \setminus \Delta^{*}$ with $T_i$ in the boundary. 
Thus by construction, ($U_1, U_2)$ is a partition of the set of vertices $\left( V(\Omega_H) \setminus V(\Omega_{J'})\right) \cup \pi_{\rho_{J'}}((T_1 \cup T_2) \cap N(\rho_{J'}))$.
Finally, for $i\in[2]$, fix the closed disk $\Delta_i$ in $\Delta_H \setminus \Delta^{*}$ such that $\Delta_i \cap( \bd(\Delta_H) \cup T_i ) = \pi_{\rho^\star}^{-1}(U_i)$.


{\bf Defining the subgraphs contained in $\Delta_1$ and $\Delta_2$.}
Let $L$ be the minimal subgraph of $G$ such that $G = G' \cup L$.  
Note that $L$ is a subgraph of $H$.

\begin{claim}
$V(L) \cap V(G') \subseteq U_1 \cup U_2$.
\end{claim}

\begin{cproof} 
Consider a vertex $x \in V(L) \cap V(G')$.  
By the minimality of $L$, the vertex $x$ cannot be deleted from $L$ and still have the property that $L \cup G' = G$.  
Thus, there exists an edge $xy$ incident to $x$ which is not contained in $G'$ and consequently, $xy$ is in $H$, the inner graph of $C_7$.  
Given that $x\in V(G')$, the vertex $x$ is either in the outer graph of $C_7$ or in $J'$.  
If $x$ is in the outer graph of $C_7$, but not in $J'$, then $x \in V(\Omega_H)\setminus V(\Omega_{J'})$
 and thus $x \in U_1 \cup U_2$.  
Thus, we may assume $x \in V(J')$.  
However, in this case as $(J', \Omega_{J'})$ is isolated in $(H, \Omega_H)$, it follows that $x \in \pi_{\rho_{J'}} \left ( ( T_1 \cup T_2) \cap N(\rho_{J'}) \right)$.  
 We conclude that $x \in U_1 \cup U_2$ as claimed.
 \end{cproof}

\begin{claim}\label{cu}
There is no $U_1-U_2$-path in $L$.
\end{claim}

\begin{cproof}
Assume otherwise that a $U_1-U_2$-path $R$ exists.
$R$ can be extended to a $\Omega$-path $R'$ using $Q_1$, $Q_m$, and $C_7$.
Then $R'$ along with $\{Q_2,...,Q_{m-1}\}$ and $C_7$ forms a $k$-long-jump transaction in $(G,\Omega)$ that is coterminal with $\Pcal_1\cup\Pcal_2\cup\Qcal$ up to level $C_8$, given that $m-2\ge k$.
\end{cproof}
 
Let $K_1$ be the union of all components of $L$ which contain a vertex of $U_1$ and
let $K_2$ be the union of all components of $L$ which contain no vertex of $U_1$.
Hence, by \autoref{cu}, $K_2$ contains all components of $L$ which contain a vertex of $U_2.$


{\bf Defining the new rendition $\rho'$.}
We are now ready to define the desired rendition $\rho' = (\Gamma', \Dcal')$ of $(G, \Omega)$.  
Let $\rho^\star=(\Gamma^\star,\Dcal^\star)$.
Define $\Dcal' = \{D \in \Dcal^\star: \hbox{int}(D) \subseteq \Delta \setminus (\Delta_1 \cup \Delta_2)\} \cup \{\Delta_1, \Delta_2\}$.  
The drawing $\Gamma'$ is obtained from the restriction of $\Gamma^\star$ to $\Delta \setminus (\hbox{int}(\Delta_1) \cup \hbox{int}(\Delta_2))$ along with an arbitrary drawing of $K_i$ in $\Delta_i$ such that the only points on the boundary of $\Delta_i$ are exactly the vertices of $U_i\cap V(K_i)$ for $i\in[2]$.


{\bf Defining the nests.}
Observe that for all $i\in[2,m-1]$, the path $Q_{i}$ is contained in $G'$
and the subgraph $Q_{i} \cup C_{7 + i}$ contains exactly two cycles which are distinct from $C_{7 + i}$.  
If we consider the track in $\rho^\star$ of each these two cycles, one bounds a disk in $\Delta$ which contains $\Delta_1$ and the other bounds a disk which contains $\Delta_2$.  
For $i\in[s'+1]$,
we define $C_i^1$ (resp. $C_i^2$) to be the cycle contained in $Q_{i+1} \cup C_{7 + i}$ (resp. $Q_{m-i} \cup C_{7 + i}$)
and distinct from $C_{7+i}$ such that the disk bounded by the track of $C_i^1$ (resp. $C_i^2$) in $\rho^\star$ contains $\Delta_1$ (resp. $\Delta_2$).  
Observe that, for $i\in[2]$, $\Ccal_i=(C_1^i, \dots, C_{s'}^i)$ forms a nest around $\Delta_i$.
Additionally, note that $\Ccal_1$ and $\Ccal_2$ are disjoint given that they do not intersect $\{Q_i\mid i\in[s'+2,m-s'-1]\}$ and that $m-s'-1\ge s'+4$.
Moreover, $\Pcal_1$ and $\Pcal_2$ immediately validate the conditions of the statement.

{\bf Defining $\Delta_i'$.}
Let $\Delta_i'$ be the closed subdisk bounded by the track of $C_{s'+1}^i$.
Then $\rho_i=\rho'[\Delta_i']$ is indeed a cylindrical rendition with vortex $c_i=\Delta_i$.
Note that $\Delta'_1$ and $\Delta'_2$ do not intersect given that $m\ge 2s'+5$, and thus that the track of $\Qcal_{s'+3}$ is disjoint from both $\Delta'_1$ and $\Delta'_2$.
Additionally, for each $P\in\Pcal_i$, $i\in[2]$, there is a unique subpath from $\Omega_{\Dcal_i'}$ to $V(\sigma(c_i))$.
Then the set $\Pcal_i'$ of all such subpaths checks all constraint of the second item of the lemma.
This concludes the proof.
\end{proof}

\subsection{Finding a planar rendition of small depth and breadth}
\label{sec:main1}

We now prove that a cylindrical rendition contains either a long-jump transaction, or a crosscap transaction, or has a rendition in the plane of small depth and small breadth.

\medskip
We define recursively
$$f_{\ref{main1}}(k,r)=a_kf_{\ref{main1}}(k-1,r)+b_k$$
with $f_{\ref{main1}}(1,r)=14$, $a_k=6(r-1)(2k+1),$ and $b_k=6(r-1)k(3k+5)+6.$
Hence,
$$f_{\ref{main1}}(k,r) = (\prod_{i=1}^k a_i) f_{\ref{main1}}(1,r) + \sum_{i=2}^k (\prod_{j+1}^k a_j) b_i = 2^{O(k\log(k\cdot r))}.$$

\begin{theorem}\label{main1}
Let $k, r, s \in \Nbbb$ with $k \ge 1$ and $s, t \ge f_{\ref{main1}}(k, r)$.
Let $(G, \Omega)$ be a society and $\rho=(\Gamma,\Dcal,c_0)$ be a cylindrical rendition of $(G, \Omega)$.
Let $(\Ccal = (C_1,C_2, \dots, C_s),\Pcal)$ be a railed nest in $\rho$ around $c_0$ of order $(s,t)$.
Then one of the following holds.
\begin{itemize}
\item[\rm{(i)}] There is an $r$-crosscap transaction in $(G, \Omega)$ that is coterminal with $\Pcal$ up to level $f_{\ref{main1}}(k,r)$,
\item[\rm{(ii)}]  There is a $k$-long-jump transaction in $(G, \Omega)$ that is coterminal with $\Pcal$ up to level $f_{\ref{main1}}(k,r)$, or
\item[\rm{(iii)}] $(G, \Omega)$ has a rendition in the disk of breadth at most $k-1$ and depth at most
$f_{\ref{main1}}(k,r)$.
\end{itemize}
Additionally, in case (iii), the closures of the vortex cells are pairwise disjoint.
\end{theorem}
\begin{proof}
Assume the theorem is false, and pick a counterexample to minimize $k$.

We fix the following values.
Let $s'= f_{\ref{main1}}(k-1,r)$ and $t'= \lceil f_{\ref{main1}}(k-1,r)/2\rceil$.
Additionaly, let
\begin{align*}
m_3 = & 2s'+5,\\
m_2' = & k(m_3+3k),\\
m_2 = & m_2'+2t', \text{ and}\\
m_1 = & (m_2-1)(r-1)+1\le f_{\ref{main1}}(k,r)/6.
\end{align*}

Fix the cylindrical rendition $\rho$ in a disk $\Delta$.
If $(G, \Omega)$ has no cross, then by \autoref{thm:crossreduct} and given that $k\ge1$, it satisfies (iii).  
Given a cross in $(G, \Omega)$,
since $s\ge 11$ and $t\ge 14$,
by \autoref{lem:crookedrooted}, there is a cross in  $(G, \Omega)$ that is coterminal with $\Pcal$ up to level $C_{11}$.
Given that $f_{\ref{main1}}(k,r)\ge 11$, (i) and (ii) hold in the case $r\le 2$ and $k\le 1$, respectively.
We conclude that $r\ge 3$ and $k\ge 2$.


\medskip
\noindent{\bf Step 1: Finding a crooked transaction.}
By \autoref{lem:GM9}, the society $(G, \Omega)$ has either a cylindrical rendition of depth $6m_1\le f_{\ref{main1}}(k,r)$,
or a crooked transaction of cardinality $m_1$.  In the former case, (iii) holds. Note we are using the fact that $m_1 \ge 4$.
Hence, we can assume that there exists a crooked transaction $\Qcal_1$ of cardinality $m_1$.
By \autoref{lem:crookedrooted}, given that $s\ge 6m_1\ge2m_1+7$ and $t\ge 6m_1\ge4m_1+6$, we may assume that $\Qcal_1$ is coterminal with $\Pcal$ up to level $C_{2m_1+7}$, and hence $C_{f_{\ref{main1}}(k,r)}$.


\medskip
\noindent{\bf Step 2: Finding an exposed planar transaction.}
Given that $m_1\ge (m_2-1)(r-1)+1$, by \autoref{planar or crosscap},
there is $\Qcal_2\subseteq \Qcal_1$ that is either a crosscap transaction of cardinality $r$, or a planar transaction of cardinality $m_2$.
In the first case, (i) holds, so we conclude that $\Qcal_2$ is a planar transaction.


\medskip
\noindent{\bf Step 3: Creating rails.}
Let $X_1$ and $X_2$ be two disjoint segments of $\Omega$ such that $\Qcal_2$ is a transaction from $X_1$ to $X_2$.
Let the elements of $\Qcal_2$ be enumerated $Q_1,\dots,Q_{m_2}$ by the order in which their endpoints occur in $X_1$.
Let $\Qcal_2'=\{Q_i\mid i\in[t'+1,t'+m_2']\}$.
Let $\Pcal_1$ be the rail truncation of $\{Q_i\mid i\in[1,t']\}$ and $\Pcal_2$ be the rail truncation of $\{Q_i\mid i\in[t'+m_2'+1,m_2'+2t']\}$.
Hence, $\Pcal_1$ and $\Pcal_2$ both have $2t'$ elements.


\medskip
\noindent{\bf Step 4: Finding a rural and isolated strip society.}
Apply \autoref{planarstrip} to the transaction $\Qcal_2'$ in $(G, \Omega)$ with the rendition $\rho$, $l'=m_1$, $l=m_2'$, and the nest $(C_{2m_1+7}, C_{2m_1+8}, \dots, C_s)$.
We can do so because $s-2m_1-6\ge 9$ and $m_2'\ge k(m_3+3k)$.
If we find a long-jump transaction of order $k$ that is coterminal with $\Pcal$ up to level $C_{2m_1+15}$, then we satisfy (ii) given that $f_{\ref{main1}}(k,r)\ge2m_1+15$.
Thus, we may assume that we find a transaction $\Qcal_3\subseteq \Qcal_2'$ of size $m_3$ such that the $\Qcal_3$-planar strip society in $(G, \Omega)$ is rural and isolated.


\medskip
\noindent{\bf Step 5: Splitting the vortex in two.}
We apply \autoref{main2} to $(G,\Omega)$ with $m=m_3$, $l=2t'$, railed nest $((C_{2m_1+7},\dots,C_s),\Pcal_1\cup\Pcal_2\cup\Pcal_3)$ where $\Pcal_3$ is the rail truncation of $\Qcal_3$.
We can do so because $t'\ge1$, $m_3\ge2s'+5\ge k+2$, and $s-2m_1-6\ge 4m_1-6\ge 8ks'-6\ge s'+8$.
If there is a long-jump transaction of order $k$ that is coterminal with $\Pcal_1\cup\Pcal_2\cup\Pcal_3$ up to level $C_{2m_1+15}$, then this transaction is coterminal with $\Pcal$ up to level $C_{f_{\ref{main1}}(k,r)}$ since $f_{\ref{main1}}(k,r)\ge 2m_1+15$.
We conclude that there is a rendition $\rho'$ of $(G,\Omega)$ of breadth two, with vortices $c_1$ and $c_2$, two disjoint $\rho'$-aligned disks $\Delta_1$ and $\Delta_2$ such that, for $i\in[2]$,
$\rho'[\Delta_i]$ is a cylindrical rendition $\rho_i=(\Gamma_i,\Dcal_i,c_i)$ of $(G_i,\Omega_i)=(\inG_\rho'(\Delta_i),\Omega_{\Delta_i})$ with a railed nest $(\Ccal_i=(C_1^i,\dots,C_{s'}^i),\Pcal_i')$ of order $(s',t')$
such that, for each $P\in\Pcal_i'$, $P$ is a subpath of an element of $\Pcal_i$ and, for each $j\in[s']$, $P\cap C_j^i= P\cap C_{j+\alpha}$, where $\alpha=2m_1+21$.
Additionally, $(C_{2m_1+14},\dots,C_s)$ is a nest around both $c_1$ and $c_2$, and there is a path $Q\in\Qcal_3$ such that $\Delta_1$ and $\Delta_2$ are contained in different connected components of $\Delta\setminus T$, where $T$ is the track of $Q$ in $\rho'$.


\medskip
\noindent{\bf Step 6: Finding a contradiction.}
For $i\in[2]$,
let $k_i$ be the largest value such that there exists a $k_i$-long-jump transaction $\Rcal_i$ in $(G_i,\Omega_i)$ that is coterminal with $\Pcal_i'$ up to level $C_{s'}^i$.

\begin{claim}\label{qy}
For $i\in[2]$, $k_i>0$.
\end{claim}

\begin{cproof}
Let $X_1'$ and $X_2'$ be minimal segments of $\Omega$ such that $\Qcal_3$ is a linkage from $X_1'$ to $X_2'$.
Let $Y_1$ and $Y_2$ be the two maximal segments of $\Omega\setminus (X_1'\cup X_2')$.
Let $i\in[2]$.
Since the planar transaction $\Qcal_3$ is a subset of the crooked transaction $\Qcal_1$,  there exists a path $P_i\in\Qcal_1$ with an endpoint in $Y_i$  that is crossed by another path $P_i'\in\Qcal_1$.
If $P_i$ crosses some $Q\in\Qcal_3$, then, given that the $\Qcal_3$-strip society in $(G,\Omega)$ is isolated, it follows that the other endpoint  of $P_i$ is in $Y_2$.
But then, given that $m_3\ge k$, we conclude that there is a long-jump transaction $\Qcal\subseteq\Qcal_1$ of order $k$ in $(G,\Omega)$.
Given that $Q_1$ is coterminal with $\Pcal$ up to level $C_{f_{\ref{main1}}(k,r)}$, (ii) holds, a contradiction.
Hence, $P_i$ crosses no $Q\in\Qcal_3$, and neither does $P_i'$.
Hence, $P_i$ and $P_i'$ have both endpoints in $Y_i$.
The pairs of paths $(P_1,P_1')$ and $(P_2,P_2')$ form a cross, one in $(G_1,\Omega_1)$ and the other in $(G_2,\Omega_2)$ since there is a path $Q\in\Qcal_3$ such that the interiors of $\Delta_1$ and $\Delta_2$ are contained in different connected components of $\Delta\setminus T$, where $T$ is the track of $Q$ in $\rho'$.
Therefore, since $s'\ge 11$ and $t'\ge14$,
by \autoref{lem_crookedrooted}, for $i\in[2]$, there is a cross, and hence a 1-long-jump transaction, in  $(G_i, \Omega_i)$ that is coterminal with $\Pcal_i'$ up to level $C_{11}^i$, and hence $C_{s'}^i$.
We conclude that $k_i>0$.
\end{cproof}

\begin{claim}\label{qi}
$k_1+k_2< k$.
\end{claim}

\begin{cproof}
Given that $\Pcal_1'$ and $\Pcal_2'$ are disjoint linkages, whose paths are subpaths of $\Pcal_1$ and $\Pcal_2$ respectively, that $(C_{2m_1+14}, \dots, C_s)$ is a nest around both $c_1$ and $c_2$, and that for each $P \in \Pcal_i'$ and $j\in[s']$, $P\cap C_j^i = P\cap C_{i+\alpha}$ by \autoref{main2}, we conclude that, for $i\in[2]$, $\Rcal_i$ can be extended using $\Pcal_i\subseteq \Qcal_1$ to a long-jump transaction $\Rcal_i'$ of $(G,\Omega)$, and that $\Rcal_1'$ and $\Rcal_2'$ form either a $(k_1,k_2)$-double-jump transaction or, along with $Q$, an alternative $(k_1,k_2)$-double-jump transaction that is coterminal with $\Pcal$ up to level $C_{s'+\alpha}$.
Therefore, by \autoref{allinone}, there is a $(k_1+k_2)$-long-jump transaction in $(G,\Omega)$ that is coterminal with $\Pcal$ up to level $C_{s'+\alpha+1}$, and thus $C_{f_{\ref{main1}}(k,r)}$, since $f_{\ref{main1}}(k,r)\ge6m_1\ge s'+2m_1+22$.
Given that there is no $k$-long-jump transaction in $(G,\Omega)$ that is coterminal with $\Pcal$ up to level $C_{f_{\ref{main1}}(k,r)}$, we conclude that 
$k_1+k_2< k$.
\end{cproof}

By \autoref{qy} and \autoref{qi}, we conclude that, for $i\in[2]$, $1\le k_i\le k-2$.
By the choice of our counterexample to minimize $k$, it follows that one of (i)-(iii) must hold for $(G_i,\Omega_i)$ with parameters $k_i+1$ and $r$.
We can do so because $s',2t'\ge f_{\ref{main1}}(k-1,r,a)$.
By the maximality of $k_i$, outcome (ii) does not hold.

If (i) holds, then there is a $r$-crosscap transaction in $(G_i,\Omega_i)$ that is coterminal with $\Pcal_i'$ up to level $C_{f_{\ref{main1}}(k_i+1,r)}^i$.
Then, given that for each $P \in \Pcal_i'$ and $j\in[s']$, $P\cap C_j^i= P\cap C_{i+\alpha}$, we can extend $\Qcal_i$ to a crosscap transaction of $(G,\Omega)$ that is coterminal with $\Pcal$ up to level $C_{f_{\ref{main1}}(k_i+1,r)+\alpha}$.
Given that $f_{\ref{main1}}(k_i+1,r)+\alpha\le f_{\ref{main1}}(k_i+2,r)\le f_{\ref{main1}}(k,r)$, (i) thus holds for $(G,\Omega)$.

Hence, outcome (iii) holds for both $(G_1,\Omega_1)$ and $(G_2,\Omega_2)$.
Thus, for $i\in[2]$, there is a rendition $\rho_i'$ of $(G_i,\Omega_i)$ of breadth at most $k_i$ and depth at most $f_{\ref{main1}}(k_i+1,r)$.
Given that $k_1+k_2<k$ by \autoref{qi}, 
we conclude that $(G,\Omega)$ has a rendition of breadth at most $k_1+k_2\le k-1$ and depth at most $ f_{\ref{main1}}(k-1,r)$, by restricting $\rho'$ to $\Delta\setminus ({\sf int}(\Delta_1)\cup {\sf int}(\Delta_2))$ and using $\rho_1'$ and $\rho_2'$ in the disks $\Delta_1$ and $\Delta_2$, respectively.
In particular, given that $\Delta_1$ and $\Delta_2$ are disjoint, the closures of the vortex cells of this rendition are pairwise disjoint.
This contradiction completes the proof.
\end{proof}

\subsection{Finding a projective rendition of small depth and breadth}
\label{sec:mainter}

In the previous section, we proved that if a society has neither a  long-jump nor a crosscap transaction, then it has a rendition in the plane of small breadth and depth.
Our goal is now to only exclude a long-jump transaction to get a rendition in the projective plane of small breadth and depth.

\medskip
We first borrow the following result from \cite{kawarabayashi2020quickly}.
It essentially says that, if our cylindrical rendition with vortex $c_0$ contains a crosscap transaction $\Qcal$, then we can use $\Qcal$ to add a crosscap to the surface (that is, we are now in the projective plane), and find a new cylindrical rendition with vortex $c_1$ avoiding the crosscap.

\begin{proposition}[Lemma 10.2, \cite{kawarabayashi2020quickly}]\label{vortexcrosscap}
\label{lem:xcaptrans}
Let $m,m',s,s',t\in\Nbbb$ with $s\ge s'+8$ and $m\ge m'+2s'+7$.
Let $(G, \Omega)$ be a society and $\rho = (\Gamma, \Dcal, c_0)$ be a cylindrical rendition of $(G, \Omega)$ in the disk $\Delta$.
Let $X_1,X_2$ be disjoint segments of $V(\Omega)$.
Let $\Ccal = (C_1, \dots, C_{s})$ be a nest in $\rho$ of size $s$ and $\Pcal=\{P_1,P_2,\ldots ,P_t\}$ a linkage from $X_1$
to $\widetilde{c_0}$ orthogonal to $\Ccal$.  Let $\Sigma^\star$ be a surface homeomorphic to the projective plane minus an open disk obtained from $\Delta$ by adding a crosscap to the interior of $c_0$.
If there exists an $m$-crosscap transaction $\Qcal$ in $(G , \Omega )$ orthogonal to $\Ccal$ and disjoint from
$\Pcal$ such that every member of $\Qcal$ has both endpoints in $X_2$ and
  the $\Qcal$-strip society in $(G , \Omega )$ is isolated and rural, then there exists a subset $\Qcal'$ of $\Qcal$
of size $m'$ and a rendition $\rho'$ of $(G, \Omega)$
in $\Sigma^\star$ such that there exists a unique vortex $c'\in C(\rho')$ and the following hold:
\begin{itemize}
\item[\rm{(i)}] $\Qcal'$ is disjoint from $\sigma (c')$,
\item [\rm{(ii)}] the vortex society of $c'$ in $\rho'$ has a cylindrical rendition $\rho_1=(\Gamma_1,\Dcal_1,c_1)$,
\item [\rm{(iii)}] every element of $\Pcal$ has an endpoint in $V(\sigma_{\rho_1}(c_1))$,
\item [\rm{(iv)}] $\rho_1$ has a nest $\Ccal' = (C_1', \dots, C_{s'}')$ of size $s'$ such that $\Pcal$ is orthogonal to $\Ccal'$ and for every $i$, $1 \le i \le s'$, and for all $P \in \Pcal$, $C_i' \cap P = C_{i+7} \cap P$,
\item [\rm{(v)}] for $i=1,2,\ldots ,t$, let $x_i$ be the endpoint 
of $P_i$ in $X_1$, and let $y_i$ be the last entry of $P_i$ into $c'$ (which exists),
then if $x_1,x_2,\ldots ,x_m$ appear in $\Omega$ in the order listed,
then $y_1,y_2,\ldots ,y_m$ appear on $\widetilde{c'}$ in the order listed,
\item [\rm{(vi)}] let $\Delta'$ be the open disk bounded by the track of $C_{s}$; then $\rho$ restricted to $\Delta \setminus \Delta'$ is equal to $\rho'$ restricted to $\Delta \setminus \Delta'$.
\end{itemize}
\end{proposition}

We can finally prove our local structure theorem.

\begin{theorem}\label{mainter}
Let $k, r,s,t \in\Nbbb$ such that $s\ge f_{\ref{main1}}(k,m_1)+f_{\ref{main1}}(k,k+1)+7$ and $t\ge f_{\ref{main1}}(k,m_1)$ where $m_1=(6k+1)f_{\ref{main1}}(k,k+1)+3k(4k+9)$.
Let $(G, \Omega)$ be a society and $\rho=(\Gamma,\Dcal,c_0)$ be a cylindrical rendition of $(G, \Omega)$ in the disk $\Delta$.
Let $(\Ccal = (C_1,C_2, \dots, C_s),\Pcal)$ be a railed nest in $\rho$ around $c_0$ of order $(s,t)$.
Then one of the following holds.
\begin{itemize}
\item[\rm{(i)}] There is a $k$-long-jump transaction in $(G,\Omega)$ that is coterminal with $\Pcal$ up to level $f_{\ref{main1}}(k,m_1)$, or
\item[\rm{(ii)}] $(G, \Omega )$ has a rendition in the plane of depth at most $f_{\ref{main1}}(k,m_1)$ or in the projective plane of depth at most $f_{\ref{main1}}(k,k+1)$, and of breadth at most $k-1$ in both cases
\end{itemize}
Additionally, in case (ii), the closures of the vortex cells of the rendition are pairwise disjoint.
\end{theorem}


\begin{proof}
Assume that the theorem is false.
We fix the following values.
Let
\begin{align*}
s',t' = & f_{\ref{main1}}(k,k+1),\\
m_3 = & k+2s'+9,\\
m_2 = & 3k(m_3+3k), \text{ and}\\
m_1 = & m_2+t'.
\end{align*}

\medskip
\noindent{\bf Step 1: Finding a crosscap transaction.}
We apply \autoref{main1} with $r=m_1$.
We can do so because $s,t\ge f_{\ref{main1}}(k,m_1)$.
Outcome (ii) and (iii) of \autoref{main1} immediately imply outcome (i) and (ii) of the theorem respectively.
Hence, there is a $m_1$-crosscap transaction $\Qcal_1$ in $(G,\Omega)$ that is coterminal with $\Pcal$ up to level $f_{\ref{main1}}(k,m_1)$.

\medskip
\noindent{\bf Step 2: Separating $\Qcal_1$ into a crosscap transaction $\Qcal_2$ and rails $\Pcal'$.}
Let $X$ be a minimal segment of $\Omega$ containing both endpoints of every path in $\Qcal_1$.
Let $X_1$ and $X_2$ be two disjoint segments contained in $X$ such that $X_2$ contains both endpoints of $m_2$ paths of $\Qcal_1$; we call this set of paths $\Qcal_2$; and $X_1$ contains one endpoint of the $t'$ paths in $\Qcal_1\setminus\Qcal_2$. We can do so because $m_1\ge m_2 + t'$.
Let $\Pcal'$ be the rail truncation of $\Qcal_1\setminus\Qcal_2$.
$\Pcal'$ has size at least $t'$.

\medskip
\noindent{\bf Step 3: Finding a rural and isolated strip society in $\Qcal_2$.}
Let $X_1'$ and $X_2'$ be the minimal segments contained in $X_2$ so that each path of $\Qcal_2$ has one endpoint in $X_1'$ and the other in $X_2'$.
We apply \autoref{planarstrip} with input the society $(G,\Omega)$, the nest $(C_{f_{\ref{main1}}(k,m_1)},\dots,C_s)$, the crosscap transaction $\Qcal_2$, segments $X_1'$ and $X_2'$, $l=m_3$, and $l'=m_2$.
We can do so because $s-f_{\ref{main1}}(k,m_1)+1\ge 9$ and $m_2\ge 3k(m_3+3k)$.
If there is a long-jump transaction of order $k$ in $(G,\Omega)$ that is coterminal with $\Pcal$ up to level $C_{f_{\ref{main1}}(k,m_1)+9}$, then (i) holds given that $s\ge f_{\ref{main1}}(k,m_1)+9$.
Hence, there exists a transaction $\Qcal_3\subseteq \Qcal_2$ of size $m_3$ such that the $\Qcal_3$-strip of $(G,\Omega)$ with respect to $(X_1',X_2')$ is isolated and rural.

\medskip
\noindent{\bf Step 4: Embedding the society in the projective plane.}
Let $\Sigma^\star$ be a surface  homeomorphic to the projective plane minus an open disk obtained from $\Delta$ by adding a crosscap to the interior of $c_0$.
We apply \autoref{vortexcrosscap} to the society $(G,\Omega)$, segments $X_1,X_2$, nest $(C_{f_{\ref{main1}}(k,m_1)},\dots,C_s)$, $\Pcal'$, $\Qcal_3$, and $m'=k$.
We can do so because $s-f_{\ref{main1}}(k,m_1)+1\ge s'+ 8$ and $m_3\ge k+2s'+7$.
Thus, there exists a transaction $\Qcal_4\subseteq \Qcal_3$ of size $k$, a rendition $\rho'$ of $(G,\Omega)$ in $\Sigma^\star$, and a unique vortex $c_0'\in C(\rho')$ satisfying items (i)-(vi) of \autoref{vortexcrosscap}.
In particular, $\Qcal_4$ is disjoint from $\sigma(c_0')$, and the vortex society $(G_1,\Omega_1)$ of $c_0'$ in $\rho'$ has a cylindrical rendition $\rho_1=(\Gamma_1,\Dcal_1,c_1)$ with nest $\Ccal'=(C_1',...,C_{s'}')$.


\medskip
\noindent{\bf Step 5: Finding a contradiction.}
We apply \autoref{main1} to $(G_1,\Omega_1)$ with the cylindrical rendition $\rho_1$, nest $\Ccal'=(C_1',...,C_{s'}')$, the truncation $\Pcal''$ of $\Pcal'$ in $(G_1,\Omega_1)$ for $\rho_1$, and $r=k+1$.
We can do so because $s',t'\ge f_{\ref{main1}}(k,k+1)$.
There are three cases.

\noindent{\bf Case 1: There is a long-jump transaction $\Qcal_5$ of order $k$ in $(G_1,\Omega_1)$ that is coterminal with $\Pcal''$ up to level $C'_{f_{\ref{main1}}(k,k+1)}$.}
By items (i)-(vi) of \autoref{vortexcrosscap} and given that $\Pcal''$ is the truncation of $\Pcal'$, $\Qcal_5$ can extended to a
long-jump transaction of order $k$ in $(G,\Omega)$ that is coterminal with $\Pcal'$ up to level $C_{f_{\ref{main1}}(k,k+1)+7}$, and hence with $\Pcal$ up to level $C_{g(k)}$ given that $s\ge f_{\ref{main1}}(k,m_1)\ge f_{\ref{main1}}(k,k+1)+7$.

\noindent{\bf Case 2: there is a crosscap transaction $\Qcal_5$ in $(G_1,\Omega_1)$ of size $k+1$ that is coterminal with $\Pcal''$ up to level $C_{f_{\ref{main1}}(k,k+1)}'$.}
Similarly to Case 1, $\Qcal_5$ can be extended to a crosscap transaction $\Qcal_6$ of size $k+1$ in $(G,\Omega)$ that is coterminal with $\Pcal'$ up to level $C_{f_{\ref{main1}}(k,k+1)+7}$.
In particular, $\Qcal_4$ and $\Qcal_6$ are disjoint, are coterminal to $\Pcal$ up to level $C_{f_{\ref{main1}}(k,k+1)+7}$, $\Qcal_4$ has size $k$ and $\Qcal_6$ has size $k+1$.
Therefore, there is a $(k,k+1)$-klein transaction $\Qcal_4\cup\Qcal_6$ that is coterminal with $\Pcal$ up to level $C_{f_{\ref{main1}}(k,k+1)+7}$.
But then, by \autoref{allinone}, there is a $k$-long-jump transaction that is coterminal with $\Pcal$ up to level $C_{f_{\ref{main1}}(k,k+1)+7+k}$.
(i) thus holds given that $f_{\ref{main1}}(k,m_1)\ge f_{\ref{main1}}(k,k+1)+7+k$.

\noindent{\bf Case 3: $(G_1,\Omega_1)$ has a rendition $\rho_1'$ in the disk of breadth at most $k-1$ and depth at most $f_{\ref{main1}}(k,k+1)$,} such that the closure of the vortex cells in $\rho_1'$ are pairwise disjoint.
Hence, $(G,\Omega)$ has a rendition in $\Sigma^\star$ of breadth at most $k-1$ and depth at most $f_{\ref{main1}}(k,k+1)$ by restricting $\rho_1$ to $\Delta\setminus(\mathsf{int}(c_0'))$ and using $\rho_1'$ in $c_0'$.
Trivially, the closure of the vortex cells in this rendition are pairwise disjoint.
This contradiction completes the proof.
\end{proof}

\section{From societies to a local structure theorem}
\label{sec_loc_to_glob}

In this section, we will prove our local structure theorem using the main result of previous section (\autoref{mainter}).
To apply \autoref{mainter}, we need to find a cylindrical rendition in the input graph $G$.
We do so using the Flat Wall theorem (which originates from \cite{RobertsonS95bthed}) which find a grid-like structure (a \emph{wall}) with a vortex-free rendition, called a \emph{flat wall}, in a graph $G$ that excludes a big clique as a minor, after the removal of a small (apex) set $A\subseteq V(G)$.
More specifically, we prove in \autoref{sec_flat} a new version of the flat wall theorem (\autoref{lem:FWt}), where we exclude a long-jump grid instead of a clique as a minor, which allows us to find a flat wall without needing to remove an apex set.
It is then enough to make a vortex out of everything that is not part of the flat wall to obtain a cylindrical rendition.
In \autoref{sec:tangles} we thus put everything together and prove that if $G$ contains no big long-jump grid as a minor, then $G$ has a $\Sigma$-decomposition $\delta$ of small breadth and depth, where $\Sigma$ is the projective plane (\autoref{mainresult}).
Additionally, we add another property on $\delta$ concerning \emph{tangles}, that we define there.

\subsection{Finding a flat wall}\label{sec_flat}
In the section, we adapt the Flat Wall theorem \cite{kawarabayashi2018anewp,RobertsonS95bthed} to long-jump grids (\autoref{th:FWt}).
The Flat Wall theorem essentially states that, given $k,r\in\Nbbb$ and a big enough wall $W$ in a graph $G$, either $G$ contains a $K_k$ minor or there is a set $A$ such that an $r$-subwall of $W$ is a flat wall of $G-A$.
In this section, we essentially prove that given given $k,r\in\Nbbb$ and a big enough wall $W$ in a graph $G$, either $G$ contains a $\mathscr{J}_k$ minor or there is a $r$-subwall of $W$ that is a flat wall of $G$.

Let us first define walls, flat walls, and some related notions.

\paragraph{Walls.}
An \emph{$(n\times m)$-grid} is the graph $G_{n,m}$ with vertex set $[n]\times[m]$ and edge set \[\{(i,j)(i,j+1) \mid i\in[n],j\in[m-1]\}\cup\{(i,j)(i+1,j) \mid i\in[n-1],j\in[m]\}.\]
We call the path where vertices appear as $(i,1),(i,2),\dots, (i,m)$ the \emph{$i$th row} and the path where vertices appear as $(1,j),(2,j),\dots, (n,j)$ the \emph{$j$th column} of the grid.
An \emph{elementary $(k,\ell)$-wall}~$W_{k,\ell}$ for~${k,\ell \geq 3},$ is obtained from the ${(k\times 2\ell)}$-grid $G_{k,2\ell}$ by deleting every odd edge in every odd column and every even edge in every even column, and then deleting all degree-one vertices.
The \emph{rows} of~$W_{k,\ell}$ are the subgraphs of~$W_{k,\ell}$ induced by the rows of~$G_{k,2\ell},$ while the \emph{$j$th column} of~$W_{k,\ell}$ is the subgraph induced by the vertices of columns~${2j-1}$ and~${2j}$ of~$G_{k,2\ell}.$
We define the perimeter of $W_{k,\ell}$ to be the subgraph induced by~${\{ (i,j) \in V(W_{k,\ell}) \mid j \in \{1,2,2\ell,2\ell-1\} \text{ and } i \in [k] \textnormal{, or } i \in \{1,k\} \text{ and } j \in [2\ell] \}}.$
A \emph{$(k,\ell)$-wall}~$W$ is a graph isomorphic to a subdivision of~$W_{k,\ell}.$
The vertices of degree three in $W$ are called the \emph{branch vertices}.
In other words, $W$ is obtained from a graph $W'$ isomorphic to $W_{k,\ell}$ by subdividing each edge of~$W'$ an arbitrary (possibly zero) number of times.
We define rows and columns of $(k,\ell)$-walls analogously to their definition for elementary walls.
A \emph{(elementary) $k$-wall} is a (elementary) $(k,k)$-wall.
A \emph{wall} is a $(k,\ell)$-wall for some~$k,\ell.$
The vertices
in the perimeter of an elementary $r$-wall that have degree two are called \emph{pegs},
while the vertices $(1,1), (2,r), (2r-1,1), (2r,r)$ are called \emph{corners} (notice that the corners are also pegs).

An $h$-wall $W'$ is a \emph{subwall} of some $k$-wall $W$ where $h\leq k$ if every row (column) of $W'$ is contained in a row (column) of $W.$

Notice that, as $k \geq 3,$ an elementary $k$-wall is a planar graph that has a unique (up to topological isomorphism) embedding in the plane $\mathbb{R}^{2}$ such that all its finite faces are incident to exactly six edges. The perimeter of an elementary $r$-wall is the cycle bounding its infinite face, while the cycles bounding its finite faces are called \emph{bricks}. A cycle of a wall $W,$ obtained from the elementary wall $W',$ is the \emph{perimeter} of $W$, denoted by $D(W)$, if its branch vertices are the vertices of the perimeter of $W'.$

The following is a statement of the Grid Theorem.
While we will not explicitly use the Grid Theorem, we will, later on, make use of the existence of a function that forces the existence of a large wall in a graph with large enough treewidth. 

\begin{proposition}[Grid Theorem \cite{RobertsonS86exclu,chuzhoy2021towards}]\label{@reestablish}
	There exists a universal constant $c\geq 1$ such that for every $k\in\mathbb{N}$ and every graph $G,$ if $\mathsf{tw}(G)\geq ck^{10},$ then $G$ contains the $(k\times k)$-grid as a minor. 
	\end{proposition}
Notice that any graph that contains a $(2k\times 2k)$-grid as a minor contains a $k$-wall as a subgraph.

\paragraph{Tilts.}
The \emph{interior} of a wall $W$ is the graph obtained from $W$ if we remove from it all edges of $D(W)$ and all vertices of $D(W)$ that have degree two in $W$.
Given two walls $W$ and $\tilde{W}$ of a graph $G$, we say that $\tilde{W}$ is a \emph{tilt} of $W$ if $\tilde{W}$ and $W$ have identical interiors.

\paragraph{Layers of a wall.} 
Let $r \in \Nbbb^{\mathsf{even}}.$ The \emph{layers} of an $r$-wall $W$ are recursively defined as follows.
The first layer of $W$ is its perimeter.
For $i = 2, \ldots, r/2,$ the $i$-th layer of $W$ is the $(i-1)$-th layer of the subwall $W'$ obtained from $W$ after removing from $W$ its perimeter and removing recursively all occurring vertices of degree one.

\paragraph{Central walls.} 
Let $r,q\in\Nbbb$ be with the same parity and $q\le r$, and let $W$ be an $r$-wall.
We define the \emph{central $q$-subwall} of $W$ to be the $q$-wall obtained from $W$ after removing its first $(r-q)/2$ layers and all occurring vertices of degree one. 



\paragraph{Separations.}
We say that a pair $(A, B) \in 2^{V(G)} \times 2^{V(G)}$ is a \emph{separation} of a graph $G$ if $A \cup B = V(G)$ and there is no edge in $G$ between $A \setminus B$ and $B \setminus A.$
We call $|A \cap B|$ the \emph{order} of $(A, B)$.

\paragraph{Flat walls.}
Let $G$ be a graph and let $W$ be an $r$-wall of $G$, for some odd integer $r\geq 3$.
We say that a pair $(P,C)\subseteq V(D(W))\times V(D(W))$ is a {\em choice of pegs and corners for $W$} if $W$ is a subdivision of an elementary $r$-wall $\bar{W}$ where $P$ and $C$ are the pegs and the corners of $\bar{W}$, respectively (clearly, $C\subseteq P$).

We say that $W$ is a \emph{flat $r$-wall} of $G$ if there is a separation $(X,Y)$ of $G$ and a choice $(P,C)$
of pegs and corners for $W$ such that:
\begin{itemize}
	\item $V(W)\subseteq Y$,
	\item $P\subseteq X\cap Y\subseteq V(D(W))$, and
	\item if $\Omega$ is the cyclic ordering of the vertices $X\cap Y$ as they appear in $D(W)$,
	      then $(G[Y],\Omega)$ is rural.
\end{itemize}

We say that $W$ is a {\em flat wall}
of $G$ if it is a flat $r$-wall for some odd integer $r \geq 3$.

\begin{proposition}[Theorem 5, \cite{SauST21accu}]
\label{tilt}
There is an algorithm that, given a graph $G$, a flat wall $W$ of $G$, and a subwall $W'$ of $W$, outputs, in time $\Ocal(n+m)$, a flat wall $\tilde{W}'$ that is a tilt of $W'$.
\end{proposition}

Let us prove our flat wall theorem when we exclude a long-jump grid as a minor.

\begin{lemma}\label{lem:FWt}
Let $k,r\geq 1$ be integers, with $r$ odd, let $G$ be a graph, and let $W$ be a $((r+3k)k+6k-4,r+8k-4)$-wall in $G$.
Then one of the following holds.
\begin{itemize}
\item Either $G$ contains $\mathscr{J}_k$ as a minor, or
\item there is a flat wall $W'$ in $G$, where $W'$ is the tilt of an $r$-subwall of $W$.
\end{itemize}
\end{lemma}

\begin{figure}[ht]
\begin{center}
\includegraphics[scale=1]{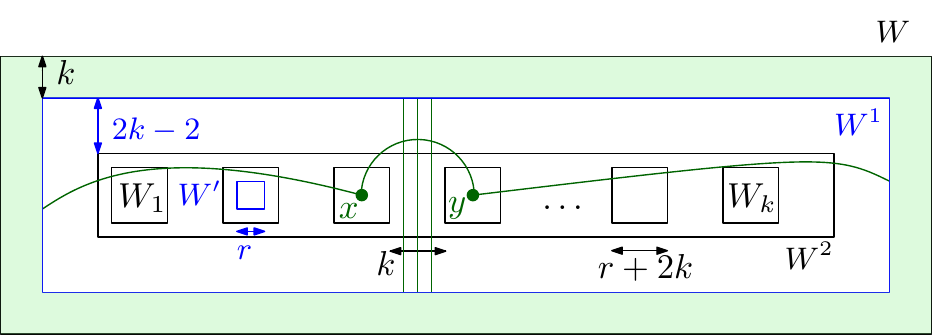}
\end{center}
  \caption{The walls $W_1,...,W_k$ in $W$. In darkgreen is depicted the $\mathscr{J}_k$ minor of \autoref{ca}.}
\label{figwallsinwall}
\end{figure}

\begin{proof}
Let $W^1$ be the $((r+3k)k+4k-4,r+6k-4)$-wall obtained by removing the first $k$ layers of $W$.
Let $W^2$ be the $((r+3k)k,r+2k)$-wall obtained by removing the first $2k-2$ layers of $W^1$.
Let $a=(r+3k)k$,and $b=r+2k$.
Let $P_1,...,P_b$ (resp. $Q_1,...,Q_a$) be the horizontal (resp. vertical) paths of $W^2$.
Let $I_1',...,I_k'$ be intervals of length $r+3k$ with union $[a]$.
For $i\in[k]$, let $I_i$ be obtained from $I_i'$ by deleting the last $k$ elements.
Thus, $|I_i|=r+2k$.
For $i\in[k]$, we define $W_i$ as the $(r+2k)$-subwall of $W^2$ obtained by removing all paths $Q_j$ for $j\in[a]\setminus I_i$.
Refer to \autoref{figwallsinwall} for an illustration.
Observe that the $k$ first layers of $W$ do not contain vertices from the $W_i$, $i\in[k]$, and, hence, form, along with horizontal and vertical subpaths, a railed nest $(\Ccal,\Pcal)$ of order $(k,2a+2b)$ (in light green in \autoref{figwallsinwall}).

Let $B$ be a $W$-bridge in $G$ with at least one attachment in $V(W_i)$.
Let $B'$ be obtained from $B$ by deleting all its attachments that do not belong to $V(W_i)$.
Let $H_i$ be the graph obtained from the union of $W_i$ and all graphs $B'$ defined as above.

\begin{claim}\label{ca}
If there are indices $i$ and $j$ with $i<j$ such that $H_i$ and $H_j$ share a vertex, then $G$ contains $\mathscr{J}_k$ as a minor.
\end{claim}
\begin{cproof}
Assume that there are indices $i$ and $j$ with $i<j$ such that $H_i$ and $H_j$ share a vertex.
Then there exists a $W$-bridge with an attachment $x\in V(W_i)$ and $y\in V(W_j)$.
Therefore there exists a $W$-path $P$ in $G$ with endpoints $x$ and $y$.
Then, the union of $P$, the $k$ vertical paths $Q_l$ for $l\in I_i'\setminus I_i$, the railed nest $(\Ccal,\Pcal)$, and an appropriate horizontal subpath from $x$ (resp. $y$) to $\Ccal$ in $W$, contains $\mathscr{J}_k$ as a minor.
\end{cproof}

In this case, we directly conclude, so we may assume that the subgraphs $H_i$ are pairwise disjoint.
Let $C_i$ be the perimeter of $W_i$ for $i\in[k]$.
Let $N_i=V(C_i)\cap N_G(V(G)\setminus V(H_i))$.
Let $\Omega_i$ be the cyclic ordering of the vertices of $N_i$ as they appear in $C_i$.
If, for each $i\in[k]$, $(H_i,\Omega_i)$ contains a cross, then the nested-crosses grid $\mathscr{NC}_k^{k+1}$ is a minor of $G$.
By \autoref{allinone2}, $\mathscr{J}_k$ is hence a minor of $G$.
In this case, we can directly conclude, so we may assume that there exists an $i\in[k]$ such that $(H_i,\Omega_i)$
do not have a cross.
Then, by \autoref{thm:crossreduct}, the society $(H_i,\Omega_i)$ is rural. 
It implies that $W_i$ is a flat wall of $H_i$.

We now want to obtain a flat wall of $G$.
Let $W''$ be the $r$-wall obtained from $W_i$ by deleting the $k$ first layers of $W_i$.
By \autoref{tilt}, there is a flat wall $W'$ of $H_i$ such that $W'$ is a tilt of $W''$.
Let $(X',Y)$ be a separation of $H_i$ and $(P,C)$ be a choice of pegs and corners certifying that $W'$ is flat wall of $H_i$, with $V(W)\subseteq Y$.
Let $X=X'\cup (V(G)\setminus V(H_i))$.
Let us show that $(X,Y)$ along with $(P,C)$ certifies that $W'$ is a flat wall of $G$.
For this, it is enough to prove the following.
\begin{claim}
$(X,Y)$ is a separation of $G$ or $G$ contains $\mathscr{J}_k$ as a minor.
\end{claim}
\begin{cproof}
Suppose that $(X,Y)$ is not a separation of $G$, and thus that there are $x\in X\setminus Y$ and $y\in Y\setminus X$ that are adjacent.
Given that $(X',Y)$ is a separation of $H_i$, it implies that $x\notin V(H_i)$ and $y\in V(H_i)$.

Assume that $y\in V(W_i)\setminus V(W')$.
Let $D$ be the perimeter of $W'$.
Then there is a path from $y$ to $V(C_i)$ disjoint from $V(D)$, contradicting the fact that $V(C_i)\subseteq X'$, $y\in Y$, and $X'\cap Y\subseteq V(D)$.
Hence, $y\notin V(W_i)\setminus V(W')$.

The edge joining $x$ and $y$ belongs to a $W$-bridge $B$ of $G$, and hence, $x$ is an attachment of $B$ outside $W_i$.
Thus, there is a $W$-path with one endpoint $x$ and the other $y'\in V(W')$.
If $x\in V(W^1)$, then, by a similar argument to \autoref{ca}, $G$ contains $\mathscr{J}_k$ as a minor. 
Otherwise, $x\in V(W)\setminus V(W^1)$.
But then, there are $k$ paths from $W_i\setminus V(W')$ plus $2k-2$ paths from $W^1\setminus W^2$ separating $x$ from $y'$.
Hence, $G$ contains the alternative jump grid $\hat{\mathscr{J}}_{2k}^{3k-2}$, and thus $\mathscr{J}_k$ by \autoref{allinone2} as a minor. 
\end{cproof}
Thus, the separation $(X,Y)$ witnesses that $W'$ is a flat wall in $G$.
\end{proof}

Given that a $(a,b)$-wall is a subset of an $\Ocal(\sqrt{ab})$-wall, \autoref{lem:FWt} immediately implies the following.

\begin{theorem}\label{th:FWt}
Let $k,r\geq 1$ be integers, with $r$ odd, let $G$ be a graph, and let $W$ be a $f_{\ref{th:FWt}}(k,r)$-wall in $G$, where $f_{\ref{th:FWt}}(k,r)=\Ocal(\sqrt{k}(r+k))$.
Then one of the following holds.
\begin{itemize}
\item Either $G$ contains $\mathscr{J}_k$ as a minor, or
\item there is a flat wall $W'$ in $G$, where $W'$ is the tilt of an $r$-subwall of $W$.
\end{itemize}
\end{theorem}

\subsection{Tangles}\label{sec:tangles}

To create the tree decomposition of the global structure theorem, we require tangles, first introduced in~\cite{RobertsonS91obst}.

%

\paragraph{Tangles.} Let $G$ be a graph and $k$ be a positive integer. 
We denote by $\mathcal{S}_{k}$ the collection of all separations $(A, B)$ of order less than $k$ in $G$.
An \emph{orientation} of $\mathcal{S}_{k}$ is a set $\mathcal{O}$ such that for all $(A, B) \in \mathcal{S}_{k}$ exactly one of $(A, B)$ and $(B, A)$ belongs to $\mathcal{O}.$
A \emph{tangle} of order $k$ in $G$ is an orientation $\mathcal{T}$ of $\mathcal{S}_{k}$ such that for all $(A_{1}, B_{1}), (A_{2}, B_{2}), (A_{3}, B_{3}) \in \mathcal{T}$, 
we have $A_{1} \cup A_{2} \cup A_{3} \neq V(G).$

Let $\mathcal{T}'\subsetneq\mathcal{T}$ be a tangle which is {properly contained} in the tangle $\mathcal{T}.$
We say that $\mathcal{T}'$ is a \emph{truncation} of~$\mathcal{T}.$
This implies in particular that the order of $\Tcal'$ is smaller that the order of $\Tcal$.

\paragraph{Tangles  of  walls.}
Let $W$ be a $k$-wall in a graph $G$ and $(A, B)$ be a separation of order strictly less than $k$.
Then exactly one of $A\setminus B$ or $B\setminus A$ contains a row and a column of $W.$
Let $\mathcal{T}_{W}$ be the orientation of $\mathcal{S}_{k}$ where $(A, B) \in \mathcal{T}_{W}$ if and only if 
$B\setminus A$ contains a row and a column of $W.$
Then it is easy to observe that $\mathcal{T}_{W}$ is a tangle, which we call the \emph{tangle of $W$}. 

\paragraph{Flat walls in a $\Sigma$-decomposition.}
Let $W$ be a wall in a graph $G$.
We say that $W$ in \emph{flat} is a $\Sigma$-decomposition $\delta$ of $G$ if there exists a $\delta$-aligned disk $\Delta$ such that 
\begin{itemize}
\item $\pi(N(\delta)\cap \bd(\Delta))\subseteq V(D(W))$,
\item if $S$ is the collection of corners and branch vertices of $W$ that are not in $\ground(\delta)$, then, for any $c\in C(\delta)$, there exists at most one $v\in S$ such that $v \in V(\sigma(c))$,
\item no cell $c\in C(\delta)$ with $c\subseteq \Delta$ is a vortex, and 
\item $W-V(D(W))$ is a subgraph of $\bigcup_{c\subseteq \Delta}\sigma(c)$.
\end{itemize}

In the next result, we reformulate the local structure theorem in the form we need to prove the global structure theorem in the next section.

\begin{theorem}\label{mainresult}
  There exist functions $f_{\ref{mainresult}}\colon\mathbb{N}^2\to\mathbb{N}$, $d_{\ref{mainresult}}\colon\mathbb{N}\to\mathbb{N}$ such that, for every choice of non-negative integers $k,r$ with odd $r\ge3$ and every graph $G$ with an $f_{\ref{mainresult}}(k,r)$-wall $W$, one of the following holds
  \begin{enumerate}
    \item[\rm{(i)}] $G$ contains the long-jump grid of order $k$ as a minor, or
    \item[\rm{(ii)}]  $G$ has a $\Sigma$-decomposition $\delta$ of breadth at most $k-1$ and depth at most $d_{\ref{mainresult}}(k)$ such that the closures of the vortex cells in $\delta$ are pairwise disjoint.
    Moreover, $\Sigma$ is either the sphere or the projective plane and there exists a wall of height at least $r$ which is flat in $\delta$ and whose tangle is a truncation of the tangle of $W$. 
  \end{enumerate}
Moreover, $f_{\ref{mainresult}}(k,r)=\max\{2^{\Ocal(k\log k)},\Ocal(\sqrt{k}(r+k))\}$ and $d_{\ref{mainresult}}(k)=2^{\Ocal(k\log k)}$.
\end{theorem}

\begin{proof}
We set $m_1=(6k+1)f_{\ref{main1}}(k,k+1)+3k(4k+9)$, $\ell=f_{\ref{main1}}(k,m_1)+f_{\ref{main1}}(k,k+1)+7$, and $d_{\ref{mainresult}}(k)=f_{\ref{main1}}(k,m_1)$. We also set $r_2$ to be the smallest odd integer bigger than $\max\{r,f_{\ref{main1}}(k,m_1)/4\}$, $r_1=2(\ell+1)+r_2$, and $f_{\ref{mainresult}}(k,r)=f_{\ref{th:FWt}}(k,r_1)$.
By \autoref{th:FWt}, given that $f_{\ref{mainresult}}(k,r)=f_{\ref{th:FWt}}(k,r_1)$, either $G$ contains $\mathscr{J}_k$ as a minor, in which case we conclude, or there is a flat wall $W_1$ in $G$, where $W_1$ is the tilt of an $r_1$-subwall of $W$.
In the latter case, let $W_2$ be the central $r_2$-subwall of $W_1$. 
Given that $W_1$ is a flat wall in $G$, there is a separation $(X,Y)$ of $G$ and a cyclic ordering $\Omega$ of the vertices of $X\cap Y$ witnessing the flatness of $W_1$.
In particular, $(G[Y],\Omega)$ has a vortex-free rendition $\rho=(\Gamma,\Dcal)$ in a disk $\Delta$.
Let $T$ be the track of the perimeter of $W_2$.
$\mathbb{S}_0\setminus T$ is the union of two disks whose closure is $\Delta_1$ and $\Delta_2$ respectively.
We assume without loss of generality that $W_2- V(D(W_2))$ is a subgraph of $\bigcup_{c\subseteq \Delta_2}\sigma(c)$.
Notice that $\pi(N(\rho)\cap \bd(\Delta_2))\subseteq V(D(W_2))$, given that $T$ is the track of $D(W_2)$ and the boundary of $\Delta_2$.
Let $G'$ be the graph obtained from $G$ after removing, for each $c \in C(\rho)$ with $c\subseteq \Delta_2$, the edges of $\sigma_{\rho}(c)$ and the vertices of $\sigma_{\rho}(c)-\pi_\rho(T\cap N(\rho))$.
Let $\Omega'$ be the cyclic ordering of the vertices of $\pi_\rho(T\cap N(\rho))$ with the cyclic order induced by the perimeter of $W_2$. 
We construct a cylindrical rendition $\rho'=(\Gamma',\Dcal',c_0)$ of $(G',\Omega')$ in the disk $\Delta_1$ as follows.
We set $\Dcal'=\{c_0\}\cup\{c\in\Dcal\mid c\subseteq\Delta_1\}$, where $c_0=\mathbb{S}_0\setminus\Delta$ with $\sigma(c_0)=G[X]$ and $\pi_{\rho'}(\tilde{c}_0)=X\cap Y$.
We define $\Gamma'$ to be obtained from the restriction of $\Gamma$ to $\Delta_1$ by drawing $G[X]-Y$ arbitrarily in $c_0$ and adding the appropriate edges with $\pi_{\rho'}(\tilde{c}_0)$.
Given that, for $i\in[2,\ell+1]$, the $i$-th layer of $W_1$ is drawn in $\Delta_1\setminus c_0$, $\rho'$ is a cylindrical rendition of $(G',\Omega')$ with a railed nest ($\Ccal,\Pcal)$ in $\rho$ around $c_0$ of order $(\ell,4r_2)$.
See \autoref{fig:reversewall} for an illustration.

\begin{figure}[h]
\center
\includegraphics[scale=0.7]{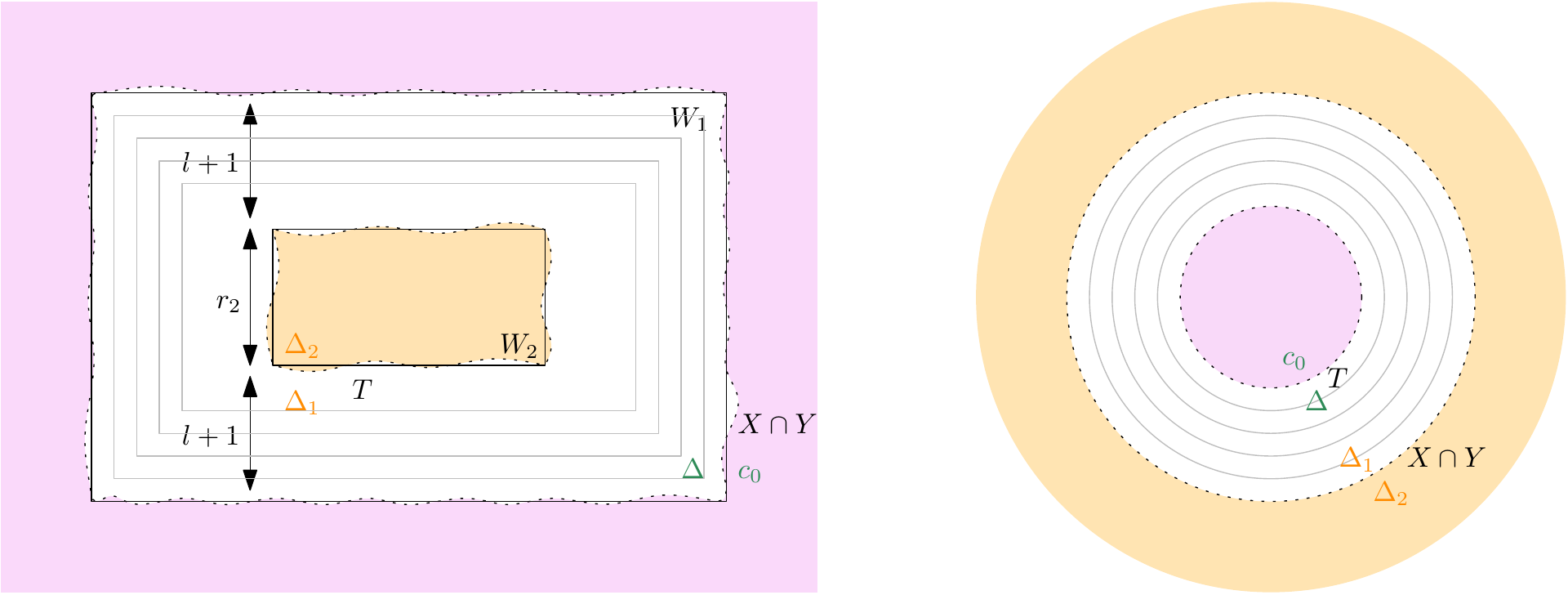}
\caption{Illustration for the proof of \autoref{mainresult}. $\rho$ is the vortex-free rendition in the disk $\Delta$ witnessing that $W_1$ is a flat wall. $\rho'$ is the cylindrical rendition composed of all but the orange disk $\Delta_2$, where the vortex $c_0$ is composed of all but the disk $\Delta$.}
\label{fig:reversewall}
\end{figure}

Given that $\ell=f_{\ref{main1}}(k,m_1)+f_{\ref{main1}}(k,k+1)+7$ and that $4r_2\ge f_{\ref{main1}}(k,m_1)$, by \autoref{mainter}, either (a) $(G',\Omega')$ contains a $k$-long-jump transaction that is coterminal with $\Pcal$ up to level $f_{\ref{main1}}(k,m_1)$, or (b) $(G',\Omega')$ has a rendition in the plane of depth at most $f_{\ref{main1}}(k,m_1)$ or in the projective plane of depth at most $f_{\ref{main1}}(k,k+1)$, and of breadth at most $k-1$ in both cases, such that the closures of the vortex cells are pairwise disjoint.
Given that $\ell-f_{\ref{main1}}(k,m_1)+1\ge k$, it implies in case (a) that $\mathscr{J}_k$ is a minor of $G$, so we can conclude.
In case (b), by combining the restriction of $\rho$ to $\Delta_2$ and $\rho'$, we get a $\Sigma$-decomposition $\delta$ of breadth at most $k-1$, where either $\Sigma$ is the sphere and $\delta$ has depth at most $f_{\ref{main1}}(k,m_1)=d_{\ref{mainresult}}(k)$, or $\Sigma$ is the projective plane and $\delta$ has depth at most $f_{\ref{main1}}(k,k+1)\le d_{\ref{mainresult}}(k)$, such that the closure of the vortex cells are pairwise disjoint.

Given that $W_2$ is a subwall of $W$, its tangle is obviously a truncation of the tangle of $W$.
Moreover, from the flatness of $W_1$, we easily derive that $W_2$ is flat in $\delta$.
Hence the result.
\end{proof}

If we want to exclude both a long-jump grid and a crosscap grid, then we get a similar result.

\begin{theorem}\label{mainresult_pl}
 There exist functions $f_{\ref{mainresult_pl}}\colon\mathbb{N}^3\to\mathbb{N}$, $d_{\ref{mainresult_pl}}\colon\mathbb{N}^2\to\mathbb{N}$ such that, for every choice of non-negative integers $k,c,r$ with odd $r\ge3$ and every graph $G$ with an $f_{\ref{mainresult_pl}}(k,r)$-wall $W$, one of the following holds
  \begin{enumerate}
    \item[\rm{(i)}] $G$ contains the long jump grid of order $k$ as a minor, or
    \item[\rm{(ii)}]  $G$ contains the crosscap grid of order $c$ as a minor, or
    \item[\rm{(iii)}] $G$ has a $\mathbb{S}_0$-decomposition $\delta$ of breadth at most $k-1$ and depth at most $d_{\ref{mainresult_pl}}(k,c)$ such that the closures of the vortex cells in $\delta$ are pairwise disjoint.
    Moreover, there exists a wall of height at least $r$ which is flat in $\delta$ and whose tangle is a truncation of the tangle of $W$. 
  \end{enumerate}
Moreover, $f_{\ref{mainresult_pl}}(k,r)=\max\{2^{\Ocal(k\log (k\cdot c))},\Ocal(\sqrt{k}(r+k))\}$ and $d_{\ref{mainresult_pl}}(k)=2^{\Ocal(k\log(k\cdot c))}$.
\end{theorem}

\begin{proof}
The proof is exactly the same as for \autoref{mainresult}, but using \autoref{main1} instead of \autoref{mainter}, and setting $\ell=d_{\ref{mainresult_pl}}(k)=f_{\ref{main1}}(k,c)$, $r_2$ to be the smallest integer bigger than $\max\{r,f_{\ref{main1}}(k,c)/4\}$, $r_1=2(\ell+1)+r_2$, and $f_{\ref{mainresult_pl}}(k,c,r)=f_{\ref{th:FWt}}(k,r_1)$.
\end{proof}

\section{The global structure theorem}\label{sec_global}

In this section, we prove our global structure theorem (\autoref{upper_b}) using the local structure theorem (\autoref{mainresult}) proved in the previous section.
For now, we know that if $G$ contains no big long-jump grid as a minor, then $G$ has a $\Sigma$-decomposition $\delta$ of small breadth and depth, where $\Sigma$ is the projective plane.

Instead of the depth, the parameter we actually care about on vortices is their \emph{width}.
A graph with a $\Sigma$-decomposition of breadth and width at most $k$ (with some additional properties) is said to be \emph{$k$-almost embeddable} in $\Sigma$.
We prove in \autoref{corr_proj_global} in \autoref{sec_global} that, if $G$ excludes a long-jump grid as minor, then it has a tree decomposition $\Tcal=(T,\beta)$ such that the torso at each node has an almost embedding on the projective plane of small breadth and width.
Imagine that all of $V(G)$, and thus the $\Sigma$-decomposition $\delta$ of small breadth and depth, is originally in the root $r$ of $T$.
Then, essentially, most of what is inside each vortex $c_v$ is pushed out to a child $t_v$ of $r$.
Then, in the root, we obtain that each vortex has now small width instead of small depth, and it remains to recurse on the children of $r$.
This is the very standard ``local to global'' approach, used for instance in \cite{kawarabayashi2020quickly,diestel2016graph,DiestelKMW12onth,ThilikosW24kill,DvorakT14list}.
However, there is a catch here.
To the authors' knowledge, the structure theorem proved in this paper is the first structure theorem with vortices but \textsl{no apices} (note that the structure theorem of \cite{RobertsonS91} for singly-crossing graphs has no apices, but also no vortices).
Indeed, usually, one proves that, the torso of $\Tcal$ at each node is ``almost embeddable'' in some surface after the removal of a bounded set of vertices (the \emph{apices}), with sometimes some additional properties.
These apices are very useful, in particular to hide among them the set $X$ of vertices in the intersection of $r$ and $t_v$.
By definition of the torso, we need to make a clique out of $X$, which we can do here safely without destroying the almost embeddability of the rest of the bag.
In our case however, we have no apices. Hence, we develop in \autoref{thm_projective_global} a new technique to go from the local to the global structure theorem \textsl{in the absence of apices}.
This technique will be explained more in detail later and, as we already mentioned, it works because the surface we consider is either the plane or the projective plane, but that it would not work on other surfaces given that we do not know how to handle cycles that do not bound a disk.

Finally, in \autoref{sec_id_global}, we easily derive from \autoref{corr_proj_global} our global structure theorem (\autoref{upper_b}, \autoref{th_proj_global}) in terms of identifications: the set of vertices in vortex cells has bounded bidimensionality, so we can identify each vortex separately to obtain an embedding in the projective plane.

\subsection{From local to global}\label{subsec:localtoglobal}

The proof of the global structure theorems, providing the upper bounds for our main results, follows a well-established strategy, that was formalized in \cite{DiestelKMW12onth}.
This strategy allows us to prove a slightly stronger ``rooted'' version of the global structure theorem (\autoref{thm_projective_global}).
The main advantage of this stronger version is that it allows for a straightforward proof by induction.

Let us first give some definition before sketching how this well-established strategy usually works.

%
%
\paragraph{Rooted tree decompositions and adhesion.}

Let $G$ be a graph.
A triple $\mathcal{T} = (T,r,\beta)$ is called a \emph{rooted tree decomposition} of $G$ if $(T,\beta)$ is a tree decomposition of $G$ and $r\in V(T)$.

For each $t\in V(T)$, we define the \emph{adhesions of $t$} as the sets in $\{ \beta(t)\cap\beta(d) \mid d\text{ adjacent with }t \}$ and the maximum size of them is called the  \emph{adhesion of $t$}.
The \emph{adhesion} of $\mathcal{T}$   is the maximum adhesion of a node of $\mathcal{T}$.


\paragraph{Linear decompositions.}
A \emph{linear decomposition} of a society $(G,\Omega)$ is a labeling $v_1,\dots, v_\ell$ of $V(\Omega)$, such that $v_1,\dots,v_n$ occur in that order on $\Omega$, and subsets $(X_1,\dots,X_\ell)$ such that 
\begin{itemize}
\item for each $i\in[\ell]$, $v_i\in X_i\subseteq V(G)$,
\item $\bigcup_{i\in[\ell]}X_i=V(G)$ and, for each $uv\in E(G)$, there exists $i\in[\ell]$ such that $\{u,v\}\in X_i$, and
\item for each $x\in V(G)$, the set $\{i\mid x\in X_i\}$ is an interval in $[\ell]$.
\end{itemize}
The \emph{width} of a linear decomposition is $\max_{i\in[\ell]}|X_i|$.

It is not hard to see that every society with a linear decomposition of adhesion at most $d$ has depth at most $2d$. 
For the converse, we have the following result.

\begin{proposition}[\!\!\cite{kawarabayashi2020quickly,RobertsonS90disj}]\label{prop_width}
Let $d\in\Nbbb$.
Every society of depth at most $d$ has a linear decomposition of adhesion at most $d$.
\end{proposition}

\paragraph{Almost embeddings.}
Let $G$ be a graph and $\Sigma$ be a surface.
An \emph{almost embedding} of $G$ in $\Sigma$ of \emph{breadth $b$} and \emph{width $d$} is a $\Sigma$-decomposition $\delta$ of $G$ such that there is a set $C_0\subseteq C(\delta)$ of size at most $b$ containing all vortex cells of $\delta$ such that:
\begin{itemize}
\item no vertex of $G$ is drawn in the interior of a cell of $C(\delta)\setminus C_0$ and, 
\item for each vortex cell, there exists a linear decomposition of its vortex society of width at most $d$ (and for each non-vortex cell $c\in C_0$, $|V(\sigma(c))|\le d$).
\end{itemize}
$C_0$ is called the \emph{vortex set} of $\delta$.

\paragraph{Well-linked sets.}
Let $\alpha \in [2/3, 1)$.
Moreover, let $G$ be a graph and $X \subseteq V(G)$ be a vertex set. 
A set $S \subseteq V(G)$ is said to be an \emph{$\alpha$-balanced separator} for $X$ if for every component $C$ of $G - S$ it holds that $|V(C) \cap X| \leq \alpha|X|$. 
Let $k$ be a non-negative integer.
We say that $X$ is a \emph{$(k, \alpha)$-well-linked set} of $G$ if there is no $\alpha$-balanced separator of size at most $k$ for $X$ in $G$.
\medskip

Given a $(k, \alpha)$-well-linked set $X$ of $G$ we define $$\mathcal{T}_{X} \coloneqq \{ (A, B) \in \mathcal{S}_{k+1}(G) \mid |X \cap B| > \alpha|X| \}.$$ 
It is not hard to see that $\mathcal{T}_{S}$ is a tangle of order $k+1$ in $G$.

We need an algorithmic way to find, given a well-linked set, a large wall whose tangle is a truncation of the tangle of the well-linked set.
This is done in \cite{ThilikosW2025excluding} by algorithmatising a proof of Kawarabayashi, Wollan, and Thomas found in \cite{kawarabayashi2020quickly}.

\begin{proposition}[Thilikos and Wiederrecht \cite{ThilikosW2025excluding} (see Theorem 3.4.)]\label{thm:algogrid}
Let $k\geq 3$ be an integer and $\alpha\in [2/3,1)$.
There exist universal constants $c_1,c_2\in\mathbb{N}\setminus\{ 0\}$, and an algorithm that, given a graph $G$ and a $(c_1k^{20},\alpha)$-well-linked set $X\subseteq V(G)$ computes in time $2^{\mathcal{O}(k^{c_2})}|V(G)|^2|E(G)|\log(|V(G)|)$ a $k$-wall $W\subseteq G$ such that $\mathcal{T}_W$ is a truncation of $\mathcal{T}_X$.
\end{proposition}

The rooted version of the structure theorem is usually stated along the lines of: \emph{Let $G,H$ be graphs and $X\subseteq V(G)$ be a set of small size. Then either $G$ contains $H$ as a minor, or there is a rooted tree decomposition $(T,\beta,r)$ of $G$ such that the torso at each node has an almost embedding in $\Sigma$ after removing a small apex set $A$, and such that $X\subseteq \beta(r)$.}

Obviously, if $X=\emptyset$, then this is the global structure theorem.
$X$ essentially corresponds to vertices inherited from a parent bag in the induction, from which we will make a clique to obtain the torso.
The proof of such a result goes as follows.
If there is a balanced separator $S$ in $G$ for $X$, then we inductively find a rooted tree decomposition $\Tcal_C=(T_C,\beta_C,t_C)$ for each connected component $C$ of $G-S$ with $X_C=X\cap C\cup S$.
Then $\Tcal=(T,\beta,r)$ is the tree decomposition where the children of $r$ are the nodes $t_C$, $\beta(r)=S\cup X$, and the restriction of $\Tcal$ to the subtree rooted at $t_C$ is $\Tcal_C$.
Otherwise, $X$ is a well-linked set from which we can derive a wall (\autoref{thm:algogrid}) whose tangle agrees with the tangle of $X$.
Then, we can apply the local structure theorem (here \autoref{mainresult}) to find an apex set $A$ (for us $A=\emptyset$) such that $G-A$ has a $\Sigma$-decomposition $\delta$ of small breadth and depth at most $d$.
For each non-vortex cell $c$, we recurse on $G_c^1=\sigma(c)$ with $X_c^1$ that is the union of $X\cap\sigma(c)$ and the boundary $A_c^1$ of the cell, to find a tree decomposition $\Tcal_c^1$.
For each vortex cell $c$, we fix a linear decomposition $(Y_1,\dots,Y_\ell)$ of adhesion at most its depth.
Then, for each $i\in[\ell]$, we recurse on $G_c^i=G[Y_i]$ with $X_c^i$ being the union of $X\cap Y_i$ and the set $A_c^i$ composed of its adhesion with its neighbors as well as the $i$th vertex of the cyclic ordering~$\Omega_c$, and we obtain a tree decomposition $\Tcal_c^i$.
Then, we put all those tree decompositions together, that we attach to the root $r$ with $\beta(r)$ that is the union of $A$, $X$, and the sets $A_c^i$, to obtain a tree decomposition $\Tcal=(T,\beta,r)$ of $G$.
It remains to prove that the torso of $\Tcal$ at each node $t$ has an almost embedding in $\Sigma$ after removing a small apex set $A$ of small depth and small width.
This is immediate for $t\in V(T)$ that is not a child of $r$.
For $r$, torsifying corresponds to making a clique out of each $X_c^1$.
For each $c,i$, let us add $X_c^i\setminus A_c^i$ to the apex set.
The size of the apex set increases by at most $|X|$.
Now, it is enough to prove that making a clique out of each $A_c^i$ does not destroy the almost embedding.
For non-vortex cells, we have $|A_c^1|\le3$, so making a clique trivially does not destroy the almost embedding.
For vortex cells, $A_c^i$ is now a bag of the linear decomposition, so the width does not increase after the torsification.
Additionally, $A_c^i$ has size at most $2d+1$, so the width of a vortex is at most $2d+1$.
Hence, we have an almost embedding of $\beta(r)$ of small breadth and width at most $2d+1$ after removing the apex set $A+X$.
Finally, for each child $t_c^i$ of $r$, after removing an apex set $B_c^i$, we already have an almost embedding of the torso of $\Tcal_c^i$ at $t_c^i$, where $\Tcal_c^i$ is a tree decomposition of $G_c^i$.
To make it an almost embedding of the torso of $\Tcal$ at $t_c^i$, we need to add all edges between the vertices of $X_c^i$ (which is the adhesion of $\beta(r)$ and $\beta(t_c^i)$), which might destroy the almost embedding.
To handle this problem, it is enough to remove $X_c^i$ from the almost embedding, that is to add $X_c^i$ to the apex set.
Hence, we have an almost embedding of the torso of $\Tcal$ at $t_c^i$ of small breadth and width at most $2d+1$ after removing the apex set $B_c^i\cup X_c^i$, which concludes the proof.

In our case, we cannot add $X_c^i$ to the apex set.
That is, we need to argue that, even if we add edges between the vertices of $X_c^i$, \textsl{it does not destroy} the almost embedding too much, in such a way that, by creating new vortices, there is still an almost embedding of bounded breadth and bounded depth.
Before going further, let us define a stellation in a graph and the torso of a set in a graph (to distinguish with that torso of a tree decomposition at a node).

\paragraph{Stellation and torso.}
Let $G$ be a graph and $\Scal$ be a collection of subsets of $V(G)$.
We denote by $G_\Scal^\star$ the graph obtained from $G$ by adding, for each $S\in\Scal$, a vertex $v_S$ adjacent to the vertices in $S$.
The vertices $v_S$ are called \emph{stellation vertices}.
Let $X\subseteq V(G)$. 
The \emph{torso} of $X$ in $G$, denoted by $\torso_G(X)$ is the graph derived from the induced subgraph $G[X]$ by turning $N_G(V(C))$ into a clique for each connected component $C$ of $G-X$.
We denote by $G_\Scal^\circ$ the graph obtained from $G$ by turning each $S\in\Scal$ into a clique.
In other words, $G_\Scal^\circ$ is the torso of $V(G)$ in $G_\Scal^\star$.

Note that we now have two definitions of torso.
One is torso of a set $X\subseteq V(G)$ in  a graph $G$ defined just above.
The other one (see \autoref{subsec_frame}) is, given a tree decomposition $\Tcal=(T,\beta)$ of $G$, the torso of $\Tcal$ at node $t$.
This second notion is stronger, in the sense that the torso of $\Tcal$ at node $t$ is the graph obtained from $G[\beta(t)]$ by making a clique out of each set $\beta(t)\cap\beta(t')$ with $t'$ adjacent to $t$, while the torso of $\beta(t)$ in $G$ makes a clique out of subsets of the sets $\beta(t) \cap \beta(t')$.
However, we can prove the following.

\begin{lemma}\label{lem_2torso}
If a graph $G$ has treewidth $k$, then there is always a tree decomposition $\Tcal=(T,\beta)$ of $G$ of width $k$ such that the torso of $\Tcal$ at node $t$ is exactly the torso of $\beta(t)$ in $G$.

Moreover generally, let $\Hcal$ be a hereditary\footnote{A graph class $\Hcal$ is \emph{hereditary} if, for each $G\in\Hcal$ and each $v\in V(G)$, $G-v\in\Hcal$.}graph class and suppose that $G$ admits a tree decomposition $\Tcal'=(T',\beta')$ of width $k$ such that the torso $\Tcal'$ at each node is in $\Hcal$. 
Then, there is a tree decomposition $\Tcal=(T,\beta)$ of $G$ of width $k$ such that the torso $G_t$ of $\Tcal$ at node $t$ is exactly the torso of $\beta(t)$ in $G$ and that $G_t\in\Hcal$.
\end{lemma}

\begin{proof}
While there is $t\in V(T')$ such that the torso of $\Tcal'$ at node $t$ is different from the torso of $\beta'(t)$ in $G$, we modify the tree decomposition $\Tcal'$ as follows.
If the two notion of torso differ for $t$, then this means that there is a neighbor $t'$ of $t$ such that several connected components $C_1,\dots,C_\ell$ of $G-\beta'(t)$, with $\ell\ge2$, contain vertices of $\beta'(t')$, meaning that we add more edges in the torso of $\Tcal$ at $t$ (where we make a clique out of all of $\beta'(t)\cap\beta'(t')$) than in the torso of $\beta'(t)$ in $G$ (where we only make a clique out of each $N_G(C_i)\subseteq \beta'(t)\cap\beta'(t')$).
We may modify the tree decomposition by removing the subtree $T_{t'}'$ rooted at $t'$ (where we assume that the root is $t$), and adding instead $\ell$ copies $T_{t'}^1,\dots,T_{t'}^\ell$ with $\beta'(u_i)=\beta'(u)\cap V(C_i)$ for $u\in V(T_{t'})$ and $u_i$ its copy in $T_{t'}^i$, and joining $t$ to each copy of $t'$.
In this new tree decomposition, for both notions of the torso, we make a clique out of the entirety of $\beta'(t)\cap\beta'(t_i)$ for $i\in\ell$.
For each $u\in V(T'_t)$, if the torso $G_u$ at $u$ is in $\Hcal$, then the torso at $u_i$ is equal to $G_u\cap V(C_i)$, which is also in $\Hcal$ given that $\Hcal$ is hereditary.
Hence, by proceeding as such by induction, we obtain the desired tree decomposition.
\end{proof}

\begin{figure}[h]
  \begin{center}
  \scalebox{0.65}{\includegraphics{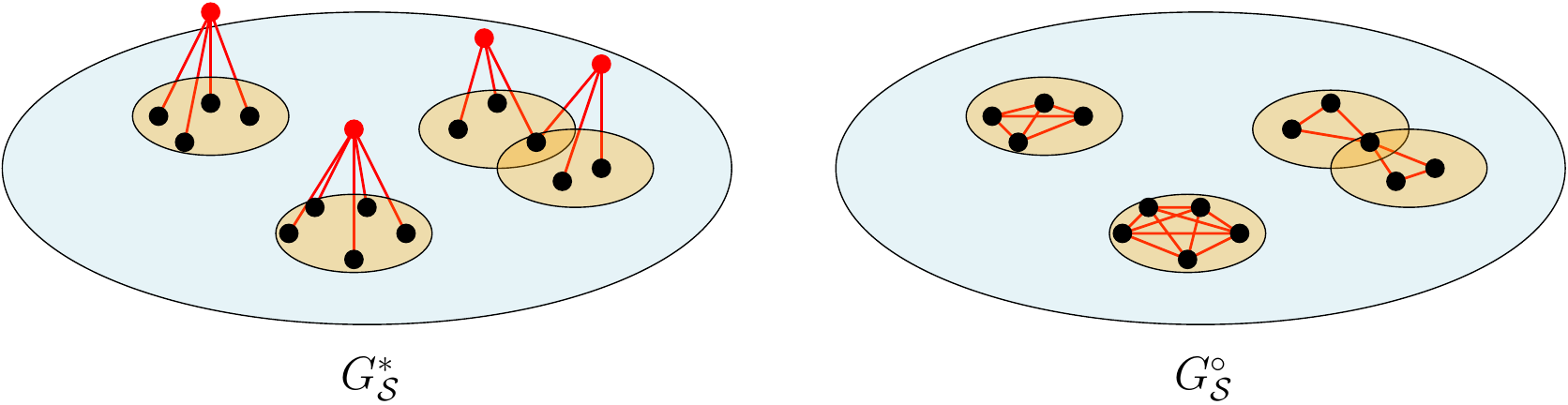}}
  \end{center}
    \caption{Illustration of $G_\Scal^\star$ and $G_\Scal^\circ$, where the sets in $\Scal$ are depicted in orange, and the new vertices and edges are depicted in red.}
  \label{fig:stelltorso}
\end{figure}

\smallskip
\autoref{thm_projective_global} is the rooted version of our global theorem and it essentially goes as follows:
\begin{quote}
Let $G,H$ be graphs ($H$ is a long-jump grid), $X\subseteq V(G)$ be a set of small size, and $\Scal$ be a collection of subsets of $X$. Then either $G_\Scal^\star$ contains $H$ as a minor, or there is a rooted tree decomposition $\Tcal=(T,\beta,r)$ of $G_\Scal^\circ$ such that, for each node $t$ of $T$, the torso of $\beta(t)$ in $G_\Scal^\circ$ has an almost embedding in $\Sigma$ and such that $X\subseteq \beta(r)$.
\end{quote}
If $X=\emptyset$, the above gives our global structure theorem by \autoref{lem_2torso}.
Let us give intuition on $X$ and $\Scal$.
As previously, $X$ corresponds to vertices that are present in the parent bag inherited from the induction.
Given that we now work with the torso of $\beta(t)$ and not the torso of $\Tcal$ at $t$, we will not make a clique out of all of $X$.
Instead, we make a clique out of a subset of $X$ if it neighbors some connected component with respect to the parent.
This is what $\Scal$ represents.

Let us sketch the proof of our result.
As previously, if $X$ has a balanced separator, we easily conclude, so we may assume that there is a wall in $G$, and thus in $G_\Scal^\star$.
We now apply the local structure theorem on $G_\Scal^\star$ and conclude that, if $G_\Scal^\star$ is $H$-minor-free, then there is a $\Sigma$-decomposition of $G_\Scal^\star$ of small breadth and depth at most $d$.
For each cell $c$, we define $G_c^i$, $A_c^i$, and $X_c^i$ as previously.
We additionally define $\Scal_c^i$ to be the collection of all sets $N_{G_\Scal^\star}(C)$ such that $C$ is a connected component of $G_\Scal^\star-G_c^i$.
Hence, we can recurse on $G_c^i$, $X_c^i$, and $S_c^i$.
$(G_c^i)_{\Scal_c^i}^\star$ is a minor of $G_\Scal^i$ (obtained by contracting each connected component $C$ of $G_\Scal^\star-G_c^i$ to a point, so if $H$ is a minor of $(G_c^i)_{\Scal_c^i}^\star$, it is also a minor of $G_\Scal^\star$.
So we can assume that we found a tree decomposition $\Tcal_c^i$ respecting the criteria.
Then we define the tree decomposition $\Tcal=(T,\beta,r)$ just as before.
It remains to prove that there is an almost embedding of the torso of $\beta(t)$ in $G_\Scal^\circ$ of small breadth and depth for each $t\in V(T)$.
This is immediate for $t\ne r$ by induction.
The difficult part is to prove so for $r$.
We currently have an almost embedding in $G_\Scal^\star$ of small breadth and width at most $2d+1$ (each cell containing vertices of $X$ is added to the vortex set).
To obtain the torso of $\beta(r)$ in $G_\Scal^\circ$, we need to make a clique out of each $X_c^i$, as well as each $S\in\Scal$.
Making a clique out of each $X_c^i$ does not destroy the embedding.
Moreover, for each $S$, if the corresponding stellation vertex $v_S$ is in the interior of a cell $c$, then its neighborhood $S$ is in $\sigma(c)$, so making a clique out of $S$ does not destroy the embedding.
The problem is when $v_S$ is on the boundary of at least two cells.
Then, it has neighbors in different cells between which we need to add an edge, hence destroying the almost embedding.
The idea is to create a new vortex containing all cells with $v_S$ on its boundary (whose number can be assumed to be bounded by $|X|$).
The problem is that another stellation vertex $v_{S'}$ could possibly be on the boundary of this new vertex.
So, we need to add all cells with $v_{S'}$ on its boundary to the new vortex, and so on.
Two sets $S,S'\in\Scal$ whose stellation vertices are on the boundary of the same cell are said to be adjacent, which allows us to talk about connected components of $\Scal$.
What we will show (using topological arguments) is that we can create a vortex $c_Y$ for each connected component $Y$ of $\Scal$ such that, if $S\in Y$, then $v_S$ is in the interior of $c_Y$, and thus $S$ is in $\sigma(c_Y)$, allowing us to make a clique out of $S$ safely.
These new vortices can be chosen to be distinct and to have width bounded by a function of $|X|$ and $d$.
Hence, we find the desired almost embedding: instead of increasing the size of the apex set, what we grow is the width and the breadth of the embedding.

\smallskip
Before proving the rooted version of the global structure theorem, let us prove the following results, that essentially says that the number of stellation vertices (here $|S|$) is bounded by $|X|$.

Given a graph $G$ embedded in a surface $\Sigma$ such that the faces of $G$ are disks, the \emph{degree} of a face of $G$ is the number of edges bounding the face (counted with multiplicity).

\begin{lemma}\label{lem_propproj}
Let $G$ be a graph embeddable in the projective plane, and $(X,S)$ be a partition of $V(G)$ such that $S$ is an independent set and $|X|\ge1$.
Then $|\{N_G(s)\cap X\mid s\in S\}|\le 6|X|-4$.
\end{lemma}

\begin{proof}
Without loss of generality, we can assume that the vertices in $S$ have pairwise distinct neighborhood.
Then it is enough to prove that $|S|\le 6|X|-5$.
Note that, if $|X|=1$, then $|S|\le 1=6|X|-5$, and, if $|X|=2$, then $|S|\le 3\le6|X|-5$.
Hence, we can assume that $|X|\ge 3$.
If $G$ is not planar (resp. planar), then there is an embedding of $G$ in the projective plane (resp. the sphere), such that each face of $G$ is a disk.

Let $d_i$ be the number of vertices in $S$ of degree $i$, for $i\in[0,2]$, and $d_{\ge 3}$ be the number of vertices of $S$ of degree at least three.
Given that $S$ is an independent set and that the vertices of $S$ have distinct neighborhoods, we have $d_0\le1$ and $d_1\le|X|$.
Let $G'$ be the simple graph obtained from $G$ by removing each vertex of $S$ and
\begin{itemize}
\item for each vertex of $S$ of degree two, adding an edge between its neighbors, and
\item for each vertex of $S$ of degree three or more with neighbors $x_1,\dots, x_\ell$ appearing in this order in the embedding, adding an edge between $x_1$ and $x_{i+1}$ for $i\in[\ell]$ (modulo $\ell$).
\end{itemize}
$G'$ is a graph with vertex set $X$ that is embeddable in the projective plane.
In particular, $d_2$ is bounded by the number of edges of $G'$, and $d_{\ge 3}$ is bounded by the number of faces of $G'$.

Let $v$ be the number of vertices, $e$ be the number of edges, and $f$ be the number of faces of $G'$.
Let us compute the sum $s$ of the degree of the faces of $G'$.
Given that each face has degree at most three when $|V(G')|\ge3$, we have $s\ge 3f$. 
Additionally, given that each edge bounds two faces (or maybe once, but then it counts twice), we have $s=2e$.
Therefore, $3f\le 2e$.
By Euler's formula for the projective plane (resp. the sphere) \cite{diestel2016graph}, we also have $v+f-e=1$ (resp. $v+f-e=2$).
Therefore, we deduce that $e\le 3v-3$ and $f\le 2v-2$.
Therefore, $|S|= d_0+d_1+d_2+d_{\ge 3}\le 1+|X|+(3|X|-3)+(2|X|-2)\le 6|X|-4.$
\end{proof}

We finally prove the rooted version of our global structure theorem.

\begin{theorem}\label{thm_projective_global}
There exists functions $f_{\ref{thm_projective_global}},b_{\ref{thm_projective_global}},w_{\ref{thm_projective_global}}\colon\mathbb{N}\to\mathbb{N}$ such that for every positive integer $k$, every graph $G$, every set $X\subseteq V(G)$ of size at most $3f_{\ref{thm_projective_global}}(k)+1$, and every collection $\Scal$ of subsets of $X$,  either
\begin{enumerate}
    \item $G_\Scal^\star$ contains the long jump grid of order $k$ as a minor, or
    \item there exists a rooted tree decomposition $(T,r,\beta)$ of $G_\Scal^\circ$ where
    \begin{enumerate}
      \item $X\subseteq\beta(r)$,
      \item $(T,\beta)$ has adhesion at most $3f_{\ref{thm_projective_global}}(k)+1$, and
      \item for every $t\in V(T)$, $\torso_{G_\Scal^\circ}(\beta(t))$ has an almost embedding in the projective plane of breadth at most $b_{\ref{thm_projective_global}}(k)$ and width at most $w_{\ref{thm_projective_global}}(k)$.
    \end{enumerate}
\end{enumerate}
Moreover, $f_{\ref{thm_projective_global}}(k),w_{\ref{thm_projective_global}}(k),b_{\ref{thm_projective_global}}(k)=2^{\Ocal(k\log k)}$.
\end{theorem}

\begin{proof}
Let $c_1$ be the constant from \autoref{thm:algogrid}.
We set
\begin{align*}
f_{\ref{thm_projective_global}}(k) &\coloneqq c_1\cdot\big(f_{\ref{mainresult}}(k,3)\big)^{20},\\
w_{\ref{thm_projective_global}}(k) &\coloneqq(24f_{\ref{thm_projective_global}}(k)+4)\cdot(2d_{\ref{mainresult}}(k)+1),\text{ and}\\
b_{\ref{thm_projective_global}}(k) &\coloneqq 3f_{\ref{thm_projective_global}}(k)+k.
\end{align*}

We prove the claim by induction on $|V(G)\setminus X|$.

\paragraph{If $G$ is small.}
In case $|V(G)|\leq 3f_{\ref{thm_projective_global}}(k)+1$, we may select $T$ to be the tree on one vertex, say $t$, and set $\beta(t)\coloneqq V(G)$.
The resulting tree decomposition $(T,\beta)$ of $G$, which is also a tree decomposition of $G_\Scal^\circ$, trivially meets the requirements of our assertion given that $3f_{\ref{thm_projective_global}}(k)+1\le w_{\ref{thm_projective_global}}(k)$.
Hence, we may assume that $|V(G)|>3f_{\ref{thm_projective_global}}(k)+1$.

\paragraph{If $X$ is small.}
Moreover, if $|X|\leq 3f_{\ref{thm_projective_global}}(k)$, we may now select an arbitrary vertex $v\in V(G)\setminus X$ and set $X'\coloneqq X\cup \{ v\}$.
It follows that $|X'|\leq 3f_{\ref{thm_projective_global}}(k)+1$ and $|V(G)\setminus X|>|V(G)\setminus X'|$.
Hence, by applying the induction hypothesis to $G$, $X'$, and $\Scal$, we obtain either the long jump grid of order $k$ as a minor of $G_\Scal^\star$, and are therefore done, or we obtain a rooted tree decomposition $(T,r,\beta)$ of $G_\Scal^\circ$ meeting requirements \textit{(a)}, \textit{(b)}, and \textit{(c)}.
In particular, we have $X\subsetneq X'\subseteq \beta(r)$ and are therefore done.
Thus, we may also assume that $|X|=3f_{\ref{thm_projective_global}}(k)+1$.

\paragraph{If there is $X'\subseteq X$ such that $G-X'$ is disconnected.}
Let $X' \subseteq X$ and let $H_1,\dots,H_{\ell}$ be the components of $G-X'$.
If $\ell\ge 2$, then we may define a rooted tree decomposition $(T,r,\beta)$ as follows.
For each $i\in[\ell]$, let $G_i$ be the graph induced by $X'\cup V(H_i)$.
Let $\Scal_i$ be the collection of all the sets $N_{G_\Scal^\star}(V(C))$ such that $C$ is a connected component of $G_\Scal^\star-V(G_i)$. 
Note that $(G_i)_{\Scal_i}^\star$ is a minor of $G_\Scal^\star$, given that it is obtained from $G_\Scal^\star$ by contracting each component of $G_\Scal^\star-V(G_i)$ to a single vertex.
Moreover, for each $S\in\Scal_i$, $S\subseteq X\cap V(G_i)$.
Given that $|V(G_i)\setminus X|<|V(G)\setminus X|$, by applying the induction hypothesis to $G_i$, $X\cap V(G_i)$, and $\Scal_i$, we obtain either that $\mathscr{J}_k$ is a minor of $(G_i)_{\Scal_i}^\star$, and thus of $G_\Scal^\star$, and therefore we are done, or we obtain a rooted tree decomposition $(T_i,r_i,\beta_i)$ of $(G_i)_{\Scal_i}^\circ$ meeting requirements \textit{(a)}, \textit{(b)}, and \textit{(c)}.
Let $T$ be the tree with root $r$ obtained from the disjoint union of the trees $T_i$ and the vertex $r$ by joining $r_i$ to $r$ for every $i\in[\ell]$.
We set $\beta(r)\coloneqq X$, and, for each $i\in[\ell]$ and $t\in V(T_i)$, we set $\beta(t)\coloneqq \beta_i(t)$.
Requirements \textit{(a)} and \textit{(b)} are trivially satisfied.
Moreover, given that $3f_{\ref{thm_projective_global}}(k)+1\le w_{\ref{thm_projective_global}}(k)$, $\torso_{G_\Scal^\circ}(\beta(r))$ has an almost embedding in the plane composed of a single vortex of width at most $w_{\ref{thm_projective_global}}(k)$, hence meeting requirement \textit{(c)}.
Finally, notice that, for $t\in V(T_i)$, the torso of $\beta(t)$ in $G_\Scal^\circ$ is equal to the torso of $\beta_i(t)$ in $(G_i)_{\Scal_i}^\circ,$ hence meeting requirement \textit{(c)}.
Therefore $(T,r,\beta)$ is indeed a rooted tree decomposition of $G_\Scal^\circ$ as required by the assertion.
So we may assume that $G-X'$ is connected for any $X'\subseteq X$. In particular, $G$ is connected.\smallskip


\paragraph{If there exists a balanced separator for $X$.}
Suppose there exists a $2/3$-balanced separator $S$ of size at most $f_{\ref{thm_projective_global}}(k)$ for $X$.

In this case let $H_1,\dots,H_{\ell}$ be the components of $G-S$ and, for each $i\in[\ell]$, let $X_i'\coloneqq (X\cap V(H_i))\cup S$.
It follows that 
\begin{align*}
|X_i'| &\leq \frac{2}{3}(3f_{\ref{thm_projective_global}}(k)+1)+|S|\\
&\leq 2f_{\ref{thm_projective_global}}(k)+f_{\ref{thm_projective_global}}(k)\\
&\leq 3f_{\ref{thm_projective_global}}(k).
\end{align*}
Now, for each $i\in[\ell]$, if $|X_i'\cup V(H_i)|\leq d_{\ref{thm_projective_global}}(k)$, we set $X_i\coloneqq X_i'\cup V(H_i)$ and we say that $H_i$ is a \emph{leaf}.

Otherwise, we select an arbitrary vertex $v_i\in V(H_i)\setminus X_i'$ and set $X_i\coloneqq X_i'\cup \{ v_i\}$.
Observe that $|X_i|\leq 3f_{\ref{thm_projective_global}}(k)+1$ in this case.
For every $i\in[\ell]$ for which $H_i$ is \textit{not} a leaf, 
let $G_i=G[X_i\cup V(H_i)]$.
Let also $\Scal_i$ be the collection of all the sets $N_{G_\Scal^\star}(V(C))$ such that $C$ is a connected component of $G_\Scal^\star-V(G_i)$.
This implies that $(G_i)_{\Scal_i}^\star$ can be obtained from $G_\Scal^\star$ by contracting each component of $G_\Scal^\star-V(G_i)$ to a single vertex, and thus, $(G_i)_{\Scal_i}^\star$ is a minor of $G_\Scal^\star$.
Notice that the elements in $\Scal_i$ are subsets of $X_i$.
We have that $|V(H_i)\setminus X_i|<|V(G)\setminus X|$ and thus, 
by applying the induction hypothesis to $G_i$, $X_i$, and $\Scal_i$, we obtain either that $\mathscr{J}_k$ is a minor of $(G_i)_{\Scal_i}^\star$, and thus of $G_\Scal^\star$, and therefore we are done, or we obtain a rooted tree decomposition $(T_i,r_i,\beta_i)$ of $(G_i)_{\Scal_i}^\circ$ meeting requirements \textit{(a)}, \textit{(b)}, and \textit{(c)}.
For every $i\in[\ell]$ where $H_i$ is a leaf we define such a rooted tree decomposition $(T_i,r_i,\beta_i)$ by setting $T_i$ to be the tree with a single vertex $r_i$ and $\beta_i(r_i)\coloneqq X_i$.

Now let us define a rooted tree decomposition $(T,r,\beta)$ for $G$ as follows.
Let $T$ be the tree with root $r$ obtained by taking the disjoint union of the vertex $r$ and the trees $T_i$, and joining $r_i$ to $r$ for all $i\in[\ell]$.
For every $i\in[\ell]$ and $t\in V(T_i)$, we set $\beta(t)\coloneqq \beta_i(t)$, and we set $\beta(r)\coloneqq X\cup S$.
Note that $|\beta(r)|\le 4f_{\ref{thm_projective_global}}(k)+1\le w_{\ref{thm_projective_global}}(k)$, and that, for each $i\in[\ell]$ such that $H_i$ is \emph{not} a leaf, for each $t\in V(T_i)$, $\torso_{G_\Scal^\circ}(\beta(t))=\torso_{(G_i)_{\Scal_i}^\circ}(\beta_i(t))$.
Then it is straightforward to check that $(T,r,\beta)$ is indeed a rooted tree decomposition of $G_\Scal^\circ$ as required by the assertion.
Hence, we may assume that there is no $2/3$-balanced separator of size at most $f_{\ref{thm_projective_global}}(k)$ for $X$.

\paragraph{Local structure theorem.} 
If there is no $2/3$-balanced separator of size at most $f_{\ref{thm_projective_global}}(k)$ for $X$, then $X$ is $(f_{\ref{thm_projective_global}}(k),2/3)$-well-linked in $G$, and thus also in $G_\Scal^\star$.
By \autoref{thm:algogrid}, given that $f_{\ref{thm_projective_global}}(k) \coloneqq c_1\cdot\big(f_{\ref{mainresult}}(k,3)\big)^{20}$, it implies that $G_\Scal^\star$ contains a $f_{\ref{mainresult}}(k,3)$-wall $W$ such that $\Tcal_W$ is a truncation of $\Tcal_X$.
Then, by \autoref{mainresult}, this implies that, either $\mathscr{J}_k$ is a minor of $G_\Scal^\star$, in which case we conclude, or 
$G_\Scal^\star$ has a $\Sigma$-decomposition $\delta=(\Gamma,\Dcal)$ of breadth at most $k-1$ and depth at most $d_{\ref{mainresult}}(k)$, where $\Sigma$ is the projective plane, and there exists a wall $W'$ of height at least three which is flat in $\delta$ and whose tangle is a truncation of the tangle of $W$.
Additionally, the closure of the vortex cells of $\delta$ are pairwise disjoint.
Let $\Ccal_v$ be the set of vortex cells of $\delta$.

Without loss of generality, we may assume that, for each cell $c\in C(\delta)\setminus \Ccal_v$, and each distinct $u,v\in\pi_{\delta}(\tilde{c})$, there is a path from $u$ to $v$ whose internal vertices are in $\sigma_{\delta}(c)-\pi_{\delta}(\tilde{c})$.
If not, then, if $|\tilde{c}|=2$, then $\sigma_{\delta}(c)$ is disconnected and $c$ can be divided into two cells $c_u$ and $c_v$ such that $\pi_{\delta}(\tilde{c}_u)=\{u\}$ and $\pi_{\delta}(\tilde{c}_v)=\{v\}$.
And if $\pi_{\delta}(\tilde{c})=\{u,v,w\}$, then $w$ is a cut vertex of $\sigma_{\delta}(c)$ and $c$ can be divided into two cells $c_u$ and $c_v$ such that $\pi_{\delta}(\tilde{c}_u)=\{u,w\}$ and $\pi_{\delta}(\tilde{c}_v)=\{v,w\}$.

Without loss of generality, we may also assume that, for each ground vertex $v\in\ground(\delta)$, there is $c\in C(\delta)\setminus \Ccal_v$ such that $v\in\pi_{\delta}(\tilde{c})$ and $N_{\sigma_\delta(c)}(v) \neq\emptyset$.
Indeed, suppose that is not the case. Then, given that $G$ is connected, for any $c\in C(\delta)$ such that $v\in\pi_{\delta}(\tilde{c})$, $c$ is a vortex cell.
Given that the closure of the vortex cells are pairwise disjoint, $v$ is thus drawn on the boundary of a unique cell $c\in\Ccal_v$. Then, we can draw $v$ in the interior of $c$ instead of its boundary: It does not increase the depth of the vortex.

For each $\Rcal\subseteq\Scal$, let $V_\Rcal$ be the set of all stellation vertices $v_S$ such that $S\in\Rcal$.
Let $\Scal^{\sf g}$ be the set of $S\in\Scal$ such that $v_S\in\ground(\delta)$.
Let $X_1\coloneqq X\cap\ground(\delta)$, $X_2$ be the set of vertices in $X$ drawn in the interior of non-vortex cells, and $X_3$ be the set of vertices in $X$ drawn in the interior of vortex cells.
Notice that $(X_1,X_2,X_3)$ is a partition of $X$.
We have the following bound of the number of stellation vertices on the ground.

\begin{claim}\label{cl_Sbound}
If $\Scal^{\sf g}\ne\emptyset$, then $|\Scal^{\sf g}|\le 6|X_1\cup X_2|-5$.
\end{claim}
\begin{cproof}
Let $\Ccal_X$ be the set of cells $c\in C(\delta)\setminus \Ccal_v$ such that $X_2\cap \sigma_{\delta}(c)\ne\emptyset$.
Let $F$ be the graph with vertex set the union of $X_1$, $V_{\Scal^{\sf g}}$ and a vertex $v_c$ for each $c\in C_X$, and edge set the edges of $G_\Scal^\star$ with both endpoints in $X_1 \cup V_{\Scal^{\sf g}}$ and, for each $c\in C_{X}$ and each $v_S\in\pi_{\delta}(\tilde{c})\cap V_{\Scal^{\sf g}}$, an edge between $v_c$ and $v_S$.
By construction, $F$ is embeddable in the projective plane.
Let $X'\coloneqq X_1\cup \{v_c\mid c\in C_X\}$. 
$V_{\Scal^{\sf g}}$ is an independent set of $G_\Scal^\star$ and thus of $F$.
Additionally, by the assumptions on $\delta$, for each $v_S\in V_{\Scal^{\sf g}}$, there is $c\in C(\delta)\setminus \Ccal_v$ such that $v_S\in\pi_{\delta}(\tilde{c})$ with $N_{\sigma_\delta(c)}(v_S)\ne\emptyset$.
Therefore, there is $x\in (X_1\cup X_2)\cap S$ such that $x\in V(\sigma_\delta(c))$.
Thus, $N_F(v_S)\cap X'\ne\emptyset$.
We thus conclude, by \autoref{lem_propproj} applied to $V_{\Scal^{\sf g}}$ and $X'$, that $|\Scal^{\sf g}|=|V_{\Scal^{\sf g}}|\le 6|X'|-5\le 6|X_1\cup X_2|-5$.
\end{cproof}

\paragraph{Construction of the tree decomposition.}

For every non-vortex cell $c\in C(\delta)\setminus \Ccal_v$, we set $G_c^1=\sigma_\delta(c)$, $A_c^1=\pi_{\delta}(\tilde{c})$, and $X_c^1=(X\cap V(\sigma_\delta(c)))\cup A_c^1$. 
$(A,B)=(V(\sigma_\delta(c)),V(G_\Scal^\star)\setminus(V(\sigma_\delta(c))\setminus A_c^1))$ is a separation of $G_\Scal^\star$ of order $|A_c^1|\le3\le f_{\ref{thm_projective_global}}(k)$.
Given that $W'$ is flat in $\delta$, and in particular that at most one branch vertex of $W'$ is in $\sigma_\delta(c)-\pi_{\delta}(\tilde{c})$, it follows that $A\setminus B$ does not contain a row and a column of $W'$, and thus $(A,B)\in\Tcal_{W'}\subseteq\Tcal_X$.
Therefore, $|X_c^1|\le|X\cap V(\sigma_\delta(c))|+3\le\frac{1}{3}|X|+3\le 3f_{\ref{thm_projective_global}}(k)$.

For every vortex cell $c\in\Ccal_v$, let $(Y_1,\dots, Y_\ell)$ be a linear decomposition of the vortex society $(\sigma_\delta(c),\Omega_c)$ of $c$ of adhesion at most $d_{\ref{mainresult}}(k)$ (it exists by \autoref{prop_width}), with the vertices of $V(\Omega_c)$ labeled $v_1,\dots,v_\ell$.
For $i\in[\ell]$, let $A_c^i=(Y_i\cap Y_{i-1})\cup(Y_i\cap Y_{i+1})\cup\{v_i\}$ where $Y_0=Y_{\ell+1}=\emptyset$.
For each $i\in[\ell]$, we set $G_c^i$ to be the graph induced by $Y_i$ and $X_c^i=(X\cap Y_i)\cup A_c^i$.
$(A,B)=(Y_i,V(G_\Scal^\star)\setminus(Y_i\setminus A_c^i))$ is a separation of $G_\Scal^\star$ of order $|A_c^i|\le2d_{\ref{mainresult}}(k)+1\le f_{\ref{thm_projective_global}}(k)$.
Given that $W'$ is flat in $\delta$, and in particular that there is no vortex in the disk where the interior of $W'$ is drawn, it follows that $A \setminus B$ does not contain a row and a column of $W'$, and thus $(A,B) \in \Tcal_{W'} \subseteq \Tcal_X$.
Therefore, $|X\cap V(\sigma_\delta(c))|\le\frac{1}{3}|X|$, and thus that $|X_c^i|\le\frac{1}{3}|X|+2d_{\ref{mainresult}}(k)+1\le3f_{\ref{thm_projective_global}}(k)$.

For every cell $c\in C(\delta)$ and each $i$, we set $H_c^i\coloneqq G_c^i-V_\Scal$ and $Z_c^i\coloneqq X_c^i-V_\Scal$.
We also set $S_c^i$ to be the collection of all the sets $N_{G_\Scal^\star}(V(C))$ such that $C$ is a connected component of $G_\Scal^\star-V(H_c^i)$.
This implies that $(H_c^i)_{\Scal_c^i}^\star$ can be obtained from $G_\Scal^\star$ by contracting each component of $G_\Scal^\star-V(H_c^i)$ to a single point, and thus, $(H_c^i)_{\Scal_c^i}^\star$ is a minor of $G_\Scal^\star$.
Notice that, for each $R\in \Scal_c^i$, $R\subseteq Z_c^i$.

We define a rooted tree decomposition $(T,r,\beta)$ of $G$ as follows.
We define $\beta(r)$ to be the union of the sets $Z_c^i$, for all cells $c\in C(\delta)$ and all $i$.
For every cell $c\in C(\delta)$ and each $i$ such that $V(H_c^i)\setminus Z_c^i\ne\emptyset$, there exits $v_c^i\in V(H_c^i)\setminus Z_c^i$.
Let $Z_c'^i=Z_c^i\cup\{v_c^i\}$.
$|Z_c'^i|\le3f_{\ref{thm_projective_global}}(k)+1$ and $|V(H_c^i)\setminus Z_c'^i|<|V(G)\setminus X|$, so we can apply the induction hypothesis to $H_c^i$, $Z_c'^i$, and $\Scal_c^i$.
We obtain either that $\mathscr{J}_k$ is a minor of $(H_c^i)_{\Scal_c^i}^\star$, and thus of $G_\Scal^\star$, and therefore we are done, or we obtain a rooted tree decomposition $(T_c^i,r_c^i,\beta_c^i)$ of $(H_c^i)_{\Scal_c^i}^\star$ meeting requirements \textit{(a)}, \textit{(b)}, and \textit{(c)}.
$T$ is obtain from the union of $r$ and the trees $T_c^i$ by joining $r$ to each $r_c^i$.
For each $t\in V(T_c^i)$, we set $\beta(t)=\beta_c^i(t)$.

It is straightforward to check that $(T,r,\beta)$ meet requirements \textit{(a)} and \textit{(b)}, and that each $t\in V(T)\setminus \{r\}$ meet requirement \textit{(c)} (given that, for $t\in T_c^i$, $\torso_{G_\Scal^\circ}(\beta(t))=\torso_{(G_c^i)_{\Scal_c^i}^\circ}(\beta_c^i(t))$).
It remains to prove that $\torso_{G_\Scal^\circ}(\beta(r))$ has an almost embedding in the projective plane of breadth at most $b_{\ref{thm_projective_global}}(k)$ and width at most $w_{\ref{thm_projective_global}}(k)$.
The difficulty is that we need to make a clique out of $S$ for each $S\in\Scal$.
If all vertices of $S$ are drawn in the same cell, that is, if $v_S$ is drawn in the interior of a cell, then it does not change much.
However, if $v_S$ is a ground vertex, then making a clique out of $S$ destroys the $\Sigma$-decomposition.
Therefore, we need to find a $\Sigma$-decomposition $\delta^\star$ such that each stellation vertex is drawn in the interior of a cell.

\paragraph{If there is no stellation vertex on the ground.}
If $\Scal^{\sf g}=\emptyset$, then $\ground(\delta)\subseteq\beta(r)$ and, for each $S\in\Scal$, there is a cell $c\in C(\delta)$ such that $S\subseteq V(\sigma_\delta(c))$.
Then, we can construct a $\Sigma$-decomposition $\delta^\star$ of $\torso_{G_\Scal^\circ}(\beta(r))$ from the $\Sigma$-decomposition $\delta$ of $G_\Scal^\star$ by keeping only the vertices of $\beta(r)$ and making a clique out of each $Z_c^i$ for $c\in C(\delta)=C(\delta^\star)$ and out of each $S$ for each $S\in\Scal$.
Let $C_0$ be the union of $\Ccal_v$ and the cells of $C(\delta^\star)\setminus \Ccal_v$ that have a vertex drawn on their interior.
Notice that, by the definition of the sets $X_c^1$, the only vertices that can be drawn in the interior of non-vortex cells are vertices of $X_2$.
Therefore, $|C_0|\le k-1+|X_2|\le b_{\ref{thm_projective_global}}(k)$.
For each $c\in C_0\setminus \Ccal_v$, $|V(\sigma_{\delta^\star}(c))|\le 3+|X_2|$.
Let $c\in \Ccal_v$.
Remember that the vortex society $(\sigma_\delta(c),\Omega_c)$ of $c$ in $\delta$ has a linear decomposition $(Y_1,\dots, Y_\ell)$ of adhesion at most $d_{\ref{mainresult}}(k)$.
For $i\in[\ell]$, let $Y_i'\coloneqq (Y_i\cup X)\cap V(\sigma_{\delta^\star}(c))=A_c^i\cup(X_3\cap V(\sigma_{\delta^\star}(c)))$.
Then $(Y_1',\dots, Y_\ell')$ is a linear decomposition of the vortex society $(\sigma_{\delta^\star}(c),\Omega_c)$ in $\delta^\star$ of width at most $2d_{\ref{mainresult}}(k)+1+|X_3|\le w_{\ref{thm_projective_global}}(k)$.
Therefore, $\delta^\star$ is a $\Sigma$-embedding with vortex set $C_0$ of breadth at most $b_{\ref{thm_projective_global}}(k)$ and width at most $w_{\ref{thm_projective_global}}(k)$.
Hence, we now assume that $\Scal^{\sf g}\ne\emptyset$.

\paragraph{If all ground vertices are stellation vertices.}
If $V_{\Scal^{\sf g}}=\ground(\delta)$, then we can define $\delta^\star$ to be the $\Sigma$-decomposition of $\torso_{G_\Scal^\circ}(\beta(r))$ composed of a unique vortex cell $c$ 
with $V(\sigma_{\delta^\star}(c))=\beta(r)$ and one arbitrary vertex on the boundary.
Let us compute the width of the vortex society of $c$.
Given that all vertices in $\ground(\delta)$ are stellation vertices, $(X_2,X_3)$ is a partition of $X$.
Note that $\sum_{c\in \Ccal_v}|V(\Omega_c)|\le |\ground(\delta)|=|\Scal^{\sf g}|$. 
Remember that the boundary of vortex cells of $\delta$ are pairwise disjoint.
This implies that there is an injection between the bags in the linear decomposition $(Y_1,\dots,Y_\ell)$ of vortex cells in $\delta$ and the vertices in $\Scal^{\sf g}$.
Also, for each bag $Y_i$, $|Y_i\cap \beta(r)|=|(Y_i\cap Y_{i-1})\cup(Y_i\cap Y_{i+1})|\le2d_{\ref{mainresult}}(k)$.
Therefore,
\begin{align*}
|\beta(r)| &\leq |X_2|+|X_3|+\sum_{c\in\Ccal_v}|V(\Omega_c)|\cdot 2d_{\ref{mainresult}}(k)\\
&\leq |X|+|\Scal^{\sf g}|\cdot 2d_{\ref{mainresult}}(k)\\
&\le 3f_{\ref{thm_projective_global}}(k)+1+(6(3f_{\ref{thm_projective_global}}(k)+1)-5)\cdot 2d_{\ref{mainresult}}(k)\\
&\le (36f_{\ref{thm_projective_global}}(k)+2)\cdot d_{\ref{mainresult}}(k)+3f_{\ref{thm_projective_global}}(k)+1\\
&\le w_{\ref{thm_projective_global}}(k).
\end{align*}
Therefore, $\delta^\star$ is a $\Sigma$-embedding with vortex set $C_0$ of breadth at most $1\le b_{\ref{thm_projective_global}}(k)$ and width at most $w_{\ref{thm_projective_global}}(k)$.
We now assume that $\ground(\delta)\setminus V_{\Scal^{\sf g}}\ne\emptyset$.

\paragraph{Connected component of stellation vertices and its boundary.}

For $S,S'\in\Scal^{\sf g}$, we say that $S$ and $S'$ are \emph{adjacent} if $v_{S}$ and $v_{S'}$ are drawn on the boundary of the same cell of $\delta$.
We say that $S$ and $S'$ are \emph{connected} if there is a sequence $S_0=S,S_1,\dots,S_{\ell-1},S_{\ell}=S'$ such that, for each $i\in[\ell]$, $S_{i-1}$ and $S_i$ are adjacent.
Hence, a \emph{connected component} $Y$ of $\Scal^{\sf g}$ is a maximal size set of elements of $\Scal^{\sf g}$ that are pairwise connected.
Let us show that, for each connected component $Y$ of $\Scal^{\sf g}$, we can replace the cells containing a stellation vertex $v_S$ for $S\in Y$ in their boundary by a vortex, such that each stellation vertex $v_S$, for $S\in Y$, is drawn in the interior of the vortex.

Let $Y$ be a connected component of $\Scal^{\sf g}$.
It exists given that $\Scal^{\sf g}\ne\emptyset$.
We call \emph{interior} of $Y$ the set of points $x\in N(\delta)$ such that $\pi_{\delta}(x)\in V_{\Scal^{\sf g}}$. Points in the interior of $Y$ are said to be \emph{pink}.
We call \emph{boundary} of $Y$ the set of points $x\in N(\delta)$ that are not pink but are on the boundary of a cell $c\in C(\delta)$ containing a pink point.
We call \emph{exterior} of $Y$ the set of points $x\in N(\delta)$ that are neither in the boundary nor in the interior of $Y$.
Points in the boundary (resp. exterior) of $Y$ are said to be \emph{blue} (resp. \emph{green}).
Notice that, for each cell $c\in C(\delta)$, $\tilde{c}$ cannot contain both pink and green points.
Also, given that $Y\ne\emptyset$, there is at least one pink point, and given that $G$ is connected and that $\ground(\delta)\setminus V_{\Scal^{\sf g}}\ne\emptyset$, there is at least one blue point.

For each face $F$ of $\Sigma-\bigcup_{D\in\Dcal}D$, let $V_F$ be the set of blue points in the closure of $F$.
Let $\Gamma_H$ be a drawing without crossings in $\Sigma$ whose vertices are the blue points and whose edges are, for each face $F$ of $\Sigma-\bigcup_{D\in\Dcal}D$ such that $|V_F|\ge2$, edges between the blue points in $V_F$ inducing a spanning tree.
Notice that, given that $Y$ is connected, $\Gamma_H$ has at most one face containing pink points, that we call the \emph{pink face} of $\Gamma_H$.
The faces of $\Gamma_H$ containing green points are called \emph{green faces}.

\begin{claim}
There is at most one connected component of $\Gamma_H$ 
that bounds green faces of $\Gamma_H$.
\end{claim}

\begin{cproof}
Suppose not and let $x$ and $y$ be two blue points,
$g_x,g_y$ be two green points, and $\Delta_x,\Delta_y\in\Dcal$ be disks such that $x$ and $y$ are not in the same connected component of $\Gamma_H$ and that $x,g_x\in\bd(\Delta_x)$ and $y,g_y\in\bd(\Delta_y)$.
Given that $\Gamma_H$ is drawn in the projective plane $\Sigma$, at most one connected component of $\Gamma_H$ may contain non-contractible cycles, and the other contains only contractible cycles.
Hence, without loss of generality, we may assume that the connected component of $x$ contains only (if any) contractible cycles.
Then it is implied that there is a cycle $T$ in $\bigcup_{\Delta\in\Dcal}\bd(\Delta)$ avoiding blue points such that $x$ and $y$ are contained in different connected components of $\Sigma-T$ and such that the connected component of $\Sigma-T$ containing $x$ is a disk.
Given that the closure of a cell cannot contain both pink and green points, it is implied that $T\cap N(\delta)$ contains either only pink points, or only green points.

If $T\cap N(\delta)$ contains only pink points, then $g_x$ and $g_y$ are in different components of $\Sigma-T$.
This implies that $\bigcup_{S \in Y} \{ v_{S} \}$ separates $\pi_{\delta}(g_x)$ from $\pi_{\delta}(g_y)$ in $\torso_{G_\Scal^\star}(\beta(r))$, and thus that $N_{G_\Scal^\star}(\bigcup_{S \in Y} \{ v_{S} \})\subseteq X$ separates $\pi_{\delta}(g_x)$ from $\pi_{\delta}(g_y)$ in $G$.
This contradicts the fact that, for all $X'\subseteq X$, $G-X'$ is connected.

If $T\cap N(\delta)$ contains only green points, then by the definition of blue points, there is $\Delta_x',\Delta_y'\in\Dcal$ containing pink points $r_x$ and $r_y$ respectively, such that $x\in\bd(\Delta_x')$ and $y\in\bd(\Delta_y')$, and $\Delta_x'$ and $\Delta_y'$ belong to different components of $\Sigma-T$.
However, given that $\pi_{\delta}(r_x),\pi_{\delta}(r_y)\in Y$, this contradicts that fact that $Y$ is connected.
Hence the result.
\end{cproof}

\paragraph{Vortex containing $V_Y$.}
Let $\Gamma_R$ be the connected component of $\Gamma_H$ that bounds the green faces of $\Gamma_H$.
If it does not exists, let $\Gamma_R$ be any connected component of $\Gamma_H$.
By definition of blue points, any edge of $\Gamma_R$ must bound the pink face.
We now redefine the pink face to be the face of $\Gamma_R$ containing pink points, that is, it contains the other components of $\Gamma_H$, if any.
\begin{figure}[h]
\center
\includegraphics[scale=0.75]{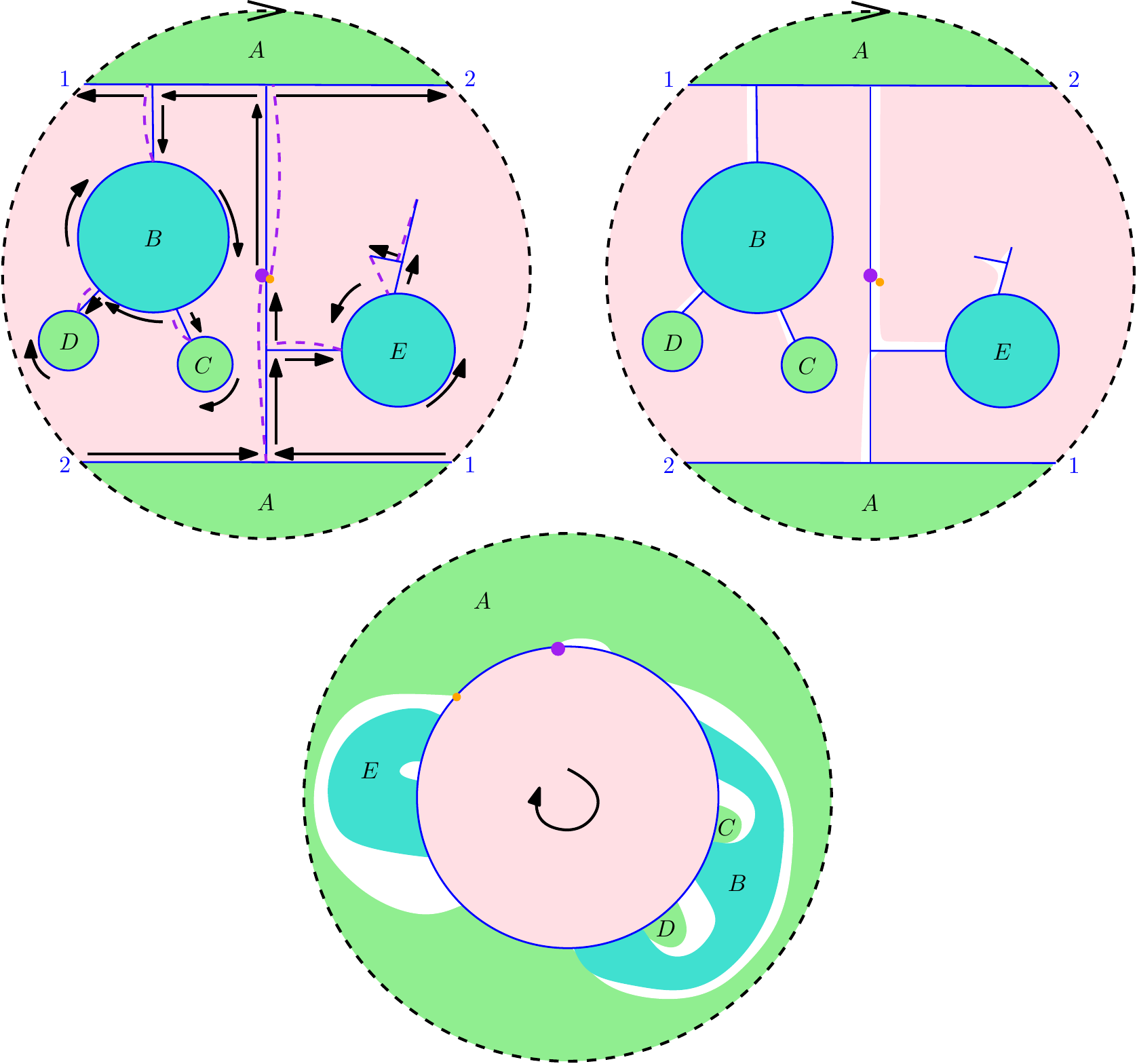}
\caption{An illustration of $\Gamma_R$ embedded on the projective plane: on the left, any point on the dashed cycle is identified to the point opposite to it with respect to the center of the cycle.
From the purple point, we add to the cyclic ordering the blue points that were not yet added following the order induced by the pink face.
The figure on the top right and bottom are the same up to some homeomorphism.
The pink face (figure on the left) becomes a vortex with boundary the blue points (i.e., vertices on the boundary of the pink face) of $\Gamma_R$ (figures on the right and bottom).}
\label{fig_redbluegreen}
\end{figure}
Let $x_1,\dots,x_\ell$ be the blue points of $\Gamma_R$, in the order induced by their appearance on the boundary of the pink face.
Notice that, given that an edge may bound this face twice, some points may appear twice in the ordering. Hence, we only keep the first appearance of each point in the ordering, and this gives us a cyclic ordering $\Omega_Y$ of the blue points of $\Gamma_R$.
See \autoref{fig_redbluegreen} for an illustration.

Let $C_Y$ be the set of all cells in the closure of the pink face, and $\Dcal_Y$ be the closure of the cells in $C_Y$.
Let $G_Y$ be the graph obtained from $\bigcup_{c\in C_Y}\sigma_{\delta}(c)$ by removing the vertices that do not belong to $\beta(r)$.
Let $T$ be any simple closed curve drawn within the pink face which contains all blue points of $\Gamma_{R}$, in the order prescribed by $\Omega_{Y}$.
If $\Sigma$ is the sphere, then both components of $\Sigma-T$ are disks and if $\Sigma$ is the projective plane, then exactly one component of $\Sigma-T$ is a disk (this would not be true if $\Sigma$ was the torus or any other surface of higher Euler genus).
Let $c_Y$ be the (possibly unique) component of $\Sigma-T$ that is contained in the pink face and $D_Y$ be its closure.
If $c_Y$ is a disk, then let $\delta_Y=(\Gamma,\Dcal_Y)$ be the $\Sigma$-decomposition of $G_\Scal^\star$ obtained from $\delta=(\Gamma,\Dcal)$ by setting $\Dcal_Y = \Dcal \setminus \Dcal_Y \cup D_Y$.
Hence, $C(\delta_Y)=C(\delta)\setminus C_Y\cup c_Y$, with $\sigma_{\delta_Y}(c_Y)=G_Y$, and $\pi_{\delta_Y}(\tilde{c}_Y)=\pi_{\delta}(V(\Gamma_R))$.
If $c_Y$ is not a disk, then we remove $c_Y$ from $\Sigma$ and replace it by a disk $c_Y'$, hence obtaining a sphere $\mathbb{S}_0$.
Then, we let $\delta_Y=(\Gamma,\Dcal_Y)$ be the $\mathbb{S}_0$-decomposition of $G_\Scal^\star$ obtained from $\delta$ by setting $\Dcal_Y=\Dcal\setminus \Dcal_Y\cup D_Y'$, where $D_Y'$ is the closure of $c_Y'$.
Since any planar graph can be embedded in the projective plane, let us assume without loss of generality that $\delta_Y$ be the $\Sigma$-decomposition of $G_\Scal^\star$ with new vortex $c_Y$.
Observe that every pink point is now drawn in the interior of $c_Y$, and that no green point is drawn in $c_Y$.
This implies that $V_Y$ is drawn in the interior $c_Y$ and that $V_{\Scal^{\sf g}\setminus Y}\subseteq \ground(\delta_Y)\setminus\pi_{\delta_Y}(\tilde{c}_Y)$.

We can apply this procedure iteratively to each connected component of $\Scal^{\sf g}$ (we do this procedure at most $18f_{\ref{thm_projective_global}}(k)+1$ times by \autoref{cl_Sbound}).
We thus obtain a $\Sigma$-decomposition $\delta'$ of $G_\Scal^\star$ such that, for each component $Y$ of $\Scal^{\sf g}$, there is a cell $c_Y\in C(\delta')$ such that $V_Y$ is drawn in the interior of $c_Y$.
Hence, every stellation vertex is drawn in the interior of a cell of $C(\delta')$.
It allows us to define a $\Sigma$-decomposition $\delta^\star$ of $\torso_{G_\Scal^\circ}(\beta(r))$ from the $\Sigma$-decomposition $\delta'$ of $G_\Scal^\star$ by keeping only the vertices of $\beta(r)$ and making a clique out of each $Z_c^i$ for $c\in C(\delta)$ and out of each set $S\in\Scal$.
Let $C_0$ be the union of the vortex cells of $\delta^\star$ and the cells of $\delta^\star$ whose interior is non empty.
A vortex cell of $\delta^\star$ is either a vortex cell of $\delta$ or a cell $c_Y$ for some component $Y$ of $\Scal^{\sf g}$.
Additionally, any cell of $\delta^\star$ whose interior is non empty is necessarily either a cell $c_Y$ for some component $Y$ of $\Scal^{\sf g}$ with $|\tilde{c}_Y|\le3$, or a cell containing a vertex of $X_2$.
Therefore, $|C_0|\le k-1+|\Scal^{\sf g}|+|X_2|\le b_{\ref{thm_projective_global}}(k)$.

For each component $Y$ of $\Scal^{\sf g}$, let us construct a linear decomposition of width at most $w_{\ref{thm_projective_global}}(k)$ of the vortex society $(G_Y,\Omega_Y)$ of $c_Y$.
Remember that $C_Y$ is the set of cells of $\delta$ that were replaced by $c_Y$.
Let $A\subseteq (V(G_Y) \cap \ground(\delta)) \setminus V_{Y}$ be the set of non-stellation vertices of $G_Y$ that are on the boundary of some non-vortex cell of $C_Y$ in $\delta$.

\begin{claim}
$|A|\le|X_1|+2|X_2|$.
\end{claim}
\begin{cproof}
By definition, $A \cap X \subseteq X_1$.
Let $v\in A\setminus X_1$.
Then there is a non-vortex cell $c\in C_{Y}$ and $v_S\in V_Y$ such that $v,v_S\in\pi_\delta(\tilde{c})$.
By the connectivity assumptions on $\delta$, there is a path $P$ from $v_S$ to $v$ whose internal vertices are in $\sigma_\delta(c)-\pi_\delta(\tilde{c})$.
In particular, the neighbor of $v_S$ in $P$ is in $X_2$.
The cell $c$ contributes for at most two vertices of $A\setminus X_1$ ($v$ and the third vertex of $\pi_\delta(\tilde{c})$, if it exists).
Therefore, $|A\setminus X_1|\le2|X_2|$.
\end{cproof}

Let $B\coloneqq (V(G_Y) \cap \ground(\delta)) \setminus A$.
By definition of $A$ and given that we assume any ground vertex of $\delta$ to be on the boundary of a non-vortex cell, it is implied that $B \subseteq V(\Omega_Y)$.
Moreover, by our connectivity assumptions, each vertex of $B$ is on the boundary of a (unique) vortex cell of $C_Y$.
Let us fix a linear decomposition of each vortex cell of $C_Y$ of adhesion at most $d_{\ref{mainresult}}(k)$.
For each vertex $v$ on the boundary of a (unique) vortex cell of $C_Y$, let $Y_v$ be the bag corresponding to $v$, and let $Y_v'\coloneqq Y_v\cap\beta(r)$.
We have $|Y_v'|\le2d_{\ref{mainresult}}(k)+1$.
Let $A'=V(G_Y)\setminus (\bigcup_{b\in B}Y_b'\setminus X)$.
Let us bound the size of $A'$.
Each vertex in $A'$ is either a vertex of $A$, or a vertex of $X_2\cup X_3$, or a vertex in $Y_a\setminus \{a\}$ for a vertex $a\in\ground(\delta)$ on the boundary of some non-vortex cell of $C_Y$, that is, for $a \in A \cup V_{Y}$.
Therefore, 
\begin{align*}
|A'| & \le|A|+|X_2|+|X_3|+(|A|+|\Scal^{\sf g}|)\cdot2d_{\ref{mainresult}}(k)\\
& \le |X_1|+3|X_2|+|X_3|+(6|X_1|+8|X_2|-5)\cdot2d_{\ref{mainresult}}(k)\\
&\le (8|X|-5)\cdot(2d_{\ref{mainresult}}(k)+1).
\end{align*}
Then, we define a linear decomposition $(Z_v)_{v\in V(\Omega_Y)}$ of $(G_Y,\Omega_Y)$ as follows.
For each $b\in B$, we set $Z_b\coloneqq Y_b'\cup A'$, and for each $a\in V(\Omega_Y)\setminus B$, we set $Z_a=A'$.
$(Z_v)_{v\in V(\Omega_Y)}$ has width at most 
\begin{align*}
|A'| +2d_{\ref{mainresult}}(k)+1& \le(8|X|-4)\cdot(2d_{\ref{mainresult}}(k)+1)\\
& \le(24f_{\ref{thm_projective_global}}(k)+4)(2d_{\ref{mainresult}}(k)+1)\\
&\le w_{\ref{thm_projective_global}}(k).
\end{align*}
Therefore, $\delta^\star$ is an almost embedding of $\torso_{G_\Scal^\circ}(\beta(r))$ in $\Sigma$ with vortex set $C_0$ of breadth at most $b_{\ref{thm_projective_global}}(k)$ and width at most $w_{\ref{thm_projective_global}}(k)$.
Hence, the result.
\end{proof}

As corollary of \autoref{thm_projective_global}, we immediately get the following.

\begin{theorem}\label{corr_proj_global}
Let $k\in\Nbbb$. Let $G$ be a graph that excludes a long jump grid of order $k$ as a minor.
Then there exists a tree decomposition $\Tcal$ of $G$ of adhesion at most $3f_{\ref{thm_projective_global}}(k)+1$ such that the torso of $\Tcal$ at each node has an almost embedding in the projective plane of breadth at most $b_{\ref{thm_projective_global}}(k)$ and width at most $w_{\ref{thm_projective_global}}(k)$.
\end{theorem}

Also, with the same proof as \autoref{thm_projective_global}, but replacing \autoref{mainresult} with \autoref{mainresult_pl}, we conclude the following.

\begin{theorem}\label{corr_plan_global}
There exist functions $f_{\ref{corr_plan_global}}, b_{\ref{corr_plan_global}}, w_{\ref{corr_plan_global}}:\Nbbb^2\to\Nbbb$ such that the following holds.
Let $k,c\in\Nbbb$. Let $G$ be a graph that excludes a long jump grid of order $k$ and a crosscap grid of order $c$ as a minor.
Then there exists a tree decomposition $\Tcal$ of $G$ of adhesion at most $3f_{\ref{corr_plan_global}}(k,c)+1$ such that the torso of $\Tcal$ at each node has an almost embedding in the plane of breadth at most $b_{\ref{corr_plan_global}}(k,c)$ and width at most $w_{\ref{corr_plan_global}}(k,c)$.

Moreover $f_{\ref{corr_plan_global}}(k,c), b_{\ref{corr_plan_global}}(k,c), w_{\ref{corr_plan_global}}(k,c)=2^{\Ocal(k\log(k\cdot c))}$.
\end{theorem}

\subsection{Identifying vortices}\label{sec_id_global}

We now deduce our global structure theorem as it was stated in the introduction.\medskip

In \cite{ThilikosW2025excluding}, it was proved that vortices of bounded width have bounded bidimensionality.

\begin{proposition}[Lemma 3.9, \cite{ThilikosW2025excluding}]\label{prop_id}
For every graph $G$, every $k\in\Nbbb$, and every surface $\Sigma$ with Euler genus at most $g$, if $\delta$ is a $\Sigma$-decomposition of $G$ with width at most $w$ and breadth at most $b$ and $X=\bigcup\{\sigma(c)\mid c \text{ is a vortex of } \delta\}$,
then $\bidim(G,X)=\Ocal((b^4\cdot(b\cdot g\cdot w)^4)$.
\end{proposition}

Therefore, we immediately get our upper bound.

\begin{theorem}\label{th_proj_global}
Let $k\in\Nbbb$. Let $G$ be a graph that excludes a long jump grid of order $k$ as a minor.
Then $\mathsf{idpr}^\star(G)=2^{\Ocal(k\log k)}$.
\end{theorem}

\begin{proof}
By \autoref{corr_proj_global}, there exists a tree decomposition $\Tcal=(T,\beta)$ of $G$ such that the torso $G_t$ of $\Tcal$ at each node $t\in V(T)$ has an almost embedding $\delta_t$ in the projective plane of breadth at most $b_{\ref{thm_projective_global}}(k)$ and width at most $w_{\ref{thm_projective_global}}(k)$.
For $t\in V(T)$, let $\Ccal_t$ be the set of all vortex cells of $\delta_t$, $\Pcal_t=(\sigma_{\delta_t}(c))_{ c\in\Ccal_t}$, and $X_t=\bigcup\Pcal_t$.
Given that $G_t/\!\!/ \Pcal_t\in\Gcal_{\sf projective}$, it implies that $\Mcal(G_t,X_t)\cap \Gcal_{\sf projective}\ne\emptyset$.
The projective plane has Euler genus one.
Hence, by \autoref{prop_id}, $\bidim(G,X_t)=\Ocal((b_{\ref{thm_projective_global}}^4\cdot(b_{\ref{thm_projective_global}}\cdot w_{\ref{thm_projective_global}})^4)=2^{\Ocal(k\log k)}$.
We conclude that $\mathsf{idpr}^\star(G)\le \max_{t\in V(T)}\bidim(G,X_t)=2^{\Ocal(k\log k)}$.
\end{proof}

Similarly, we also conclude the following from \autoref{corr_plan_global} and \autoref{prop_id}.

\begin{theorem}\label{th_plan_global}
Let $k,c\in\Nbbb$. Let $G$ be a graph that excludes a long jump grid of order $k$ and a crosscap grid of order $c$ as a minor.
Then $\mathsf{idpl}^\star(G)=2^{\Ocal(k\log (k\cdot c))}$.
\end{theorem}

Combined with \autoref{tsw_asll}, \autoref{th_proj_global} immediately implies \autoref{upper_b}.
Additionally, we know from \cite{gavoille2023minoruniversal} that $\Gcal_{\sf projective}$ is the set of all minors of crosscap grids $\Ccal$.
This, combined with \autoref{tsw_asll} and \autoref{th_plan_global}, implies \autoref{upper_sb}.

\section{The lower bound} 
\label{power_b_ec}

The previous three sections where devoted to the proof that $\mathsf{idpr}^{\star}\preceq \p_{\mathscr{J}}$. To prove the equivalence of the two parameters  we now prove 
%
%
prove that  $\p_{\mathscr{J}} \preceq \mathsf{idpr}^{\star}$. Likewise, \autoref{lower_b} follows as all graphs 
in $\mathscr{J}$ are graphs in $\gedgeapex$.
%

For the proof of  $\p_{\mathscr{J}} \preceq \mathsf{idpr}^{\star}$
we first prove in \autoref{sub1} that $\idpr(\mathscr{J}_k)=Ω(k^{1/4})$ (\autoref{jopegftasi}). 
Then, in \autoref{sub2}, using the result of the previous subsection, we prove that $\idpr^\star(\mathscr{J}_k)=Ω(k^{1/384})$ (\autoref{lowbound}). 
To do so, we need a result from \cite{ThilikosW2025excluding} (\autoref{blackbox}) that essentially states that if the bidimensionality of some set  $X$ in  the torso of some bag of  a tree decomposition of $G$ is small, then so is the bidimensionality of $X$ in  $G$.
To see why this implies    that $\p_{\mathscr{J}} \preceq \idpr^{\star}$, we need the following.

\begin{lemma}\label{lem_id_min}
Both parameters $\mathsf{idpr}$ and $\mathsf{idpl}$ are minor-monotone. \end{lemma}

\begin{proof}
Let $G$ be a graph and let  $\mathsf{idpr}(G)≤k$ or $\mathsf{idpl}(G)≤k$.
Then there is $X\subseteq V(G)$ and $\Pcal=(X_1,\dots,X_r)\in\Pcal(X)$ be such that $\bidi(G,X)≤k$ and $G/\!\!/\Pcal\in\Gcal$ where $\Gcal\in\{\gprojective,\gplanar\}$, depending the variant of the parameter that we consider.
Let $H$ be a minor of $G$ and $\{S_x\mid x\in V(H)\}$ be a model of $H$ in~$G$.
For $i\in[r]$, let $Y_i=\{x\in V(H)\mid S_x\cap X_i\ne\emptyset\}$ and $Y=\bigcup_{i\in[r]}Y_i$.
Observe that  $\bidi(H,Y)\le \bidi(G,X)≤k$.
Note that the sets $Y_i$ may intersect, so $(Y_1,\dots,Y_r)$ is not a partition.
Let $\Pcal'=(Y_1',\dots,Y_s')\in\Pcal(Y)$ be such that for each $i\in[r]$, there is $j\in[s']$ such that $Y_i\subseteq Y_j'$ and, if $Y_j'\setminus Y_i\ne\emptyset$, then there is $i'\in[r]$ such that $Y_{i'}\subseteq Y_j'$ and $Y_i\cap Y_{i'}\ne\emptyset$.
Then, $\{S_x\mid x\in V(H)\setminus Y\}\cup\{\bigcup_{x\in Y_j'}S_x\setminus X\cup\{x_i\mid Y_i\subseteq Y_j'\}\mid j\in[s]\}$ is a model of $H/\!\!/\Pcal'$ in $G/\!\!/\Pcal$, where $x_i$ is the vertex of $G/\!\!/\Pcal$ obtained from the identification of~$X_i$.
Given that, in any case,  $\Gcal$ is minor-closed and that $G/\!\!/\Pcal\in\Gcal$, it implies that $\Mcal(H,Y)\ne\emptyset.$
Therefore, $\mathsf{idpr}(H)≤k$ or $\mathsf{idpl}(H)≤k$, as required
\end{proof}

It is easy to see that if $\p$ is a minor-monotone parameter, then the same holds also for its clique-sum extension $\p^{\star}$. This along with  \autoref{lem_id_min},  imply 
that   $\idpr^\star$ is also minor-monotone. This fact, together with the fact that  that $\idpr^\star(\mathscr{J}_k)=Ω(k^{1/384})$, which we prove in the next two subsections,  implies that $\p_{\mathscr{J}} \preceq \idpr^{\star}$ which is the main result of this section.
\subsection{Identifications in a long-jump grid}\label{sub1}

We first prove that $\p_{\mathscr{J}} \preceq \idpr$.

\begin{figure}[ht]
\begin{center}
\includegraphics[scale = 0.55]{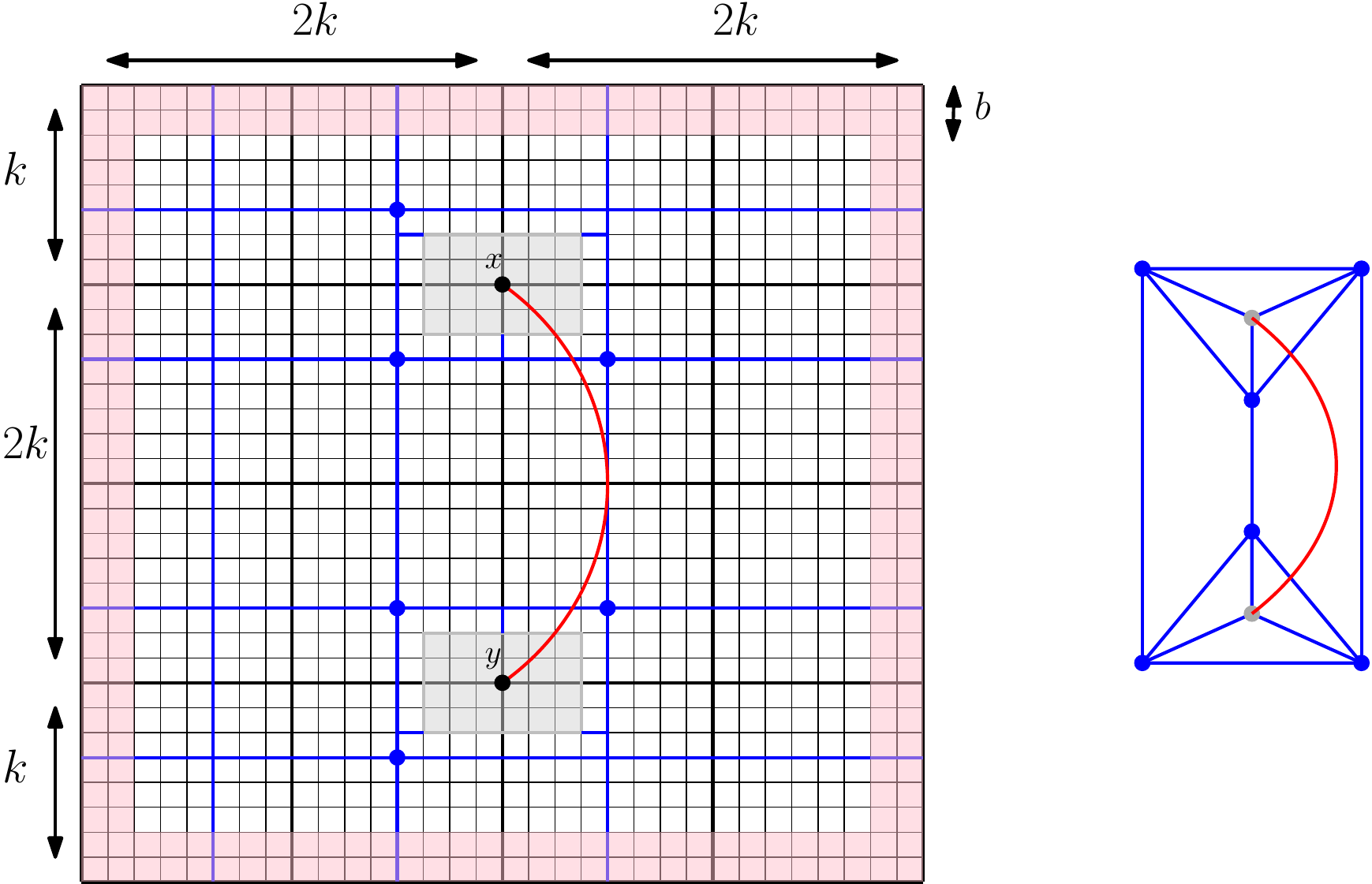}
\end{center}
  \caption{The graph $\mathscr{J}_8'$ and a  non-projective graph $F$.}
\label{fig_vgrtsrchyd}
\end{figure}

We define $\mathscr{J}' = \{\mathscr{J}_k' \mid k \in \mathbb{N}\}$ where $\mathscr{J}_k'$ is the graph obtained from a $(4k+1)\times (4k+1)$-grid $G_{4k+1}$ after adding the edge $\{x,y\}=\{(k+1,2k+1),(3k+1,2k+1)\}$ (see \autoref{fig_vgrtsrchyd} for an example where $k=8$).
It is easy to see that $\p_{\mathscr{J}}$ and $\p_{\mathscr{J}'}$ are linearly equivalent.

\medskip
Let $F$ be the graph depicted in \autoref{fig_vgrtsrchyd} (on the right).
$F$ is known to be a minor-obstruction of the projective plane (see $\Dcal_{17}$ in \cite{AsadiPT11mino,MoharT01Graphs}).

We need the following result.

\begin{proposition}[\!\!\cite{DemaineFHT04bidi}]\label{@subjection}
Let $G_m$ be the $(m\times m)$-grid and $S$ be a subset of the vertices in the central $(m-2\ell)\times(m-2\ell)$-subgrid of $G_m$, where $\ell=\lfloor\sqrt[4]{|S|}\rfloor$.
Then $G$ contains the $(\ell\times \ell)$-grid as an $S$-minor.
 \end{proposition}
 
We now have all ingredients for the main result of this part.
We essentially prove that if a set $X\subseteq V(\mathscr{J}_k')$ has too small biggest grid parameter, then every graph in $\Mcal(\mathscr{J}_k',X)$ contains $F$ as a minor.
And since $F$ is not projective, it thus implies that $\Mcal(\mathscr{J}_k',X)\cap\Gcal_{\sf projective}=\emptyset$.

\begin{lemma}\label{jopegftasi}
Let $k\in\Nbbb$.
Then $\idpr(\mathscr{J}_k)=Ω(k^{1/4})$.
\end{lemma}
\begin{proof}

Given that $\p_{\mathscr{J}}$ and $\p_{\mathscr{J}'}$ are linearly equivalent, it is enough to prove the result for $\p_{\mathscr{J}'}$.
Let $X\subseteq V(\mathscr{J}_k')$ and $\Pcal\in\Pcal(X)$ be such that $G'\coloneqq \mathscr{J}_k'/\!\!/\Pcal\in\Gcal_{\sf projective}$. Let us prove that $\bidim(\mathscr{J}_k',X)=Ω(k^{1/4})$.

For each margin $b\in[0,\lfloor k/2\rfloor]$, let $X_{b}$ be the set of vertices of $X$ that belong in the $(k-b) \times (k-b)$-subgrid $G_{k-b}$ of $G_{4k+1}$ and $\Pcal_b=\Pcal\cap X_b$.
We claim that $|X_b|\geq k-b$.
Suppose towards a contradiction that $|X_b|\leq k-b-1$.
Then, in $G_{k-b}$, there are at least two vertical paths on the left of $x$ and $y$, one vertical path on the right of $x$ and $y$, one horizontal path above $x$, two horizontal paths between $x$ and $y$, and one horizontal path below $y$, whose vertices do not intersect $X$ (drawn in blue in \autoref{fig_vgrtsrchyd}).
Let $V_x,V_y\subseteq V(G_{k-b})$ be the sets of all the vertices in the two squares containing $x$ and $y$ surrounded by the aforementioned paths (drawn in grey in \autoref{fig_vgrtsrchyd}), respectively.
Let $\Pcal'$ be obtained from $\Pcal$ by adding two new parts: $V_x$ and $V_y$.
Given that $G[V_x]$ (resp. $G[V_y]$) is connected, the identification of all vertices in $V_x$ (resp. $V_y$) is equivalent to the contraction of all edges in $G[V_x]$ (resp. $G[V_y]$).
Given that $\mathscr{J}_k'/\!\!/\Pcal\in\Gcal_{\sf projective}$ and that $\Gcal_{\sf projective}$ is closed under contractions, we thus have $G'/\!\!/V_x/\!\!/V_y=G/\!\!/\Pcal'\in\Gcal_{\sf projective}$.
However, as depicted in \autoref{fig_vgrtsrchyd}, $F$ is a minor of $G/\!\!/\Pcal'$, a contradiction.
We conclude that $|X_{b}|\geq  k-b$.

We choose the margin $b_k$ as the maximum integer $b$ such that    $\lfloor\sqrt[4]{{k-b}}\rfloor\geq 2b$. This implies that 
 $ \lfloor \sqrt[4]{|X_{b_k}|} \rfloor\geq 2b_k$, therefore, from \autoref{@subjection}, $\bidim(\mathscr{J}_k',X)\geq \lfloor\sqrt[4]{|X_{b_k}|} \rfloor=2b_{k}$. As $b_{k}=Ω(k^{1/4})$, we are done.
\end{proof}

\subsection{Lower bound under the presence of clique-sums}\label{sub2}

In this section we extend the polynomial bound  in \cref{jopegftasi} from $\idpr$ to $\idpr^\star$.
For this we need a series of definitions and preliminary results from \cite{ThilikosW2025excluding}.

We say that a graph G is \emph{$(f,k)$-tightly connected} for some $f:\Nbbb\to\Nbbb$ and $k\in\Nbbb$, if
for every separation $(B_1,B_2)$ of $G$ of order $q<k$ such that both $G[B_{1}\setminus B_{2}]$ and $G[B_{2}\setminus B_{1}]$ are connected, it holds that one of $B_1,$ $B_2$ has at most $f(q)$ vertices.
We say that a parametric graph $\mathscr{H}=\langle \mathscr{H}_k \rangle_{k\in \Nbbb}$ is \emph{$f$-tightly connected} if for each $k\in\Nbbb$, $\mathscr{H}_k$ is $(f,k)$-tightly connected.

\begin{proposition}[Lemma 4.9, \cite{ThilikosW2025excluding}]
\label{letas_s}
Let $r\in\Nbbb$ and let $g,h:\Nbbb\to\Nbbb$ be two non-decreasing functions.
Let $\p$ be a minor-monotone graph parameter such that for every graph $H$, $\hw(H)≤h(\p(H))$.
Let $G$ be a $(g,h(r)+1)$-tightly connected graph with $|V(G)|>2g(h(r))$ and $\p^\star(G)\le r$.
Then 
$G$ admits a tree decomposition $({T},{\beta})$ where 
${T}$ is a star with center $t$, where for the torso $G_{t}$ of $(T, \beta)$ at $t$, $\p(G_{t})≤r$, 
 and where, for every  $e=tt'\in E({T})$, $G[{\beta }(t')\setminus {\beta }(t)]$ is a connected graph and $|{\beta }(t')|≤g(h(r))$.
\end{proposition}

\begin{proposition}[Lemma 4.11, \cite{ThilikosW2025excluding}]\label{blackbox}
Let $G$ be a 4-connected  graph of maximum degree $\delta $.
Let $r\in\Nbbb$ and let $(T,\beta)$ be a tree decomposition of $G$ where $T$ is a star
with center $t$ and such that for every  $e=tt'\in E(T)$, $G[\beta (t')\setminus \beta(t)]$ is a connected graph on at most ${l}$ vertices. Let $G_{t}$ be the torso of $(T,\beta)$ at $t$ and 
we denote  $G^c=(V(G),E(G)\cup E(G_{t}))$, $m=\hw(G^c)$, 
and $B=\bigcup_{t'\in V(T)\setminus\{t\}}\beta (t')$.
There is a function $f_{\ref{blackbox}}:\Nbbb^4\to\Nbbb$ such that
if  $X$ is a subset of $V(G_t)$ where $\bidim(G_t,X)≤  r$, then  $\bidim(G^c,X\cup B)≤ f_{\ref{blackbox}}(r,m,l,\Delta)$.
Moreover, $f_{\ref{blackbox}}(r,m,l,\Delta)\in \mathcal{O}(m^{48}+r^{1/2}\cdot ({m^8l^4}+\delta m^{24}{l^2}))\subseteq \poly(m+l+r+\delta )$.
\end{proposition}

\begin{lemma}
\label{@neutrality}
For every graph $G$ it holds that  $\hw(G)\leq (\mathsf{idpr}(G))^2+5$.
\end{lemma}

\begin{proof}
Let $t=\hw(G)$ and let $\mathsf{idpr}(G)≤k$.
Recall also that $K_{7}$ cannot be embedded in the projective plane while $K_{6}$ can \cite{MoharT01Graphs}. This implies that if $G$ is projective, then $
\hw(G)≤6$.
From \autoref{lem_id_min}, 
 given that $K_t$ is a minor of $G$, there exists $Y\subseteq V(K_t)$ such that $\bidim(K_t,Y)\le k$ and $\Mcal(K_t,Y)\cap\gprojective\ne\emptyset$.
Let $G'\in \Mcal(K_t,Y)\cap\gprojective$. As $G'$ is a projective graph,  $\hw(G')\le 6$.
Therefore, the identification of a partition of $Y$ reduces the size of the clique  by at least $t-6$.
Hence, given that identifying $p+1$ vertices reduce that size by at most $p$, we conclude that $|Y|\ge t-6+1$.
Notice that $k\ge\bidim(K_{t},Y)\geq  \lfloor |Y|^{1/2}\rfloor \geq  (t-5)^{1/2}.$
This implies that $t\leq k^2+5$.
\end{proof}

\begin{figure}[h!]
\begin{center}
\includegraphics[scale = 1.1]{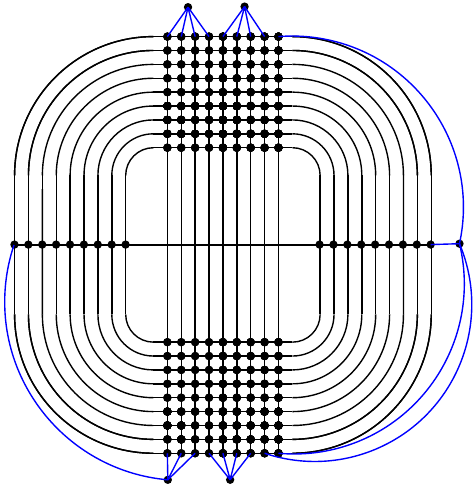}
\end{center}
  \caption{The graph $\mathscr{I}_4$.}
\label{fig_regular}
\end{figure}

\paragraph{A new variant of the long-jump grid.}

For the purposes of the proofs in this section, we consider an ``enhanced version'' of $\mathscr{J}_k,$ namely $\mathscr{I} = \langle \mathscr{I}_k \rangle_{k \in \Nbbb}$ which is obtained from $\mathscr{J}_{2k+1}$ as follows.
We partition the $4k+4$ vertices in the outermost cycle of $\mathscr{J}_{2k+1}$ into $k+1$ sets of four consecutive vertices.
For each such set, we add a new vertex adjacent to the four vertices of the set, as illustrated in \autoref{fig_regular}.
These $k+1$ new vertices are called \emph{satellite} vertices.
Note that $\mathscr{I}_k$ is 4-connected and 4-regular.
Moreover, we easily have that $\p_{\mathscr{J}}$ and $ \p_{\mathscr{I}}$ are linearly equivalent.

\begin{lemma}
\label{@vegetables}
 $\mathscr{J}$ is $f_{\ref{@vegetables}}$-tightly connected for $f_{\ref{@vegetables}}(q)=(2q+1)^2$.
\end{lemma}

\begin{proof}
Let $G\coloneqq \mathscr{I}_{k}.$ 
Let $(B_1,B_2)$ be a separation of $G$ of order $q<k.$
Let  $A=B_{1}\cap B_{2},$ $D_{1}=B_{1}
\setminus A,$ and $D_{2}=B_{2}\setminus A.$ We assume that both $G[D_{1}]$ and  $G[D_{2}]$ are connected.
Let us show that one of $B_1$ and $B_2$ has at most $(2q+1)^2$ vertices.

Let $k'=2k+1$.
Let $Q$ be the spanning $((2k'+2)\times  k')$-annulus grid of $G.$
Clearly $Q$ contains $k'$ cycles. Also it contains  $2k'+2$ paths, each  on $k'$ vertices, which we call {\emph{tracks}}.
We also use $R$ for the set of the satellite vertices of $\mathscr{I}_{k}$.
We  enhance the tracks by extending each of them to the unique satellite vertex that is adjacent to one of its {endpoints}.
Let $Y$ be the union of all the $\geq k'-q$ cycles that are \textsl{not} met by $A$ and 
of all $\geq2k'+2-4q$ tracks that are \textsl{not} met by $A$ (observe that  if a vertex of $A$ belongs to $R$, then it meets four tracks).

As every track in $Y$ has a common endpoint to every cycle in $Y,$  we obtain that $Y$ is connected, therefore $V(Y)$ is either a subset of $D_{1}$ or a subset of $D_{2}.$
W.l.o.g., we assume that $V(Y)\subseteq D_{1}.$ 
We next prove that $|B_{2}|\leq (2q+1)^2.$

Let $x$ be a vertex of $D_{2}$
and let $P_x$ be some path of $G[B_{2}]$ starting from $x$ and finishing 
to some vertex of $A$ and with all internal vertices in $D_{2}.$
This path cannot meet more than $q=|A|$ cycles of  $Q$ because  each such cycle contains some vertex of $Y\subseteq D_{1}.$ Similarly, $P_{x}$ cannot meet more than $q=|A|$  tracks as each such track contains some vertex of  $Y\subseteq D_{1}.$
This implies that every vertex of $G[B_{2}]$ should be within distance 
at most $q$ from $x.$
It is now easy to verify that, in $G,$ the vertices within distance at most $q$
from some $x\in A$ is upper bounded by $(2q+1)^2.$
As all vertices of $B_{2}$ are accessible from $x$ within this distance in $G[B_{2}],$ we conclude that   $|B_{2}|\leq (2q+1)^2.$
\end{proof}

\begin{lemma}\label{lowbound}
Let $k\in\Nbbb$. Then $\idpr^\star(\mathscr{J}_k)=Ω(k^{1/384})$.
\end{lemma}

\begin{proof}
We prefer to work with $\mathscr{I}_k$ (cf. \autoref{fig_regular}) because it is 4-connected and 4-regular.
By \cref{jopegftasi}, we obtain that $\idpr(\mathscr{J}_k)=Ω(k^{1/4})$. 
Thus, given that $\mathscr{J}_{k}$ and $\mathscr{I}_{k}$ are linearly equivalent, it is enough to prove that $$\idpr(\mathscr{I}_{k})\in\mathcal{O}\big(\idpr^{\star}(\mathscr{I}_{k})^{96}\big).$$

For simplicity, we set $G=\mathscr{I}_k$.
Let $r$ be such that $\idpr^\star(G)\leq r$.
Our objective is to {prove that 
$\idpr(G)\in \mathcal{O}(r^{96})$}.
%
Let $h,g:\Nbbb\to\Nbbb$ be such that
\begin{eqnarray}
h(x) & = &  x^2+5,  \mbox{~and} \label{@naturphilosophie}\\
g(x) & = & (2x+1)^2. \label{@eco}
\end{eqnarray}
Thus, by applying \autoref{@neutrality}, for any graph $H$, we have
\begin{eqnarray}
\hw(H) & \leq &h(\idpr(H)).\label{@eq}
\end{eqnarray}
Given that $|V(G)|=8k^2+13k+5$, and that $\idpr(G)\leq|V(G)|$, we may assume that
\begin{eqnarray}
|V(G)| & > & 2g(h(r)),  \mbox{~and} \label{@assurances}\\
h(r)+1 & \leq & k \label{@introjected}
\end{eqnarray}
Indeed, otherwise, in the first case, $\idpr(G)\leq2g(h(r))\in \mathcal{O}(r^4)\subseteq \mathcal{O}(r^{96})$, and in the second case,
$\idpr(G)\leq8k^2+13k+5\leq8h(r)^2+13h(r)+5\in \mathcal{O}(r^4)\subseteq \mathcal{O}(r^{96})$. In both cases, we are done.
By \autoref{@vegetables}, $G$ is $(g,k)$-tightly connected, and thus, by \eqref{@introjected}, $G$ is $(g,h(r)+1)$-tightly connected.
Thus, we have all ingredients to apply \autoref{letas_s}.
We obtain a tree decomposition $(T,\beta)$ of $G$ where $T$ is a star with center $t$, where $\idpr(G_t)\leq r$ and where for every $tt'\in E(T')$, $G[\beta(t')\setminus\beta(t)]$ is a connected graph and $|\beta(t')|\leq g(h(r))$.

Since $\idpr(G_t)\leq r$, there is a set $X_t\subseteq \beta(t)$ and $\Pcal_t=(X_1,\dots,X_p)\in\Pcal(X_t)$ such that $G_t/\!\!/\Pcal_t\in\Gcal_{\sf projective}$ and $\bidim(G_t,X_t)\leq r$.
Let $G^c$ denote the graph $(V(G),E(G)\cup E(G_{t}))$ and $B=\bigcup_{t'\in V(T)\setminus\{t\}}\beta (t')$.
By \eqref{@eq}, $\hw(G^c)\leq h(r)$.
Given that $G$ is 4-connected and has maximum degree four, we can hence apply \autoref{blackbox} with $\Delta=4, l=g(h(r))$ and $m=h(r)$ on the tree decomposition $(T,\beta)$.
We conclude that 
\begin{eqnarray}
\bidim(G^c,X_t\cup B)&\leq & f_{\ref{blackbox}}(r,\eta,g(h(r)),4)\in\mathcal{O}(r^{96}).
\label{@ubi}
\end{eqnarray}

Clearly $G$ is a (spanning) subgraph of $G^{\mathsf{c}}.$
For each $t'\in N'_{T}(t)$, let $A_{t'}$ be the adhesion of $t'$ and $t$, let $C_{t'}=\beta(t')\setminus A_{t'}$,
and let $Y_{t'}$ be a set composed of $C_{t'}$ and an arbitrary vertex $a_{t'}$ of $A_{t'}$.
Remember that $G^c[A_{t'}]$ is a clique.
Hence, the identification of $C_{t'}$ with $a_{t'}$ does not create additional edges, and thus has the same effect as removing $C_{t'}$.
Therefore, $G_t=G^c/\!\!/\{Y_{t'}\mid t'\in N_T(t)\}$, and thus $G^c/\!\!/\{Y_{t'}\mid t'\in N_T(t)\}/\!\!/\Pcal_t\in\Gcal_{\sf projective}$.
Thus $\idpr(G^c)\leq\bidim(G^{\mathsf{c}},X_t\cup B_t)\in^{\eqref{@ubi}} \Ocal(r^{96})$.
Given that $\idpr(G)\leq \idpr(G^c)$, the statement follows.
\end{proof}

%

If we want to go to the plane instead of the projective plane, the following lemma is an easy corollary of \autoref{lowbound} and \cite[Lemma 7.13]{ThilikosW2025excluding}.
Remember that $\mathscr{C}_k$ is the crosscap grid.

\begin{lemma}
\label{i9opb1}
Let $k\in\Nbbb$.
Then $\mathsf{idpl}^\star(\mathscr{J}_k)=Ω(k^{1/384})$ and $\mathsf{idpl}^\star(\mathscr{C}_k)=Ω(k^{1/480})$.
\end{lemma}

\begin{proof}
By \autoref{lowbound}, $\idpr^\star(\mathscr{J}_k)=Ω(k^{1/384})$.
Given that $\Gcal_{\sf planar}\subseteq \Gcal_{\sf projective}$, we thus have that, for every graph $G$, $\mathsf{idpl}(G)\ge\idpr(G)$, and therefore $\mathsf{idpl}^\star(\mathscr{J}_k)=Ω(k^{1/384})$.

We define the parameter $\p$ such that  $$\p(G)=\min\{k\mid\text{ there exists } X\subseteq V(G)\text{ s.t. } \bidim(G,X)\le k \text{ and } G- X\in\Gcal_{\sf planar}\}.$$
By \cite[Lemma 7.13]{ThilikosW2025excluding}, 
$\p^\star(\mathscr{C}_k)=Ω(k^{1/480})$.
Given that $\Mcal(G,X)\in\Gcal_{\sf planar}$ implies $G-X\in\Gcal_{\sf planar}$, we thus have that, for every graph $G$, 
$\mathsf{idpl}(G)\ge\p(G)$, and therefore $\mathsf{idpl}^\star(\mathscr{C}_k)=Ω(k^{1/480})$.
\end{proof}

\section{Conclusion and open problems}
%

In this paper, we prove a decomposition-based 
min-max theorem for the structure of graphs excluding edge-apex graphs  or, alternatively, graphs embedded in the pinched sphere as minors.
We prove that such graphs can be tree-decomposed so that each torso 
contains a set of vertices of bounded bidimensionality whose identification into few vertices results in a projective graph.
Moreover, this decomposition is optimal in the sense that apex-graphs do not admit such a decomposition.
If the excluded graph is embeddable in a surface, analogous decomposition
theorems are proven in \cite{ThilikosW2025excluding}.
In that sense, this paper can be seen as the first step towards extending the results of \cite{ThilikosW2025excluding} to pseudosurfaces.
This seems particularly interesting for those that define graphs classes that are minor-closed (see \cite{Knor96charac}).
Thinking further in this direction, we believe that the decomposition theorems for pinched surfaces of higher genus cannot be expected to avoid the presence of apices.
This is due to the special nature of the sphere and the projective plane that allowed for the arguments used in \cref{subsec:localtoglobal}.

Notice that the operations we apply to the torsos of our decomposition are not removals of low bidimensionality vertex sets, as it is the case of the interpretation of the GMST in \cite{ThilikosW2025excluding}, but vertex identifications within such sets.
Interestingly, vertex identification, as a modification operation become a topic of research only recently 
(see \cite{MorelleST2024vertex,MorelleST25grap} for some first results).
Our structure theorem says that, in the absence of a long jump, the vortices that appear as part of the almost embeddings can be identified into a small number of vertices, turning the almost embedding of the torso at any node of our tree decomposition into a proper embedding.
More generally, we may combine vertex removals and vertex identifications in order to describe the modification operations for the version of GMST proposed in \cite{ThilikosW2025excluding}:
when it comes to apices we remove vertex sets of small size and when it comes to vortices we identify them as vertex sets of low bidimensionality in order to obtain  some surface embedding of the resulting graph. 
The removal of few additional vertices is unavoidable because, similar to the result in \cite{DvorakT14list}, apices generally serve two purposes:
Firstly, they are responsible for guaranteeing a flat structure in the embedded part, and secondly, they \textsl{separate} the vortices from each other.
This second purpose becomes unnecessary when excluding a long jump.
However, the long-jump grid itself cannot be turned into a planar graph by identifying a vertex set of small bidimensionality.
Hence, one would still need to allow for an additional apex set, or one would need to allow for an, of course bounded, increase of the Euler-genus of the surface.

Another direction of research is whether one may use the decomposition of \cref{upper_b}  (or \cref{upper_sb})
for algorithmic purposes. Are there interesting problems for which the presence of  $\mathscr{J}_k$ (or of $\mathscr{J}_k$ and $\mathscr{C}_k$) as a minor in the input graph directly certifies an answer?

A last question is whether the functional dependencies of our results (i.e., the function $f$ in  \cref{upper_b}  or \cref{upper_sb}) can be reduced to polynomial ones.
The techniques used in this paper originate from the proof of the GMST by Kawarabayashi, Thomas, and Wollan \cite{kawarabayashi2020quickly}.
In particular, \cref{main2} and \cref{main1} are heavily based on the approach of Kawarabayashi, Thomas, and Wollan.
The core theorems from \cite{kawarabayashi2020quickly} as well as \cref{main2} and \cref{main1} inherently force exponential dependencies.
One source of these dependencies can be found in the one generates new nests for the two freshly found young vortices obtained from splitting an older vortex.
These theorems require the sacrifice of roughly half of the previous nest every time a vortex is split.
Recently, Gorsky, Seweryn, and Wiederrecht \cite{gorsky2025polynomial} found a way to avoid these sources for exponential blowup when splitting vortices, culminating in polynomial bounds for the GMST.
We believe that it should be possible to adapt their methods to our setting.
This would allow us to reprove variants of \cref{main2} and \cref{main1} that do not force exponential blowup, allowing for overall polynomial bounds for the functions in our main theorem.

%

\end{document}